\documentclass[12pt]{article}
\usepackage{note}

\renewcommand{\thetab}{\gammab}
\renewcommand{\thetabs}{\gammabs}
\renewcommand{\thetabh}{\gammabh}
\def \sigmae   {\sigma_{\epsilon}}
\def \sigmaeh  {\hat{\sigma}_{\epsilon}}
\def \betanom  {\betab_{-m}}
\def \betahnom {\hat{\betab}_{-m}}
\def \betasnom {{\betab}^{\ast}_{-m}}

\def \betahat  {\hat{\betab}_{a,T}}
\def \betahas  {\hat{\betab}_{a,S}}

\def \betasm   {{\betab}^{\ast}_{m}}
\def \betaha   {\hat{\betab}_a}
\def \ts       {T\cup S_{a}}
\def \betahats {\hat{\betab}_{a,T\cup S_{a}}}
\def \betasts  {\betabs_{T\cup S_{a}}}
\def \thetam   {\gammab_{m}}
\def \thetasm  {\gammabs_{m}}

\def \gammam   {\gammab_{m}}
\def \gammasm  {\gammabs_{m}}

\def \gammaha  {\hat{\gammab}_a}

\def \fh       {\hat{\f}}
\def \Fh       {\hat{\F}}
\def \Uh       {\hat{\U}}
\def \Uhts     {\hat{\U}_{T\cup S_{a}}}
\def \Uts      {{\U}_{T\cup S_{a}}}
\def \Uht      {\hat{\U}_{T}}
\def \Uhs      {\hat{\U}_{S_a}}
\def \Ut       {{\U}_{T}}
\def \Us       {{\U}_{S_a}}
\def \wh       {\hat{\w}}
\def \ws       {{\w}^{\ast}}
\def \ss       {s^{\ast}}
\def \Sha      {\hat{S}_a}

\def \N        {\mathcal{N}}
\def \RR       {\bfm{R}}
\def \pr       {\mathrm{Pr}}

\def \sqlogqn  {\sqrt{\frac{\log p_{-m}}{n}}}
\def \Itrue    {\I^{\ast}_{\gammab_m|\betab_{-m}}}
\def \Itruehalf{\I^{\ast^{-1/2}}_{\gammab_m|\betab_{-m}}}
\def \Ihat     {\hat{\I}_{\gammab_m|\betab_{-m}}}
\def \Itilde   {\tilde{\I}_{\gammab_m|\betab_{-m}}}
\def \Itilde   {\tilde{\I}_{\thetab_m|\betab_{-m}}}
\def \keyrate  {\sqrt{(\log p)/n}\{1\vee (n^{1/4}/\sqrt{p_{\min}}) \}}
\def \gammat   {\tilde{\gammab}}
\def \Sigmaunom{\Sigmab_{u_{-m}}}

\def \Sigmaxnom{\Sigmab_{x_{-m}}}
\def \varthetab{\boldsymbol{\vartheta}}
\DeclareMathOperator*{\supp}{supp}
\newcommand{\Lsupnorm}[1]{\lVert#1\rVert_{L_{\infty}}}
\newcommand{\ssupnorm}[1]{\Big\lVert#1\Big\rVert_{{\infty}}}
\newcommand{\slonenorm}[1]{\Big\lVert#1\Big\rVert_{{1}}}
\def \Itildehalf   {\tilde{\I}^{-1/2}_{\gammab_m|\betab_{-m}}}
\def \betabt   {\tilde{\betab}}

\newcounter{CondCounter}
\newenvironment{con}{
  \par\noindent\refstepcounter{CondCounter} 
  \textbf{Condition \theCondCounter.}}{} 
\def\spacingset#1{\renewcommand{\baselinestretch}%
  {#1}\small\normalsize} \spacingset{1}

\newcounter{ExaCounter}
\newenvironment{exa}{
  \par\noindent\refstepcounter{ExaCounter} 
  \textbf{Example \theExaCounter.}}{} 
\def\spacingset#1{\renewcommand{\baselinestretch}%
  {#1}\small\normalsize} \spacingset{1}

\begin{document}

\title{\Large{\textbf{Integrative Factor Regression and Its Inference}}\\
\Large{\textbf{for Multimodal Data Analysis}}}
\author{
\bigskip
Quefeng Li and Lexin Li \\
\normalsize{University of North Carolina, Chapel Hill and University of California, Berkeley}
}
\date{}
\maketitle

\begin{abstract}
Multimodal data, where different types of data are collected from the same subjects, are fast emerging in a large variety of scientific applications. Factor analysis is commonly used in integrative analysis of multimodal data, and is particularly useful to overcome the curse of high dimensionality and high correlations. However, there is little work on statistical inference for factor analysis based supervised modeling of multimodal data. In this article, we consider an integrative linear regression model that is built upon the latent factors extracted from multimodal data. We address three important questions: how to infer the significance of one data modality given the other modalities in the model; how to infer the significance of a combination of variables from one modality or across different  modalities; and how to quantify the contribution, measured by the goodness-of-fit, of one data modality given the others. When answering each question, we explicitly characterize both the benefit and the extra cost of factor analysis. Those questions, to our knowledge, have not yet been addressed despite wide use of factor analysis in integrative multimodal analysis, and our proposal bridges an important gap. We study the empirical performance of our methods through simulations, and further illustrate with a multimodal neuroimaging analysis.
\end{abstract}
\bigskip

\noindent
{\bf Keywords:} Data integration; Dimension reduction; Factor analysis; High-dimensional inference; Multimodal neuroimaging; Principal components analysis.
\vfill

\newpage
\spacingset{1.45} 

\section{Introduction}
\label{sec:introduction}

Thanks to rapid technological advances, multiple types of data are now frequently collected for a common set of experimental subjects. Such a new data structure, often referred as multi-view, multi-source or multimodal data, is fast emerging in a wide range of scientific fields. Examples include multi-omics data in genomics, multimodal neuroimaging data in neuroscience, multimodal electronic health records data in health care administration, among others. Numerous empirical studies have found that, by combining diverse but usually complementary information from different types of data, an integrative analysis of multimodal data is often beneficial; see \citet{Uludag2014, LiYF2016, richardson2016statistical} for reviews and the references therein. 

In view of the promise of multimodal data, a number of statistical methods have recently been developed for integrative analysis. An important class of such solutions is matrix or tensor factorization, which decomposes multimodal data into the components that capture joint variations shared across modalities, and the components that characterize modality-specific variations \citep{lock2013joint, yang2015non, li2017incorporating, lock2018supervised, gaynanova2019structural}. Another class is canonical correlation analysis, which seeks maximum correlations between different data modalities through decomposition of the between-modality dependency structure \citep{li2018general, ShuZhu2019}. However, all these methods are unsupervised, in the sense that there is not a response variable involved. \citet{LiChen2018} recently proposed an integrative reduced-rank regression to model multivariate responses given multi-view data as predictors. \citet{XueQu2019} developed an estimating equations approach to accommodate block missing patterns in multimodal data. Their methods are supervised, but both focused on parameter estimation and variable selection instead of statistical inference. 

It is of ubiquitous interest to study the predictive associations between responses and multimodal predictors. However, there are some unique characteristics of multimodal data that make the problem challenging. First, multimodal data are often high-dimensional. In plenty of applications, even a single modality contains more variables than the sample size. Second, some variables in multimodal data can be highly correlated. This can happen for the variables within a single modality, or for the related variables across multiple modalities as they often measure related features of the same subject. This phenomenon has been constantly observed, and is actually the base upon which those matrix or tensor factorization solutions are built \citep{lock2013joint}. Such high correlations pose challenges when directly applying many standard high-dimensional methods, such as LASSO \citep{tibshirani_regression_1996}, to multimodal data, as they usually require the predictors not to be highly correlated in order to achieve some desired statistical properties, e.g., the variable selection consistency. Finally, multimodal data pose new questions; for instance, how to quantify the contribution and statistical significance of one data modality conditioning on the other data modalities in the regression model.

Factor analysis is a well-known approach to both reduce high dimensionality and high correlations among the variables. For data with a single modality, \citet{Fan2013b} employed a factor model to estimate a non-sparse covariance matrix. \citet{kneip2011factor} proposed to include the latent factors as additional explanatory variables in a high-dimensional linear regression, and established the model selection consistency. \citet{fan2016factor} proposed a factor-adjusted model selection method for a general high-dimensional $M$-estimation problem. They separated the latent factors from the idiosyncratic components to reduce correlations among the covariates, and showed that their method can reach the variable selection consistency under milder conditions than standard selection methods. \citet{li2018embracing} studied estimation of a covariance matrix of variables. They showed that leveraging on additional auxiliary variables can improve the estimation, when the auxiliary variables share some common latent factors with the variables of interest. For data with multiple modalities, \citet{shen2013sparse} proposed an integrative clustering method based on identifying common latent factors from multi-omics data. \citet{ZhangQu2019} developed an imputed factor regression model for dimension reduction and prediction of multimodal data with missing blocks. Despite these efforts, however, there is little work on statistical inference for supervised modeling of multimodal data. Moreover, there is no explicit quantification of the benefit of factor analysis in a multimodal regression setting, and many important inference-related questions remain unanswered.

In this article, we aim to bridge this gap. We consider an integrative linear regression model built upon the latent factors extracted from multimodal data. We show that this model alleviates high dimensionality and high correlations of multimodal data. Based on this model, we address three important questions: how to infer the significance of one data modality given the other modalities in the model; how to infer the significance of a combination of variables from one or more modalities; and how to quantify the contribution, measured by the goodness-of-fit, of one data modality given the others. When answering each question, we explicitly characterize both the benefit and the extra cost of factor analysis. First, by resorting to a relatively small number of latent factors, it effectively reduces the dimensionality and turns a high-dimensional test to a low-dimensional one when testing the significance of a whole modality. As a result, it enables us to derive a closed form for the limiting distribution of the test statistic; see Theorem \ref{thm:1}. Second, our method can consistently estimate the support and nonzero components of the covariate coefficients in the regression model. More importantly, by using the decorrelated idiosyncratic components from factor analysis as the pseudo predictors, instead of the original highly correlated covariates, it requires much weaker conditions to reach the variable selection and estimation consistency; see Theorem \ref{thm:2}. In addition, we show that, when there are enough variables in each modality so that the latent factors can be well estimated, the resulting estimation error can reach the minimax optimal rate, and under weaker conditions. Third, such an improvement in selection and estimation in turn benefits the inference of the significance of a linear combination of predictors, by requiring less stringent conditions to establish the limiting distribution of the test statistic; see Theorem \ref{thm:3}. Finally, by leveraging on the latent factors shared across modalities, it enables us to obtain a closed-form measure of the variance of the response explained by one modality in addition to the others. Such a measure facilitates the quantification of the contribution of an individual modality.

Our proposal contributes on several fronts. Even though factor analysis has been widely used in multimodal data analysis, there has been no formal test developed to explicitly quantify the contribution and significance of an individual modality or a related set of variables across different modalities. Our proposal provides the first inferential tools to address those important questions. Moreover, our work is built on careful examination of the benefit and trade-off of factor analysis in regression. Compared to the existing literature, the proof techniques are much more involved than those of the standard setting when the design matrix is observed and fixed with a single data modality. The technical tools we develop here are not limited to our setting alone, but are applicable to general supervised high-dimensional factor models. We also remark that, although we focus on a linear factor regression model, most of the inference-related results we obtain can be extended to more general $M$-estimation problems such as a generalized linear model.

We employ the following notation throughout this article. For a vector $\a\in\Rcal^d$, let $\supnorm{\a}=\max_j|a_j|$, $\lonenorm{\a}=\sum_{j=1}^d |a_j|$, $\ltwonorm{\a}=(\sum_{j=1}^d a_j^2)^{1/2}$ denote its sup-norm, $L_1$-norm and Euclidean norm, respectively. For an index set $S$, let $\a_S$ denote the subvector of $\a$ with indices in $S$. In particular, let the subscript $m$ denote the index set of the $m$th modality, and the subscript $-m$ denote the index set of all other modalities. Let $\a^{\otimes 2}=\a\a'$ denote its outer product. Let $\supp(\a)=\{j:a_j\neq 0 \}$ denote the support of $\a$. For a square matrix $\A=(a_{ij})\in \Rcal^{d\times d}$, let $\lambda_{\min}(\A)$ and $\lambda_{\max}(\A)$ denote its minimum and maximum eigenvalues. Let $\supnorm{\A}=\sup_{ij} |a_{ij}|$, $\lonenorm{\A}= \max_{1\leq j\leq d} \sum_{i=1}^d |a_{ij}|$, $\ltwonorm{\A}=\lambda_{\max}(\A)$, $\fnorm{\A}= (\sum_{i,j} a_{ij}^2)^{1/2}$ denote its element-wise sup-norm, $L_1$-norm, $L_2$-norm, and Frobenious norm, respectively. Let $\A_S$ denote the submatrix of $\A$ with row and column indices in $S$. For a rectangular matrix $\B=(b_{ij})\in \Rcal^{m\times n}$, let $\Lsupnorm{\B}=\max_{1\leq i\leq m} \sum_{j=1}^n |b_{ij}|$. For two sequences $a_n$ and $b_n$, write $a_n=o(b_n)$ if $a_n/b_n \to 0$, and $a_n\gg b_n$ if $b_n/a_n \to 0$. For an integer $M$, let $[M] = \{1,\ldots,M \}$. For a set $S$, let $|S|$ denote the number of elements in $S$.

The rest of the article is organized as follows. We introduce the integrative factor regression model in Section \ref{sec:model}, and describe the parameter estimation in Section \ref{sec:estimation}. These results are mostly built upon the existing literature on factor analysis. Then we address the three questions, which to our knowledge have not been answered before. That is, we develop a test to evaluate the significance of an individual modality given the other modalities in Section \ref{sec:test-whole}, develop a test for a linear combination of predictors in Section \ref{sec:test-lincomb}, and derive a measure to quantify the contribution of an individual modality in Section \ref{sec:var-measure}. We present the simulations and a multimodal neuroimaging data example in Section \ref{sec:numerical}. We relegate all technical proofs and some additional lemmas to the Supplementary Materials.

\section{Integrative factor regression model}
\label{sec:model}

Suppose there are $M$ modalities of variables. Let $\x_m \in \Rcal^{p_m}$ denote the vector of $p_m$ random variables from the $m$th modality, and $y$ denote the response variable. Let $\x=(\x'_1,\ldots,\x'_M)'\in \Rcal^p$, and $p = \sum_{m=1}^M p_m$. We assume $\x_m$ is driven by some latent factors in that $\x_m$ can be decomposed as 
\begin{equation} \label{eq:1}
\x_m = \Lambdab_m \f_m + \u_m,
\end{equation}
where $\f_m\in \Rcal^{K_m}$ is the vector of $K_m$ random latent factors, $\u_m \in \Rcal^{p_m}$ is the vector of random idiosyncratic errors of variables in the $m$th modality that are uncorrelated with $\f_m$, and $\Lambdab_m\in \Rcal^{p_m\times K_m}$ is the loading matrix of $\x_m$ on the latent factors $\f_m$. To avoid the identifiability issue on $\Lambdab_m$ and $\f_m$, we adopt the usual assumption in the factor analysis literature by assuming that
\begin{equation*} \label{eq:2}
\var(\f_m)=\I_{K_m}, \;\; \text{ and } \;\; \Lambdab_m'\Lambdab_m=\D_m\in \Rcal^{K_m \times K_m} \text{ is a diagonal matrix.}
\end{equation*}
We also assume that $\f=(\f'_1,\ldots,\f'_M)' \in \Rcal^{K}$ is uncorrelated with $\u=(\u'_1,\ldots,\u'_M)' \in \Rcal^{p}$, where $K=\sum_{m=1}^M K_m$. Let $\Lambdab=\diag(\Lambdab_1,\ldots,\Lambdab_M)\in \Rcal^{p\times K}$ be the block diagonal matrix of the loading matrices from all modalities, and $\Sigmab_u=\var(\u)$ be the covariance matrix of the idiosyncratic errors. In the factor analysis literature, e.g. \cite{Fan2013b}, it is often assumed that $\Sigmab_u$ is sparse, i.e., after removing the variations contributed by the latent factors, the correlations among the idiosyncratic components are weak. Therefore, the idiosyncratic $\u$ can be viewed as a decorrelated version of the original variables $\x$.

We next employ a linear model to connect $\x_m$ with $y$, in that, 
\begin{equation} \label{eq:4}
  y = \sum_{m=1}^M \x_m' \betabs_m + \epsilon,
\end{equation}
where $\betabs_m \in \Rcal^{p_m}$ is the true effect of $\x_m$ on the response $y$, and $\epsilon$ is an error uncorrelated of the  covariates, with $\E(\epsilon)=0$ and $\var(\epsilon)=\sigmae^2$.

Suppose we have $n$ i.i.d.\ realizations of the data, $\Y = (y_1,\ldots,y_n)' \in \Rcal^n$, $\X_m = (\x_{1,m},\ldots, \x_{n,m})' \in \Rcal^{n \times p_m}$, $\X=(\X_1,\ldots,\X_M)\in \Rcal^{n\times p}$,  $\betabs=({\betab^{\ast}_1}',\ldots,{\betab^{\ast}_M}')'\in\Rcal^{p}$, and $\epsilonb=(\epsilon_1,\ldots,\epsilon_n)' \in \Rcal^n$. Then model \eqref{eq:4} can be written as
\begin{equation*} \label{eq:5}
  \Y \; = \; \sum_{m=1}^M\X_m\betabs_m+\epsilonb \; = \; \X\betab^{\ast}+\epsilonb. 
\end{equation*}
By the factor model \eqref{eq:1}, we have $\X_m=\F_m\Lambdab_m'+\U_m$, where $\F_m=(\f_{1,m},\ldots,\f_{n,m})'\in \Rcal^{n\times K_m}$ is the matrix of the $K_m$ factors in the $m$th modality pertaining to the $n$ subjects, and $\U_m=(\u_{1,m},\ldots,\u_{n,m})'\in \Rcal^{n\times p_m}$ is the matrix of idiosyncratic errors. Then, we have,
\begin{equation} \label{eq:6}
  \Y=\F\gammabs+\U\betabs+\epsilonb,
\end{equation}
where $\F=(\F_1,\ldots,\F_M)\in\Rcal^{n\times K}$, $\gammabs=({\betab^{\ast}_1}'\Lambdab_1,\ldots, {\betab^{\ast}_M}'\Lambdab_M)'\in \Rcal^K$, and $\U=(\U_1, \ldots,$ $\U_M) \in \Rcal^{n\times p}$. We call model (\ref{eq:6}) an integrative factor regression model. In the remainder of this article, we aim to show that model (\ref{eq:6}) can benefit estimation, selection and inference about $\betabs$. The intuition is that, after the latent factors $\F$ are separated, the idiosyncratic error $\U$ can be treated as the pseudo predictors. Such a decorrelation eases selection of $\betabs$, which in turn benefits the inference on $\betabs$. In addition, the factor decomposition also serves as a dimension reduction tool, which enables us to derive some closed-form results in inference. We remark that, in model (\ref{eq:6}), the coefficient $\betabs$ associated with $\U$ is the same as that associated with the original predictor $\X$. It is the main object of interest in our inference, as its component $\betabs_m$ reflects the effect of the $m$th modality $\x_m$ on the response $y$. Meanwhile, we treat $\gammabs$ as a nuisance parameter. We also remark that, one does \emph{not} necessarily have to perform factor decomposition for all data modalities. In practice, we first estimate the number of factors for each modality. For a particular modality that does not admit a  factor structure, we can set the corresponding $\Lambdab_m=\zero$ and $\x_m=\u_m$ in (\ref{eq:1}).

\section{Estimation}
\label{sec:estimation}

To fit model (\ref{eq:6}), we first estimate the latent variables $\F$ and $\U$, along with the number of latent factors $K_m$, using some well established methods in the factor analysis literature. We then estimate $\betabs_m$ through a penalized least squares approach. 

First, we estimate the latent variables $\F$ and $\U$, we adopt the method in \citet{Bai2012} and \citet{Fan2013b}, by running principal components analysis (PCA) on each individual modality $\X_m$. We then estimate $\F_m$ by $\sqrt{n}$ times eigenvectors corresponding to the largest $K_m$ eigenvalues of $\X_m\X_m'$. Denote this estimator by $\hat{\F}_m$. We next estimate $\Lambdab_m$ by $\Lambdabh_m=(1/n)\X_m'\hat{\F}_m$, and estimate $\U_m$ by $\hat{\U}_m=\X_m-\hat{\F}_m\Lambdabh_m'$, accordingly. \citet{Fan2013b} showed that $\hat{\F}_m$ is a consistent estimator, up to a rotation, of $\F_m$, under a pervasive condition that the latent factors should affect many variables; see Condition \ref{con:2} and its discussion in Section \ref{sec:test-whole}. We remark that, there are alternative ways to estimate the factors. For instance, methods such as \citet{Ma2013, Cai2013}, and \citet{lock2013joint} may also be applicable, as long as the resulting factor estimates are consistent. 

Next, we determine the number of the latent factors $K_m$ in each modality, which is usually unknown in practice. We use the method of \citet{Bai2002} to estimate $K_m$ by
\begin{equation} \label{eq:101}
\hat{K}_m = \argmin_{0\leq k\leq \tilde{M}}~ \log \left\{\frac{1}{np_m} \fnorm{\X_m-n^{-1}\Fh_{m_k}\Fh_{m_k}'\X_m}^2 \right\} + k g(n,p_m),
\end{equation}
where $\tilde{M}$ is a pre-defined upper bound on $K_m$, $\Fh_{m_k}$ is $\sqrt{n}$ times eigenvectors corresponding to the
largest $m_k$ eigenvalues of $\X_m\X_m'$, and $g(n,p_m)$ is a penalty function that,
\begin{equation*}
  g(n,p_m) = \frac{n+p_m}{np_m}\log \left(\frac{np_m}{n+p_m} \right), \quad \text{or} \quad
  g(n,p_m) = \frac{n+p_m}{np_m}\log \left(\min\{n,p_m \} \right).
\end{equation*}
For both choices, \cite{Bai2002} showed that $\hat{K}_m$ is a consistent estimator of $K_m$ under some regularity conditions. We make two additional remarks about $K_m$. First, in this article, we treat $K_m$ as fixed, which is reasonable in numerous scientific applications. As $K_m$ is related to the number of spiked eigenvalues of $\X_m \X_m'$, it is usually small. Second, following a common practice of the factor analysis literature, in our subsequent theoretical analysis, we treat $K_m$ as known. All the theoretical results remain valid conditioning on that there is a consistent estimator $\hat{K}_m$ of $K_m$.

Next, we estimate $\betabs$ and $\gammabs$. We replace $\F$ and $\U$ with the corresponding estimators $\Fh=(\Fh_1,\ldots,\Fh_M)$ and $\Uh=(\Uh_1,\ldots,\Uh_M)$, and solve a penalized least squares problem,
\begin{equation} \label{eq:7}
  (\gammabh, \betabh) = \argmin_{(\gammab,\betab)} \frac{1}{2n} \sum_{i=1}^n \left( y_i-\Fh_i'\gammab-\Uh_i'\betab \right)^2+\lambda \sum_{j=1}^p p(|\beta_j|),
\end{equation}
where $p(\cdot)$ is some general folded-concave penalty function, and $\lambda$ is a tuning parameter.  This class of penalty functions includes SCAD \citep{fan2011nonconcave} and MCP \citep{zhang2010nearly}. It assumes that $p(t)$ is increasing and concave in $t\geq 0$, and has a continuous first derivative $\dot{p}(t)$ with $\dot{p}(0+) > 0$. This optimization problem can be solved by standard proximal gradient descent algorithms \citep{parikh2014proximal}. We also briefly comment that, in our optimization \eqref{eq:7}, we do not impose the linear constraint that $\gammab = \Lambdab'\betab$, mainly because both $\F$ and $\Lambdab$ are unknown and unidentifiable in our setting. Besides, as we later show in Theorem \ref{thm:2} that, even without this constraint, the estimator $\betabh$ from \eqref{eq:7} achieves the minimax rate, and can consistently select the support and estimate the nonzero components of $\betabs$, as long as there are enough variables to estimate the latent factors well. We tune $\lambda$ in \eqref{eq:7} using the standard cross-validation method, following \citet{fan2011nonconcave}, while alternative criteria, e.g., the extended BIC \citep{Chen2008}, can also be used to tune $\lambda$.

Finally, we estimate $\sigmae^2$ by $\sigmaeh^2=n^{-1} \sum_{i=1}^n (y_i-\x_i'\tilde{\betab})^2$, where $\tilde{\betab}$ is the Lasso estimator of $\betabs$ obtained by $\tilde{\betab}=\argmin_{\betab} (2n)^{-1} \sum_{i=1}^n (y_i-\x_i'\betab)^2 + \lambda_{\epsilon} \lonenorm{\betab}$, where $\lambda_{\epsilon}$ is the tuning parameter. We show in the Supplementary Materials that $\sigmaeh^2$ is consistent to $\sigmae^2$. Actually, any consistent estimator of $\sigmae^2$ would suffice for the subsequent hypothesis testing procedures. Alternative methods such as scaled LASSO \citep{sun2013sparse}, refitted cross-validation \citep{fan2012variance}, or directly using the residuals from (\ref{eq:7}) can all be applied to estimate $\sigmae^2$ as well.

\section{Hypothesis test of a whole modality}
\label{sec:test-whole}

A crucial question in multimodal data analysis is to evaluate if a whole modality is significantly associated with the outcome, given other modalities in the model. For instance, in multi-omics analysis, it is of interest to test if DNA methylation correlates with the phenotypic traits related to genetic disorders given gene expression level \citep{richardson2016statistical}. In multimodal neuroimaging analysis, it is of interest to evaluate if functional imaging quantification for hypometabolism associates with the diagnosis of Alzheimer's disease, given structural magnetic resonance imaging of brain atrophy measurement \citep{ZhangShen2011}. The challenge here is that even a single modality often contains many more variables than the sample size.  

This is essentially a problem of testing a high-dimensional subvector of $\betabs$ in a high-dimensional regression model. Related testing problems have been extensively studied for a single modality data. For example, \citet{zhang2014confidence, van2014asymptotically, javanmard2014confidence} developed bias-corrected or de-sparsifying methods to test if a fixed-dimensional subvector of $\betabs$ in a high-dimensional linear or generalized linear model equals zero. In particular, under a general $M$-estimation framework, \citet{ning2017general} proposed a decorrelated score test for the same problem, i.e., to test if a subvector $\betabs_S=\zero$. They first showed that their score test statistic has a closed-form limiting distribution when the dimension of the subset $|S|$ is fixed. They then extended to the case where $\betabs_S$ can be any arbitrary subvector of $\betabs$ with $|S|$ diverging and even when $|S| > n$. Built on a pioneering work by Chernozhukov et al. (2013), they showed that the distribution of the supremum of the decorrelated score functions can be approximated by a multiplier bootstrap approach. Consequently, they employed bootstrap simulations to obtain the critical values of the limiting distribution to form the rejection region. Our test differs from \citet{ning2017general}. When $|S|$ diverges, the test of \citet{ning2017general} no longer has a closed-form limiting distribution, and they had to resort to bootstrap for critical values. By contrast, we are able to obtain a closed-form limiting distribution for our test when $|S|$ diverges. This is due to that, instead of using the observed likelihood, we perform factor decomposition on $\X_m$ first, then use the factor model as a dimension reduction tool to reduce a high-dimensional test to a fixed-dimensional one. Our method does pay the extra price that we need to estimate the latent factors to plug into the likelihood function. However, as we show later, this extra cost can be well controlled. We also numerically compare with \citet{ning2017general} in Section \ref{sec:sim-test-whole}. We show that our test is as powerful, and often more powerful than the test of \citet{ning2017general}.  

Formally, for our multimodal analysis, we aim at testing the following pair of hypotheses: 
\begin{eqnarray} \label{eqn:hypothesis-whole}
  H_0:\betabs_m=\zero \quad \textrm{ versus } \quad H_a:\betabs_m\neq \zero. 
\end{eqnarray}
We perform factor decomposition on the $m$th modality following (\ref{eq:1}). Then, 
\begin{equation*}
  \x'\betab=\x_{-m}'\betab_{-m}+\f'_{m}\thetab_m + \u_m'\betab_m,
\end{equation*}
where $\thetab_m=\Lambdab'_m\betab_m$. The null hypothesis $\betabs_m=\zero$ implies that $\thetabs_m=\zero$, where
$\thetabs_m=\Lambdab'_m\betabs_m$. Therefore, under the null hypothesis, testing $\gammabs_m$ is the same as testing $\betabs_m$. The difference is that $\thetabs_m \in \Rcal^{K_m}$ is a low-dimensional vector, while $\betabs_m\in \Rcal^{p_m}$ is high-dimensional. As such, the factor model plays the role of dimension reduction for our testing problem.  Actually, directly testing $\betabs_m$ is challenging, since the dimension of $\betabs_m$ diverges with the sample size, and there is not a closed form for the limiting distribution of $\betabh_m - \betabs_m$, where $\betabh_m$ is an estimator of $\betabs_m$. On the other hand, we note that, under the alternative hypothesis, the magnitude of $\gammabs_m$ can be different from that of $\betabs_m$. As such, the power of the test that is built on $\gammabs_m$ can be different from the one that is built on $\betabs_m$. We later study the local power property of the test based on $\gammabs_m$ in detail.

Next, we develop a factor-adjusted decorrelated score test, and show that it is asymptotically efficient when the latent factors can be well estimated. Following \citet{ning2017general}, and based on the Gaussian quasi-likelihood, we define the decorrelated score function as
\vspace{-0.05in}
\begin{equation*} \label{eq:8}
  \S(\betab,\thetab_m)= \frac{1}{n\sigmae^2} \sum_{i=1}^n (y_i-\f_{i,m}'\thetab_m-\z_i'\betab)(\f_{i,m}-{\W^{\ast}}'\z_i),
\end{equation*}
where $\z_i=(\x'_{i,-m}, \u'_{i,m})' \in \Rcal^p$, $p_{-m}=p-p_m$, and $\W^{\ast}=\E(\x_{i,-m}^{\otimes 2})^{-1} \E(\x_{i,-m}\f'_{i,m})\in \Rcal^{p_{-m}\times K_m}$, which is essentially the projection of the latent factors onto the linear space spanned by $\x_{-m}$. Such a projection is needed to control the variability of the high-order terms in establishing the central limit theorem in Theorem \ref{thm:1} \citep{ning2017general}. In the high-dimensional setting, we need some sparsity condition on $\W^{\ast}$; see Condition \ref{con:4}, and solve a regularized problem to obtain its consistent estimator. We treat the score function as a function of $\gammab_m$. Under the null hypothesis, we propose to estimate $\S(\betab,\thetab_m)$ by,
\begin{equation*} \label{eq:9}
  \hat{\S}(\betabh_{-m},\zero)= \frac{1}{n\sigmaeh^2} \sum_{i=1}^n (y_i-\x_{i,-m}'\betabh_{-m})(\hat{\f}_{i,m}-\hat{\W}'\x_{i,-m}), 
\end{equation*}
where $\hat{\f}'_{i,m}$ is the $i$th row of the estimated latent factor matrix $\Fh_m$, $\sigmaeh^2$ is any consistent
estimator of $\sigmae^2$ that satisfies Condition \ref{con:5} below, and $\betabh_{-m}\in \Rcal^{p_{-m}}$ and $\hat{\W}\in
\Rcal^{p_{-m}\times K_m}$ are obtained by solving the following optimization problems,
\begin{align} 
  (\betabh_{-m},\thetabh_m) & =  \argmin_{(\betab_{-m},\thetab_m)} \frac{1}{2n} \sum_{i=1}^n \left( y_i-\x_{i,-m}'\betab_{-m}-\hat{\f}'_{i,m}\thetab_m \right)^2 + \lambda_1 \lonenorm{\betab_{-m}},  \label{eq:10} \\
  \hat{\W} & =  \argmin ~ \lonenorm{\W}, \text{ such that } \left\lVert\frac{1}{n} \sum_{i=1}^n \x_{i,-m}\left( \hat{\f}'_{i,m}-\x'_{i,-m}\W \right) \right\rVert_{\infty} \leq \lambda_2. \label{eq:11}
\end{align}
We make a few remarks regarding our score function and compare it to \citet{ning2017general}. First, we need to estimate the latent factors in the decorrelated score function, while in the score function of \citet{ning2017general}, the covariates are fully observed. This introduces an additional layer of complexity when analyzing the statistical property of the score function. Later in Theorems \ref{con:1} and \ref{con:4}, we carefully evaluate the extra cost of estimating the latent factors. Second,  we need to involve $\gammab_m$ in (\ref{eq:10}), even under the null hypothesis $\gammab_m=\zero$. This is because, if $\gammab_m$ is removed from (\ref{eq:10}), $\betabh_{-m}$ is no longer consistent to $\betabs_{-m}$ under the alternative, which would in turn impact the power of the test. On the other hand, $\gammabh_m$ is not used in constructing the test statistic, but only $\betabh_{-m}$ is. Finally, in order to consistently estimate $\W^{\ast}$, we need to solve a non-typical Dantzig problem (\ref{eq:11}), where the latent factors are replaced by their corresponding estimators.

Next, we compute the variance of the score function by using the Fisher information. By the sandwich formula, the information matrix is 
\vspace{-0.05in}
\begin{equation*} \label{eq:12}
  \I^{\ast}_{\thetab_{m}|\betab_{-m}} = \sigmae^{-2} \left\{ \E(\f_{i,m}^{\otimes 2}) - \E(\f_{i,m}\x'_{i,-m}) \E(\x_{i,-m}^{\otimes 2})^{-1} \E(\x_{i,-m}\f_{i,m}') \right\} \in \Rcal^{K_m\times K_m}, 
\end{equation*}
which can be estimated by 
\begin{equation*} \label{eq:13}
  \hat{\I}_{\thetab_{m}|\betab_{-m}} = \sigmaeh^{-2}\left\{\frac{1}{n} \sum_{i=1}^n \hat{\f}_{i,m}^{\otimes 2} - \hat{\W}' \left(\frac{1}{n}\sum_{i=1}^n \x_{i,-m} \hat{\f}_{i,m}' \right)\right\}.
\end{equation*}
Then, our test statistic is given by 
\begin{eqnarray*}
  \T_n=\sqrt{n} \, \Ihat^{-1/2} \hat{\S}(\betahnom,\zero). 
\end{eqnarray*}
We next show that, under the null hypothesis, the asymptotic distribution of $\T_n$ is $N(\zero,\I_{K_m})$. In other words, $\T_n$ is asymptotically efficient. We first begin with a set of conditions.
\medskip

\begin{con} \label{con:1} 
For $m\in [M]$, suppose $\{(\f'_{i,m},\u'_{i,m})' \}_{i=1}^n$ are i.i.d.\ uncorrelated sub-Gaussian random vectors with zero mean. That is, $\E(\f_{i,m})=\zero$, $\E(\u_{i,m})=\zero$, and $\E(\f_{i,m}\u_{i,m}')=\zero$. Moreover, $\E\{\exp(t\alphab'\f_{i,m})\}\allowbreak \leq \exp(C \ltwonorm{\alphab}^2t^2/2)$, and $\E\{\exp (t\alphab'\u_{i,m})\} \leq \exp(C \ltwonorm{\alphab}^2t^2/2)$, for some constant $C$. In addition, for all $k\in [K_m]$, $\x_{i,-m}'\ws_k$ are i.i.d.\ sub-Gaussian such that $\E\{\exp(t\x_{i,-m}'\ws_k) \}\allowbreak \leq \exp(Ct^2/2)$, where $\ws_k$ is the $k$th column of $\Wbs$. Additionally, $\{\epsilon_i \}_{i=1}^n$ are i.i.d.\ sub-Gaussian with zero mean, and $\epsilon_i$ is uncorrelated with $(\f'_{i,m},\u'_{i,m})'$ for all $m\in [M]$.
\end{con}
\medskip

\begin{con} \label{con:2} 
For $m\in [M]$, suppose $0<c\leq\lambda_{\min}(\Lambdab_m'\Lambdab_m/p_m) \leq \lambda_{\max}(\Lambdab_m'\Lambdab_m/p_m) \leq C < \infty$, for some positive constants $c$ and $C$. 
\end{con}
\medskip

\begin{con} \label{con:3} For $m\in [M]$, $s,t\in [p_m]$, $i,j\in [n]$, suppose $\E[p_m^{-1/2}\{\u_{i,m}'\u_{j,m}-\E(\u_{i,m}'\u_{j,m}) \}^4]$ $\leq C$, and $\E \ltwonorm{p_m^{-1/2}\Lambdab_m'\u_{i,m}}^4 \leq C$. Moreover, $\supnorm{\Lambdab_m}\leq C$, $\lambda_{\min}(\Sigmab_{u_m})>c$, $\lonenorm{\Sigmab_{u_m}} \leq C$, where $\Sigmab_{u_m}=\var(\u_m)$, and $\min_{s,t\in [p_m]} \var(u_{i,m_s}u_{i,m_t})>c$.
\end{con}
\medskip

\begin{con} \label{con:4} 
Let $\ss_w=\max_{k\in [K_m]} |\supp(\ws_k)|$. Suppose $\ss_w \log (p_{-m})\{1\vee (n^{1/4}/\sqrt{p_m}) \}=o(n^{1/2})$, and $[\ss_{-m} \{\sqrt{(\log p_{-m})/n}+1/\sqrt{p_m}\}]\cdot\sqrt{\log(p_{-m})}\{1\vee (n^{1/4}/\sqrt{p_m}) \}=o(1)$.
\end{con}
\medskip

\begin{con} \label{con:5} 
Suppose $\sigmaeh^2=\sigmae^2+o_P\left(1 \right)$.
\end{con}

\begin{con} \label{con:21} Suppose $0<c\leq\lambda_{\min}(\Itrue)$.
\end{con}
\medskip

\noindent
Condition \ref{con:1} is a typical sub-Gaussian assumption for high-dimensional problems. Condition \ref{con:2} is the pervasive condition, and is common in factor analysis \citep{Fan2013b}. It requires that the latent factors affect a large number of variables. This is reasonable for a variety of multimodal data. For instance, in multi-omics data, some genetic factors are believed to impact both gene expression and DNA methylation, and in multimodal neuroimaging, some neurological factors affect both brain structures and functions. Condition \ref{con:3} imposes some technical requirements on the loading matrix and idiosyncratic component. Together, Conditions \ref{con:2} and \ref{con:3} ensure that $\f_{i,m}$ and $\u_{i,m}$ can be consistently estimated by the PCA method \citep{Fan2013b}. Condition \ref{con:4} is a sparsity condition on $\W^{\ast}$ and $\betasnom$, which requires $\ss_w$ and $\ss_{-m}$ to be much smaller than $n$. Under such a condition, $\W^{\ast}$ and $\betasnom$ can both be consistently estimated, even if the latent factors are unknown; see Lemmas 6 and 8 in the Supplementary Materials for more details. We remark that this sparsity assumption on $\W^{\ast}$ is weaker and more flexible than requiring both $\E(\x_{i,-m}^{\otimes 2})^{-1}$ and $\E(\x_{i,-m}\f'_{i,m})$ are sparse. Condition \ref{con:5} ensures the estimator of $\sigmae^2$ is consistent. Condition \ref{con:21} ensures the information matrix is invertible. We also remark that, if there is no factor in the $m$th modality, Condition \ref{con:1} reduces to the sub-Guassian assumption on $\x_{i,m}$, and Conditions \ref{con:2} and \ref{con:3} are no longer needed for that modality.

We next obtain a closed-form limiting distribution for the test statistic $T_s$. 

\begin{thm} \label{thm:1}
  Suppose Conditions \ref{con:1}--\ref{con:21} hold. Suppose $\lambda_1\asymp \sqrt{(\log p_{-m})/n}+1/\sqrt{p_m}$, and $\lambda_2 \asymp \sqrt{(\log p_{-m})/n}\{1\vee (n^{1/4}/\sqrt{p_m}) \}$. Then, under $H_0: \betabs_m=\zero$, it holds that
  \begin{equation*}
    \T_n \xrightarrow{D} N(\zero,\I_{K_m}). 
  \end{equation*}
\end{thm}

\noindent
By Theorem \ref{thm:1}, we reject the null hypothesis if $n\{\hat{\S}(\betabh_{-m},\zero) \}' \Ihat^{-1}\hat{\S}(\betabh_{-m},\zero)> \chi^2_{\alpha}(K_m,0)$, where $\chi^2_{\alpha}(K_m,0)$ is the $\alpha$-upper quantile of the $\chi^2$-distribution with $K_m$ degrees of freedom.

We next explicitly discuss the benefit and the extra cost of our factor-based test when compared with \citet{ning2017general}. The main difference is that, through latent factors, we obtain a closed-form limiting distribution and do not have to resort to bootstrap. The price we pay mainly lies in Condition \ref{con:4} and the choices of $\lambda_1$ and $\lambda_2$. Actually, the extra term $1/\sqrt{p_m}$ appearing in both Condition \ref{con:4} and $\lambda_1, \lambda_2$ reflects the estimation error caused by using $\fh_{i,m}$ to estimate $\betasnom$. The term $n^{1/4}/\sqrt{p_m}$ is due to the same reason for estimating $\ws_k$. Therefore, the choices of the tuning parameters $\lambda_1$ and $\lambda_2$ need to be adjusted accordingly, by taking into account such extra estimation errors.

We further consider three scenarios. First, when $p_m\gg n$, both $1/\sqrt{p_m}$ and $n^{1/4}/\sqrt{p_m}$ are dominated by $\sqrt{(\log p_{-m})/n}$. Therefore, the estimation errors of $\betahnom$ and $\hat{\W}$ reach the optimal oracle rate, i.e. the best rate as if the latent factors were known; see Lemmas 6 and 8 in the Supplementary Materials. In this case, using the factor estimates  actually does not incur any extra cost. The reason is that many variables are used to estimate the latent factors, and its estimation error is so small that it would not affect the inference on $\betabs_m$. Second, when $p_m=o(n)$, the estimation errors of $\betahnom$ and $\hat{\W}$ would be greater than the optimal rate. However, the central limit theorem still holds, given proper choices of $\lambda_1$ and $\lambda_2$, and more stringent sparsity conditions on $\betasnom$ and $\W^{\ast}$ in Condition \ref{con:4}. Third, in a special case where variables in all modalities are driven by exactly the same latent factors, even we perform the hypothesis test on the $m$th modality, we could use variables from all different modalities to estimate the latent factors. Then, the terms $1/\sqrt{p_m}$ and $n^{1/4}/\sqrt{p_m}$ become $1/\sqrt{p}$ and $n^{1/4}/\sqrt{p}$, respectively, which are naturally dominated by $\sqrt{(\log p_{-m})/n}$. In this case, the optimal oracle rate is again attained. Such a result can be viewed as a blessing of the dimensionality for the factor model. In summary, our method is most suitable for testing the significance of a modality containing many variables, or for multimodal data with a large number of variables driven by some common latent factors.

Next, we study the power of the proposed test under the local alternative $H_{a_n} : \betasm=\b_{m_n}$, where $\b_{m_n}$ is a sequence converging to $\zero$ as $n\to \infty$. Since we use the latent factors to transform the test on $\betasm$ to the one on $\gammasm$, we show that the local power of the test depends on $\c_{m_n}=\Lambdab_m'\b_{m_n}$. We consider the following parameter space under the local alternative, $\N = \big\{ \betabs: \betasm=\b_{m_n}, |\supp(\betasnom)|=s^{\ast}_{-m}, \text{ where } \ss_{-m}\ll n \big\}$. The next theorem gives the limiting distribution of $Q_n=\T_n'\T_n$ uniformly for all $\betabs\in \N$ under the local alternative. 
\begin{thm} \label{thm:4} Suppose the conditions of Theorem \ref{thm:1} hold. Suppose $\lambda_1\asymp \sqrt{(\log p_{-m})/n}+1/\sqrt{p_m}$, $\lambda_2\asymp \sqrt{(\log p_{-m})/n}\{1\vee (n^{1/4}/\sqrt{p_m}) \}$, $\ltwonorm{\b_{m_n}}=o(1/\sqrt{\log n})$, and $\ltwonorm{\c_{m_n}}=o(1/\sqrt{\log n})$. Then, under the $H_{a_n}$, it holds uniformly for all $\betabs\in \N$ that
\begin{equation*}
\sup_{x>0} \left| \pr(Q_n\leq x)- \pr\left\{ \chi^2(K_m,h_{m_n})\leq x \right\} \right| \to 0,
\end{equation*}
where $h_{m_{n}}=n\c_{m_n}'\I^{\ast}_{\gammab_{m}|\betanom}\c_{m_n}$.
\end{thm}

\noindent
Since $h_{m_n}\asymp n \ltwonorm{\c_{m_n}}^2$, Theorem \ref{thm:4} implies that the local power of our test essentially depends on $\ltwonorm{\c_{m_n}}$. If we let $\ltwonorm{\c_{m_n}}=Cn^{-\phi_{\gamma_m}}$, the local power is to exhibit some transition behavior depending on the value of $\phi_{\gamma_m}$, which is summarized in the next corollary. 
\begin{cor}
\label{cor:1}
Suppose the conditions of Theorem \ref{thm:4} hold. Then, 
\begin{enumerate}[(a)]
    \item $\lim_{n\to \infty} \sup_{\betabs\in \N} \sup_{x>0} |\pr(Q_n\leq x) - \pr(\chi^2(K_m,0)\leq x)| \to 0$, if $\phi_{\gamma_m}>1/2$; 
    \item $\lim_{n\to \infty} \sup_{\betabs\in \N} \sup_{x>0} |\pr(Q_n\leq x) - \pr(\chi^2(K_m,h)\leq x)| \to 0$, if $\phi_{\gamma_m}=1/2$; 
    \item $\liminf_{n\to \infty} \sup_{\betabs\in \N} \pr(Q_n> x) =1$, if $\phi_{\gamma_m}<1/2$; 
    \end{enumerate}
where $h=\lim_{n\to\infty}n\c_{m_n}'\I^{\ast}_{\gammab_m|\betanom}\c_{m_n}$ in (b), and (c) holds for any $x>0$.
\end{cor}

\noindent
We make some remarks. First, Corollary \ref{cor:1} shows that, the local power is to converge to the type I error if $\phi_{\gamma_m}>1/2$; to a non-central $\chi^2$-distribution if $\phi_{\gamma_m}=1/2$; and to 1 if $\phi_{\gamma_m}<1/2$. Such a transition behavior is analogous to the classical local power results, which showed that the root-$n$ local alternative is the transition point of the local power \citep{Vaart2000}. Second, the existing debiased method \citep{van2014asymptotically} and the decorrelated method \citep{ning2017general} only established the root-$n$ local power results when the dimension of the parameters being tested is fixed. Moreover, even though \cite{ning2017general} utilized a multiplier bootstrap method to extend their test from a single parameter to arbitrarily many parameters, they only studied the local power when testing a single parameter. Our local power result differs in that we allow the dimension of $\betasm$ to grow with $n$, whereas we fix the dimension of $\gammasm$.  Finally, Theorem \ref{thm:4} shows that the local power depends on the magnitude of $\gammasm$, or $\c_{m_n}=\Lambdab_m'\b_{m_n}$. This is again due to that we transform the test of $\betasm$ to that of $\gammasm$. Therefore, the power of our test depends on the relation between the loadings and where the alternative hypothesis occurs.

Next, we give some specific examples to further illustrate the power behavior of our proposed test. To simplify the discussion, we set $K_m=1$.
\medskip

\begin{exa} \label{exa:1} Let $\Lambdab_m=(D_L,1,\ldots,1)'$ and $\b_{m_n}=(D_Sn^{-1/2},0,\ldots,0)'$, where $D_L$ and $D_S$ are two constants. In this case, $\ltwonorm{\c_{m_n}}=D_SD_Ln^{-1/2}$, and thus the power of our test depends on the product $D_SD_L$. On the contrary, even one had known apriori that the alternative only occurs at the first coordinate, and performs a debiased or decorrelated test on that coordinate, its local power depends on $D_S$. Therefore, when $D_L$ is large and $D_S$ is small, our testing method gains power. On the other hand, when $D_L$ is small and $D_S$ is large, the alternative methods may be more powerful.
\end{exa}
\medskip

\begin{exa} \label{exa:2} 
Let $\Lambdab_m=(1,1,\ldots,1)'$, and $\b_{m_n}=(C_1n^{-2/3},C_2n^{-2/3}\ldots,C_Ln^{-2/3},0,\ldots,0)'$. In this case, if $L\gg n^{1/6}$, $\ltwonorm{\c_{m_n}}\gg n^{-1/2}$, then the power of our test converges to 1. On the contrary, if one performs a debiased or decorrelated test on each element of $\betasm$, there is no power. Our testing method gains power in this example too. 
\end{exa}
\medskip

\begin{exa} \label{exa:3} 
Let $\Lambdab_m=(0,1,\ldots,1)'$, and $\b_{m_n}=(c_n,0,\ldots,0)'$. In this case, no matter how large $c_n$ is, our testing method has no power to detect the alternative, because $\c_{m_n}=\zero$.
\end{exa}
\medskip 

As we have seen in these examples, when transforming the test from $\betasm$ to $\gammasm$, our method does not necessarily lose power, but can gain power in some situations. For instance, in Example \ref{exa:1}, the variables with large loadings on the latent factors have nonzero coefficients, while in Example \ref{exa:2}, many variables with nonzero loadings have nonzero coefficients. In such cases, our test gains power. On the other hand, in Example \ref{exa:3}, the product of the loadings and the nonzero coefficients is small, then our test has little power.

\section{Hypothesis test of linear combinations of predictors of one or more modalities}
\label{sec:test-lincomb}

Another important question in multimodal data analysis is to test if some linear combinations of predictors, within the same modality or across different modalities, is significantly correlated with the response. This is because multimodal data often measures different aspects of related quantities. For instance, in multi-omics studies, expression data measures how genes are expressed, methylation data measures how DNA molecules are methylated, and both data may be related to the same set of genes. In multimodal  neuroimaging analysis, brain structures, functions, and chemical constituents of the same brain regions are often measured simultaneously. As such, it is of great scientific interest to test if various measurements on a particular gene or brain region are associated with the outcome.

\citet{shi2019linear} considered a similar testing problem in a high-dimensional generalized linear model for a single modality data. They derived the corresponding partially penalized likelihood ratio test, score test and Wald test, and showed that the three tests are asymptotically equivalent. They allowed the dimension of the model to grow with the sample size, as long as the dimension of the subvector being tested and the number of linear combinations are smaller than the sample size. Our method differs from \citet{shi2019linear} in several ways. \citet{shi2019linear} treated the design matrix $\X$ as fixed, while we treat $\X$ as i.i.d.\ random realizations from some distributions. More importantly, we do not directly use the observed $\X$, but instead perform a factor decomposition and use the decorrelated idiosyncratic components as the pseudo design matrix. We explicitly show in Theorem \ref{thm:2} that such a factor-adjusted step leads to less stringent conditions to reach the variable selection and estimation consistency. Moreover, since variable selection consistency is needed to correctly calculate the variance of the test statistic, as shown in Theorem \ref{thm:3}, our method also requires less stringent conditions to establish the limiting distribution of the test statistic. Moreover, our model concerns with data with multiple modalities, instead of a single modality as in \citet{shi2019linear}. High correlations are commonly observed in multimodal data, and as such the factor-adjusted decorrelation step becomes essential. Relatedly, \citet{ZhuBradic2018} proposed a test for a linear combination of predictors under a unimodal linear regression model. Even though they did not restrain the size or the sparsity of the model, they only considered a single linear combination, and required the eigenvalues of the covariate covariance matrix $\var(\x)$ to be bounded. By contrast, both \citet{shi2019linear} and we consider jointly testing multiple linear combinations of predictors, and we do not require the eigenvalues of $\var(\x)$ to be bounded. We further numerically compare with \citet{shi2019linear} and \citet{ZhuBradic2018} in Section \ref{sec:sim-test-lincomb}.

Formally, we consider testing the following pair of hypotheses:
\begin{eqnarray} \label{eqn:hypothesis-lincomb}
  H_0:\A\betabs_T=\b \quad \textrm{ versus } \quad H_a: \A\betabs_T\neq \b, 
\end{eqnarray}
where $\A\in \Rcal^{r\times t}$, $\b\in \Rcal^r$, $\betabs_T\in \Rcal^t$ is a subvector of $\betabs$, and $T\subset [p]$ is a low-dimensional index set with $|T|=t < n$. This simultaneously tests $r$ linear combinations of $\betabs_T$, with $r<n$. We next develop a factor-adjusted Wald test. 

To construct the test statistic, we first consider a penalized least squares problem,
\begin{equation} \label{eq:14}
  (\gammaha,\betaha)=\argmin_{(\gammab,\betab)} \frac{1}{2n} \sum_{i=1}^n \left( y_i-\Fh_i'\gammab-\Uh_i'\betab \right)^2 + \lambda_a \sum_{j\not\in T} p(|\beta_j|).
\end{equation}
This is essentially the same as \eqref{eq:7}, except that, instead of penalizing all variables in $\betab$, we do not penalize $\beta_j$ for $j \in T$. This is to avoid introducing bias when estimating $\betas_j$ for $j\in T$, which is needed for Theorem \ref{thm:3}. A similar idea was also adopted in \citet{shi2019linear}.

Given $\betabh_a$, our factor-adjusted Wald test statistic is given by
\begin{equation*} 
T_w=(\A\betahat-\b)'(\A\Omegabh_T\A')^{-1}(\A\betahat-\b)/\sigmaeh^2,
\end{equation*} 
where $\betahat$ is the sub-vector of $\betaha$ with indices in $T$, $\Omegabh_T$ is the first $T$ rows and columns of 
\vspace{-0.05in}
\begin{equation*}
  \Omegabh_{T\cup \Sha}= n
  \begin{pmatrix}
    \Uh_T'\Uh_T & \Uh_T'\Uh_{\Sha} \\
    \Uh_{\Sha}'\Uh_T & \Uh_{\Sha}'\Uh_{\Sha}
  \end{pmatrix}^{-1},
\end{equation*}
$\Sha=\{j\in T^c: \betah_{a,j}\neq 0 \}$, and $\sigmaeh^2$ is any consistent estimator of $\sigmae^2$. In this test statistic, $\Sha$ plays a critical role in calculating the variance of $\A\betahat-\b$. In fact, $\hat{S}_a$ needs to be consistent to $S_a=\{j\in T^c:\betas_j\neq 0 \}$ in order for the variance to be valid. Such a consistency is guaranteed by Theorem \ref{thm:2}. 

We next present a set of regularity conditions. 

\medskip
\begin{con} \label{con:6}
Suppose $c\leq \lambda_{\min}\{ \E(\u^{\otimes 2}) \} \leq \lambda_{\max}\{ \E(\u^{\otimes 2}) \} \leq C$ for some postive constants $c$ and $C$.
\end{con}
\medskip

\begin{con} \label{con:7} 
Suppose $\Lsupnorm{\E(\u_{\ts}^{\otimes 2})^{-1}}\leq C$.
\end{con}
\medskip

\begin{con} \label{con:8} 
Suppose $\Lsupnorm{\E(\u_{(\ts)^c}\u'_{\ts})\{\E(\u_{\ts}^{\otimes 2})^{-1}\}}\leq C$.
\end{con}
\medskip

\begin{con} \label{con:9} 
Suppose $d_n= \min\{|\betas_j|:\betas_j\neq 0 \}/2\gg \lambda_a\gg \delta_n$, where $\delta_n=\keyrate$, $p_{\min}=\min_{m\in [M]} p_m$, and $\lambda_a \dot{p}(d_n)=o(\delta_n)$, where $\dot{p}$ is the first derivative. 
\end{con}
\medskip

\noindent
We first note that Conditions \ref{con:6}--\ref{con:8} are imposed on $\u$, instead of on $\x$. Since $\u$ can be viewed as the residual of $\x$ after the latent factors are removed, the correlations among the variables in $\u$ are much weaker than those in $\x$. In particular, Conditions \ref{con:6} and \ref{con:7} are needed to avoid singularity of $\E(\u^{\otimes 2})$ and $\E(\u_{\ts}^{\otimes 2})$. Condition \ref{con:8} is the well-known irrepresentable condition, which is necessary for establishing the variable selection consistency. \cite{shi2019linear} required such a condition to hold for the Gram matrix $\X'\X$, which essentially requires the correlations among $\X$ must be small. This condition hardly holds for multimodal data. By contrast, we only impose such a condition on $\E(\u^{\otimes 2})$, which requires the idiosyncratic components not to be highly correlated. This condition is well accepted in the factor model literature. Indeed, when an exact factor model is assumed, $\E(\u^{\otimes 2})$ is a diagonal matrix, then Condition \ref{con:6} naturally holds. 

We now establish the variable selection and estimation consistency of the estimator $\betabh_a$ in (\ref{eq:14}), which is essential for deriving the asymptotic distribution of $T_w$. 

\begin{thm} \label{thm:2}
  Suppose Conditions \ref{con:1}--\ref{con:3} and \ref{con:6}--\ref{con:9} hold. Then there exists a solution $(\gammabh_a,\betabh_a)$ of (\ref{eq:14}) such that, with probability tending to 1, the following results hold: 
  \begin{enumerate}[(a)]
  \item (sign consistency) $\mathrm{sign}(\betabh_a)=\mathrm{sign}(\betabs)$; 
  \item ($L_{\infty}$ consistency) $\supnorm{\betabh_{a,T\cup S_a}-\betabs_{T\cup S_a}}=O_P\left(\delta_n \right)$;
  \item ($L_2$ consistency)  $\ltwonorm{\betabh_{a,T\cup S_a}-\betabs_{T\cup S_a}}=O_P\left(\sqrt{t+s_a}\delta_n\right)$, where $s_a=|S_a|$;
  \item (asymptotic expansion)
    $\sqrt{n}(\betabh_{a,T\cup S_a}-\betabs_{T\cup S_a})=n^{-1/2}\K_n^{-1} \U'_{T\cup S_a} \epsilonb+o_P\left(1 \right)$,
    where $\K_n=(1/n)\Uts'\Uts$, if we further have that $p_{\min}\gg n^{3/2}$, and $\sqrt{n}\lambda_a\dot{p}(d_n)=o(1)$.
  \end{enumerate}
\end{thm}

We again explicitly examine the benefit and the extra cost of our factor-based test compared with \citet{shi2019linear}. The main difference is that we obtain the variable selection and estimation consistency under much weaker conditions than \citet{shi2019linear}. The price we pay lies in $\delta_n$, which reflects the convergence rates in (b) and (c) of Theorem \ref{thm:2}. Particularly, the component $n^{1/4}/\sqrt{p_{\min}}$ in $\delta_n$ is due to the factor estimation. We consider two scenarios. First, when all data modalities have a large number of variables, i.e. $p_{\min}\gg n^{1/2}$, then $\delta_n=\sqrt{(\log p)/n}$, which makes the convergence rates in (b) and (c) to be minimax optimal. This is because when there are enough variables to estimate the latent factors well, the extra factor estimation error becomes so small that it would not affect the estimation error on $\betabh_a$. Second, when one modality has only a small number of variables, i.e. $p_m=o(n^{1/4})$ for some $m \in [M]$, estimating the latent factors in that modality becomes challenging, and the resulting estimation error would slow the convergence of $\betabh_a$. In this case, one possible alternative solution is to skip factor decomposition for that particular modality, but directly use $\X_m$ in (\ref{eq:14}) and solve
\begin{equation*}
  (\gammaha,\betaha)=\argmin_{(\gammab,\betab)} \frac{1}{2n} \sum_{i=1}^n \left( y_i-\X_{i,m}'\betab_m-\Fh_{i,-m}'\gammab-\Uh_{i,-m}'\betab_{-m} \right)^2 + \lambda_a \sum_{j\not\in T} p(|\beta_j|),
\end{equation*}
where $\X_{i,m}'$, $\Fh_{i,-m}'$ and $\Uh'_{i,-m}$ denote the $i$th row of $\X_{-m}$, $\Fh_{-m}$ and $\Uh_{-m}$, respectively. Finally, we note that the variable selection and estimation consistency of $\betabh$ in (\ref{eq:7}) is directly implied by Theorem \ref{thm:2} if we treat $T$ as the empty set.

Next, we study the asymptotic distribution of our test statistic $T_w$, and show that it can be uniformly approximated by a $\chi^2$-distribution under both $H_0$ and $H_a$. We need two more regularity conditions.

\medskip
\begin{con} \label{con:10}
  Suppose $\ltwonorm{\h_n}=O(\sqrt{r/n} )$, and $\lambda_{\max}\{ (\A\A')^{-1} \} \leq C$ for some constant $C$, where $\h_n=\A\betabs_T-\b$.
\end{con}
\medskip

\begin{con} \label{con:11} Suppose ${r^{1/4}}{n^{-1/2}} \E |\u'_{\ts}\Sigmab_{u,\ts}^{-1}\u_{\ts}|^{3/2}\to 0$, where $\Sigmab_{u,\ts}^{-1}$ is the inverse of the submatrix of $\Sigmab_u$ with rows and columns in $\ts$.
\end{con}
\medskip

\noindent Condition \ref{con:10} regulates the local alternative $\h_n$ and avoids singularity of $\A\A'$. Condition \ref{con:11} is a Lyapunov condition to ensure the asymptotic normality of $\betabh_{a,\ts}$, which is the key to establish the $\chi^2$-approximation. 

\begin{thm} \label{thm:3} Suppose the conditions of Theorem \ref{thm:2} and Conditions \ref{con:10} and \ref{con:11} hold, $p_{\min}\gg n^{3/2}$, $\sqrt{n}\lambda_a\dot{p}(d_n)=o(1)$, and $t+s_a=o(n^{1/3})$. Then it holds that
  \begin{equation*}
    \sup_x \left| \pr(T_w\leq x)-\pr\{\chi^2(r,\nu_n)\leq x\} \right| \to 0, 
  \end{equation*}
  where $\nu_n=n\h_n'(\A\Omegab_T\A')^{-1}\h_n/\sigmae^2$, $\Omegab_{T}$ is the the submatrix of $\Sigmab_{u,\ts}^{-1}$ with rows and columns in $T$.
\end{thm}
\noindent
By Theorem \ref{thm:3}, we reject $H_0:\A\betabs_T=\b$ if $T_w>\chi^2_{\alpha}(r,0)$, where $\chi^2_{\alpha}(r,0)$ is the $\alpha$-upper quantile of the $\chi^2$-distribution with $r$ degrees of freedom. The limiting distribution we establish in Theorem \ref{thm:3} is the same as the classical Wald test result for a low-dimensional linear regression model \citep{shi2019linear}.

We also remark that, the requirement of $p_{\min}\gg n^{3/2}$ in Theorem \ref{thm:3} ensures that the latent factors in each modality can be well estimated. Therefore, the extra factor estimation error would not affect the limiting distribution of $T_w$. This condition is more stringent than that of $p_{\min}\gg n^{1/2}$, which guarantees the minimax optimal rate of estimation in Theorem \ref{thm:2}. This is because hypothesis testing is a more challenging task than estimation. 

Finally, write $\ltwonorm{(\A\Omegab_T\A')^{-1/2}\h_n}=Cn^{-\phi_{\nu}}$ for some constant $C>0$, and let $\mathcal{\tilde{N}} = \big\{\betabs : \ltwonorm{\A\betabs_T-\b}=O(\sqrt{r/n}), t+s_a=o(n^{1/3}) \big\}$. Theorem \ref{thm:3} implies the following corollary regarding the local power of the proposed test. Its proof is similar to that for Corollary \ref{cor:1} and is omitted.

\begin{cor}
\label{cor:2}
Suppose the conditions of Theorem \ref{thm:3} hold. Then, 
\begin{enumerate}[(a)]
    \item $\lim_{n\to \infty} \sup_{\betabs\in \tilde{\N}} \sup_{x>0} |\pr(T_w\leq x) - \pr(\chi^2(r,0)\leq x)| \to 0$, if $\phi_{\nu}>1/2$; 
    \item $\lim_{n\to \infty} \sup_{\betabs\in \tilde{\N}} \sup_{x>0} |\pr(T_w\leq x) - \pr(\chi^2(r,\nu)\leq x)| \to 0$, if $\phi_{\nu}=1/2$;
    \item $\liminf_{n\to \infty} \sup_{\betabs\in \tilde{\N}} \pr(T_w> x) =1$, if $\phi_{\nu}<1/2$; 
\end{enumerate}
where $\nu=\lim_{n\to\infty}\nu_n$ in (b), and (c) holds for any $x>0$.
\end{cor}

\section{Quantification of contribution of a single modality}
\label{sec:var-measure}

In addition to testing the significance of a whole data modality, it is of equal interest to quantify the amount of contribution of a modality conditioning on other data modalities in the regression model. As an example, in heritability analysis, the goal is to evaluate the contribution of genetic effects to the phenotype in addition to the environmental effects \citep{lynch1998genetics}. Motivated by the proportion of the response variance explained in the classical linear regression, we propose a measure of the contribution of a single data modality in our integrative factor regression model.

Let $\x_{-m} \in \Rcal^{p-p_m}$ denote the subvector of $\x \in \Rcal^{p}$ excluding $\x_m \in \Rcal^{p_m}$, and $\X_{-m} \in \Rcal^{n \times (p-p_m)}$ denote the submatrix of $\X \in \Rcal^{n \times p}$ excluding $\X_m \in \Rcal^{n \times p_m}$. To evaluate the contribution of $\x_m$, our key idea is to compare the goodness-of-fit of regressing $y$ on $\x$ to that of regressing $y$ on $\x_{-m}$. Toward that end, under model \eqref{eq:4}, we define
\begin{eqnarray*}
\sigma_{m|-m}^2 = \var(\x_m'\betabs_m|\x_{-m}).
\end{eqnarray*}

Next, we present a proposition regarding $\sigma_{m|-m}^2$, where statements (a) and (b) show in two different ways that $\sigma_{m|-m}^2$ can be interpreted as the improvement of the goodness-of-fit, or equivalently, additional variance of the response explained, given by the $m$th modality in addition to all other modalities. This justifies why $\sigma_{m|-m}^2$ can be used to quantify the contribution of a single modality. To simplify the presentation, we only consider the case where $p < n$. We then discuss that such an interpretation of $\sigma_{m|-m}^2$ holds for $p > n$ as well. Next, statement (c) shows that, if $\x_m$ and $\x_{-m}$ share some common factors, in that $\x_m=\Lambdab_m\f+\u_m$, and $\x_{-m}=\Lambdab_{-m}\f+\u_{-m}$, where $\f\in \Rcal^K$, we then have a closed-form expression for $\sigma_{m|-m}^2$. This expression holds true regardless of $p<n$ or $p>n$, and thus provides a unified way of computing $\sigma_{m|-m}^2$ in practice. 

\begin{pro} \label{prop:1} 
Suppose $\x$ follows a multivariate normal distribution and $p < n$. Let $\hat{\Y}$ and $\hat{\Y}_{-m}$ denote the predicted response by regressing $y$ on $\x$, and regressing $y$ on $\x_{-m}$, respectively, via least squares. Let $\sigma_y^2 = \var(y)$. Then the following results hold:
\begin{enumerate}[(a)]
  \item $\sigma_{m|-m}^2 = \E\ltwonorm{\Y-\hat{\Y}_{-m}}^2 / (n-p_{-m}) - \sigmae^2$;

  \item $\sigma_{m|-m}^2 = \sigma_y^2 (r^2 - r^2_{-m})$, where $r^2 = 1 - \E \ltwonorm{\Y-\hat{\Y}}^2 / \{ (n - p) \sigma_y^2 \}$, and $r^2_{-m} = 1 - \E \ltwonorm{\Y-\hat{\Y}_{-m}}^2 / \{ (n - p_{-m}) \sigma_y^2 \}$;  
    
  \item $\sigma_{m|-m}^2 = {\betabs_m}' \left\{ \Lambdab_m(\I_{K}+\Lambdab_{-m}'\Sigmaunom^{-1}\Lambdab_{-m})^{-1}\Lambdab_m'+\Sigmab_{u_m} \right\} \betabs_m$. 
\end{enumerate}
\end{pro}

By Proposition \ref{prop:1}(a), when regressing $y$ using all but the $m$th modality, we have $\E\ltwonorm{\Y-\hat{\Y}_{-m}}^2 = (n-p_{-m}) (\sigmae^2+\sigma^2_{m|-m})$. On the other hand, when regressing $y$ on all data modalities, we have $\E \ltwonorm{\Y-\hat{\Y}}^2=(n-p)\sigma_{\epsilon}^2$. Therefore, from a goodness-of-fit perspective, ignoring $\x_m$ leads to a ``worsened" prediction by an amount of $\sigma_{m|-m}^2$. 

For Proposition \ref{prop:1}(b), recall in the classical linear regression model, the adjusted $R^2$ measures the percentage of total variation in the response that has been explained by the predictors, and is defined as $R^2 = 1- \{RSS/(n-p)\} / \{TSS/(n-1)\}$, 
where $RSS$ and $TSS$ are the residual sum of squares and total sum of squares, respectively. Then, $r^2$ in Proposition \ref{prop:1}(b) can be viewed as an ``expected'' percentage of total variation in the response explained, in that, 
\begin{equation*}
r^2 = 1 - \frac{\E(RSS)/(n-p)}{\E(TSS)/(n-1)} = 1 - \frac{\E \ltwonorm{\Y-\hat{\Y}}^2}{(n-p) \sigma_y^2}.
\end{equation*}
As we show in the proof of Proposition \ref{prop:1}, when using all but the $m$th modality, the ``expected" percentage of total variation in the response explained is $r^2_{-m} = 1-(\sigma_{\epsilon}^2+\sigma_{m|-m}^2)/\sigma_y^2$, where $\sigma_{\epsilon}^2 = \E \ltwonorm{\Y-\hat{\Y}}^2 / (n-p)$. On the other hand, when using all data modalities, the ``expected" percentage of total variation in the response explained is $r^2 = 1-\sigma_{\epsilon}^2/\sigma_y^2$. Therefore, using the $m$th modality improves the ``expected'' percentage of total variation in the response explained by an amount of $\sigma_{m|-m}^2/\sigma_y^2$. 

We have so far justified $\sigma^2_{m|-m}$ in the setting where $p < n$. In the setting where $p > n$ and the true model is sparse, in that there are only $s$ variables associate with $y$ with $s<n$, we can still use $\sigma_{m|-m}^2$ to quantify the contribution of an individual modality. This is because Proposition \ref{prop:1}(a) and (b) continue to hold if we replace $\hat{\Y}$ and $\hat{\Y}_{-m}$ with $\tilde{\Y}$ and $\tilde{\Y}_{-m}$, and replace $(n-p)$ with $(n - s)$, where $\tilde{\Y}$ denotes the predicted response by regressing $y$ on the $s$ true variables via least squares, and $\tilde{\Y}_{-m}$ is defined similarly but excluding the variables in the $m$th modality. In practice, of course, which subset are the true variables is unknown. However, Proposition \ref{prop:1}(a) and (b) only provide conceptual justifications of $\sigma_{m|-m}^2$. We always resort to Proposition \ref{prop:1}(c) to compute $\sigma_{m|-m}^2$, which holds regardless of $p < n$ or $p>n$. Besides, it does not require the knowledge of the true variables, nor any extra variable selection step to identify them.

Next, we present a plug-in estimator of $\sigma_{m|-m}^2$ given the data. We first apply \citet{Bai2002} in (\ref{eq:101}) to the concatenated data matrix $\X=(\X_1,\ldots,\X_M)\in \Rcal^{n\times p}$ to estimate the number of shared factors. We then apply PCA to obtain $\Fh$. We next estimate $\Lambdab_m$ by $\Lambdabh_m=(1/n)\X_m'\Fh$, and obtain $\Uh_m=\X_m-\Fh\Lambdabh_m'$. We solve for $\betabh$ following  (\ref{eq:7}). We then apply the thresholding method of \cite{Fan2013b} to estimate $\Sigmab_u$ by $\Sigmabh_u$, whose $(i,j)$th element is $\sigmah^2_{u,ij}=s(n^{-1}\sum_{\ell=1}^n \hat{U}_{\ell i}\hat{U}_{\ell j}, \omega)$, $s(x,\omega)$ is a thresholding function, $\omega$ is the threshold, and $\hat{\U} = (\Uh_1,\ldots,\Uh_M)$. Let $\betabh_m$ denote the subvector of $\betabh$ with indices in the $m$th modality. By Theorem \ref{thm:2}, $\betabh_m$ is a consistent estimator of $\betabs_m$. By Theorem 3.3 of \cite{Fan2013b}, $\Lambdabh_m$ and $\Lambdabh_{-m}$ are two consistent estimators of the loading matrices. In addition, by Theorem 3.1 of \cite{Fan2013b}, $\Sigmabh_u$ is a consistent estimator of $\Sigmab_u$. Plugging all these estimators into Proposition \ref{prop:1}(c) gives a consistent estimator of $\sigma_{m|-m}^2$.

We make two additional remarks about $\sigma_{m|-m}^2$. First, the closed-form expression of $\sigma_{m|-m}^2$ utilizes the factors commonly shared by $\x_m$ and $\x_{-m}$. Indeed, such factors determine the correlations between $\x_m$ and $\x_{-m}$. When no such common factors exist, $\x_m$ and $\x_{-m}$ are uncorrelated. In that case, $\var(\x_m'\betabs_m|\x_{-m}) = \var(\x_m'\betabs_m)={\betabs_m}'\Sigmab_{x_m}\betabs_m={\betabs_m}'\Sigmab_{u_m}\betabs_m$. Therefore, the closed-form expression in Proposition \ref{prop:1}(c) can be viewed as a more general form of this special case by taking the correlations between $\x_m$ and $\x_{-m}$ into account. Second, the computation of $\sigma_{m|-m}^2$ only requires to invert a sparse high-dimensional matrix $\Sigmab_{u_{-m}}$ and a low-dimensional matrix $\I_K+\Lambdab_{-m}'\Sigmaunom^{-1}\Lambdab_{-m}$. If an exact factor model is further adopted such that $\Sigmab_u$ becomes a diagonal matrix, $\sigma_{m|-m}^2$ can be easily computed, as one only needs to invert a low-dimensional matrix. On the contrary, if one does not employ a factor model, then $\var(\x_m'\betabs_m|\x_{-m})={\betabs_m}'(\Sigmab_{x_m}-\Sigmab_{x_m,x_{-m}}\Sigmaxnom^{-1}\Sigmab_{x_{-m},x_m})\betabs_m$, where $\Sigmab_{x_m}=\E(\x_m^{\otimes 2})$, $\Sigmab_{x_m,x_{-m}}=\E(\x_m\x'_{-m})$, $\Sigmab_{x_{-m}}=\E(\x_{-m}^{\otimes2})$, and $\Sigmab_{x_{-m},x_m}=\E(\x_{-m}\x_m')$. Consequently, a large dense matrix $\Sigmab_{x_{-m}}$ has to be inverted.

\section{Numerical analysis}
\label{sec:numerical}

\subsection{Test of a whole modality}
\label{sec:sim-test-whole}

We evaluate the empirical performance of the factor-adjusted score test of a whole modality proposed in Section \ref{sec:test-whole}. We generate $M=3$ modalities, and consider two cases of $\x$. Specifically, for each modality $m = 1, 2, 3$, $\x_m$ are $n$ i.i.d.\ random samples generated from $N_{p_m}(0, \Sigmab_m)$. For Case 1, $\Sigmab_m = \Lambdab_m\Lambdab_m'+0.5\I_{p_m}$, where each column of $\Lambdab_m \in \Rcal^{p_m\times K_m}$ is generated from $N_{p_m}(0, 2)$, and the number of factors $K_m = 2$. For Case 2, the diagonal elements of $\Sigmab_m$ equal 1 and the off-diagonal elements equal 0.4. Accordingly, in Case 1, $\x_m$ indeed follows a factor model setup, and in Case 2, although $\x_m$ does not strictly follow a factor model, its covariance matrix has spiked eigenvalues. In both cases, we aim to test if the first modality $\x_1$ is significantly associated with the response, i.e. $H_0: \betas_{11}=\ldots=\betas_{1p_1} = 0$.  We then consider two types of alternatives. The first alternative is $H_{A1}: \betas_{11} = \ldots = \betas_{1p_1} = \delta/p$, where $\delta$ is a sequence approaching zero. As such, there is a weak signal in each variable of $\x_1$ and the overall signal is dense. The second alternative is $H_{A2}: \betas_{11} = \ldots = \betas_{15} = \delta/5, \betas_{16}= \ldots = \betas_{1p_1} = 0$. As such, the overall signal is sparse, as all signals come only from the first 5 variables, whereas the rest do not associate with the response. For the other two modalities $\x_2$ and $\x_3$, we set $\betas_{21}=1, \betas_{22}=2, \betas_{23}=\ldots=\betas_{2p_2}=0$, and $\betas_{31}=-1, \betas_{32}=-1, \betas_{33}=\ldots=\betas_{3p_3}=0$. We generate the error $\epsilon$ as $n$ i.i.d.\ samples from $N(0,0.5)$, and generate $y$ based on model (\ref{eq:4}). We set $p_1 = p_2 = p_3=p/3$. We consider two combinations $(n, p) = (100, 600)$, and $(200, 900)$. We compare our test with the score test of \citet{ning2017general}, where the critical values are obtained by bootstrap.

We report the proportion of rejections of $H_0$ by both tests out of 600 data replications as we vary the value of $\delta$. When $\delta = 0$, this gives the empirical size, and when $\delta > 0$, it gives the empirical power of the two tests. Figures \ref{fig:1} and \ref{fig:2} report the results for Cases 1 and 2, respectively. In both cases, we see that our  proposed test controls the Type I error at the nominal level of $\alpha=0.05$ when $\delta = 0$. However, the test of \citet{ning2017general} often yields an inflated size. This may be due to that their multiplier bootstrap method rejects the null hypothesis if the maximum of the decorrelated score functions of variables in that modality is greater than a threshold, and as such, it is easier to reject the null hypothesis. Moreover, our test achieves an as good or often a better power than the test of \citet{ning2017general} as $\delta$ increases.

\begin{figure}[t!]
\centering
\begin{tabular}{cc}
\includegraphics[width=0.4\textwidth]{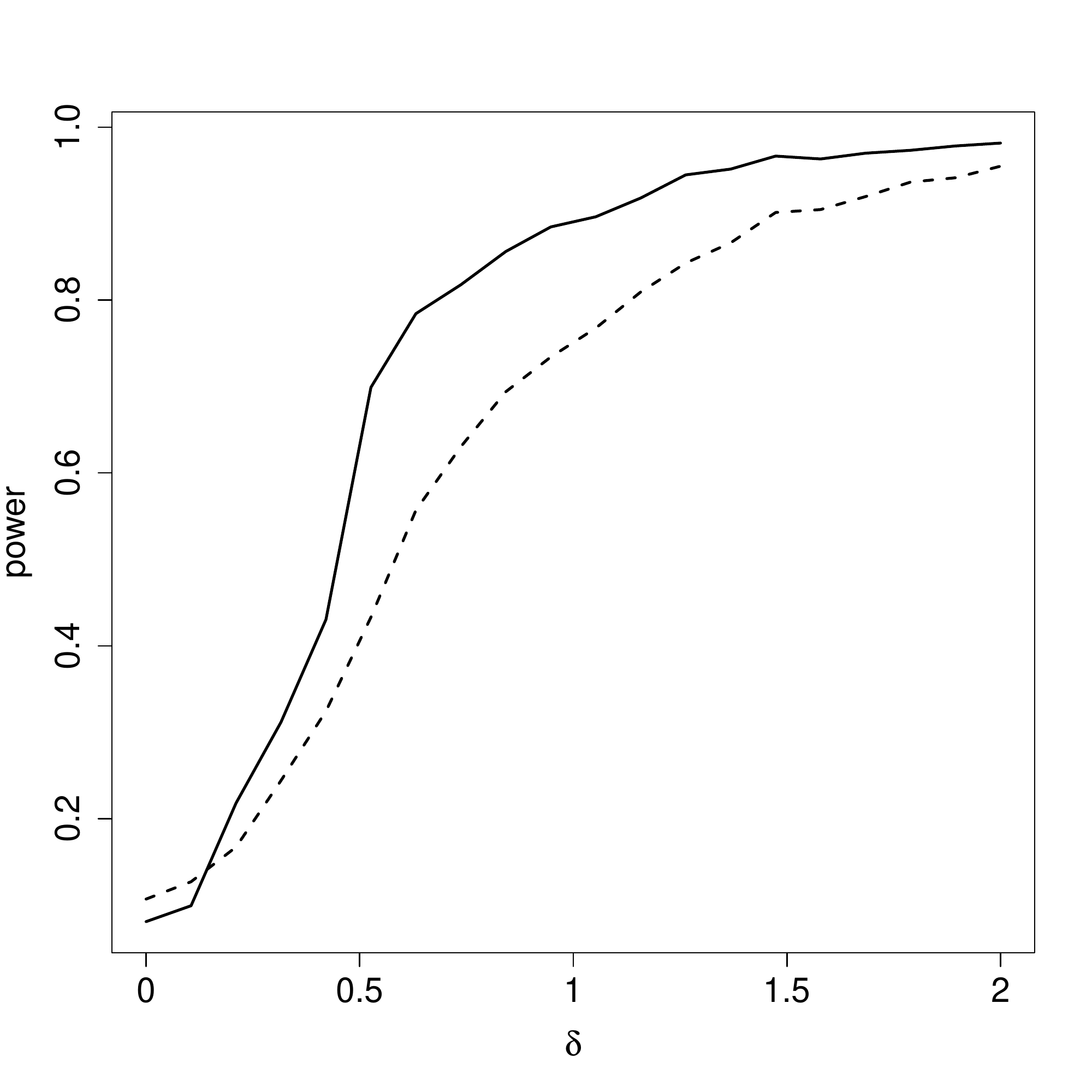} & 
\includegraphics[width=0.4\textwidth]{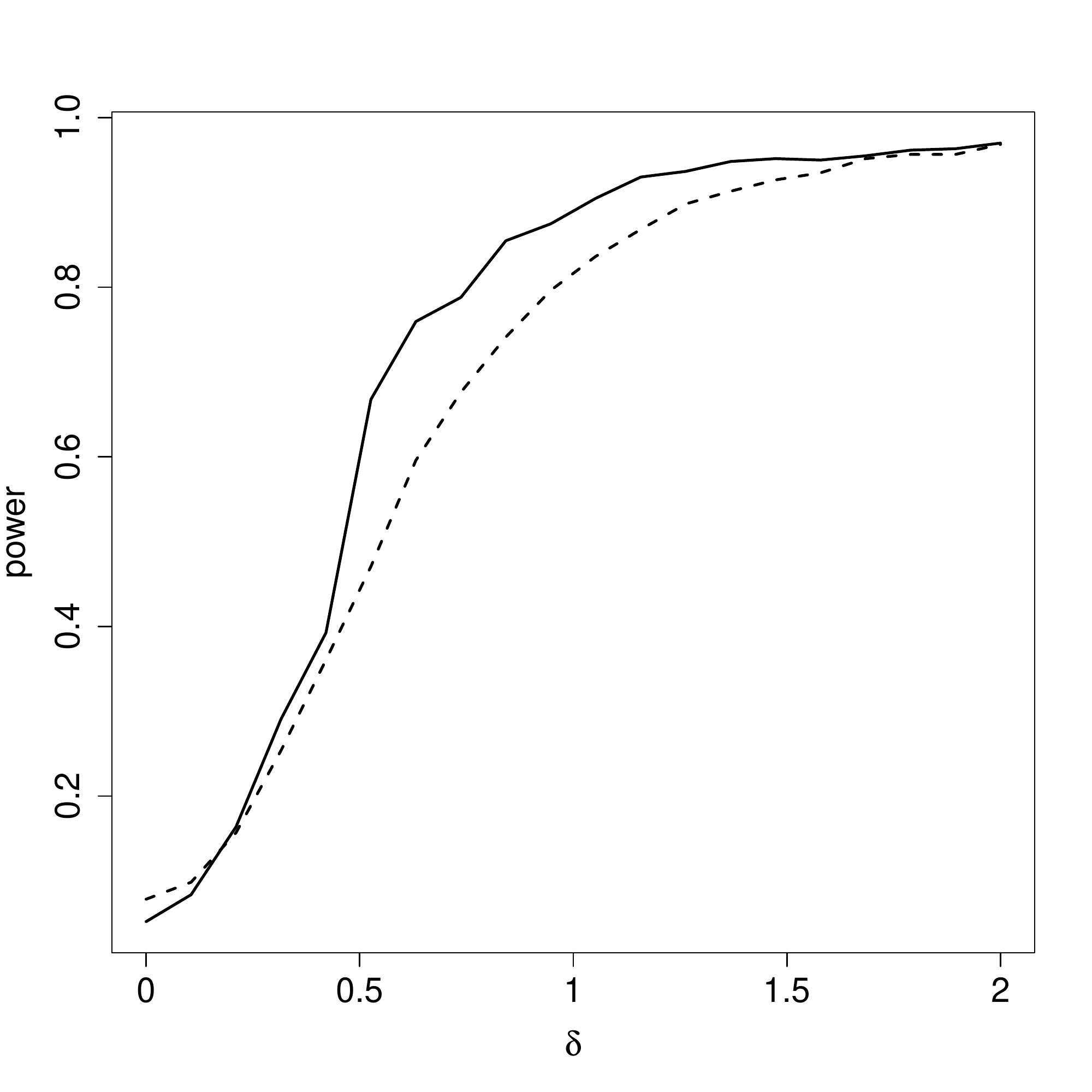} \\
$H_{A1}$: $(n,p)=(100,600)$  &  $H_{A1}$: $(n,p)=(200,900)$  \\
\includegraphics[width=0.4\textwidth]{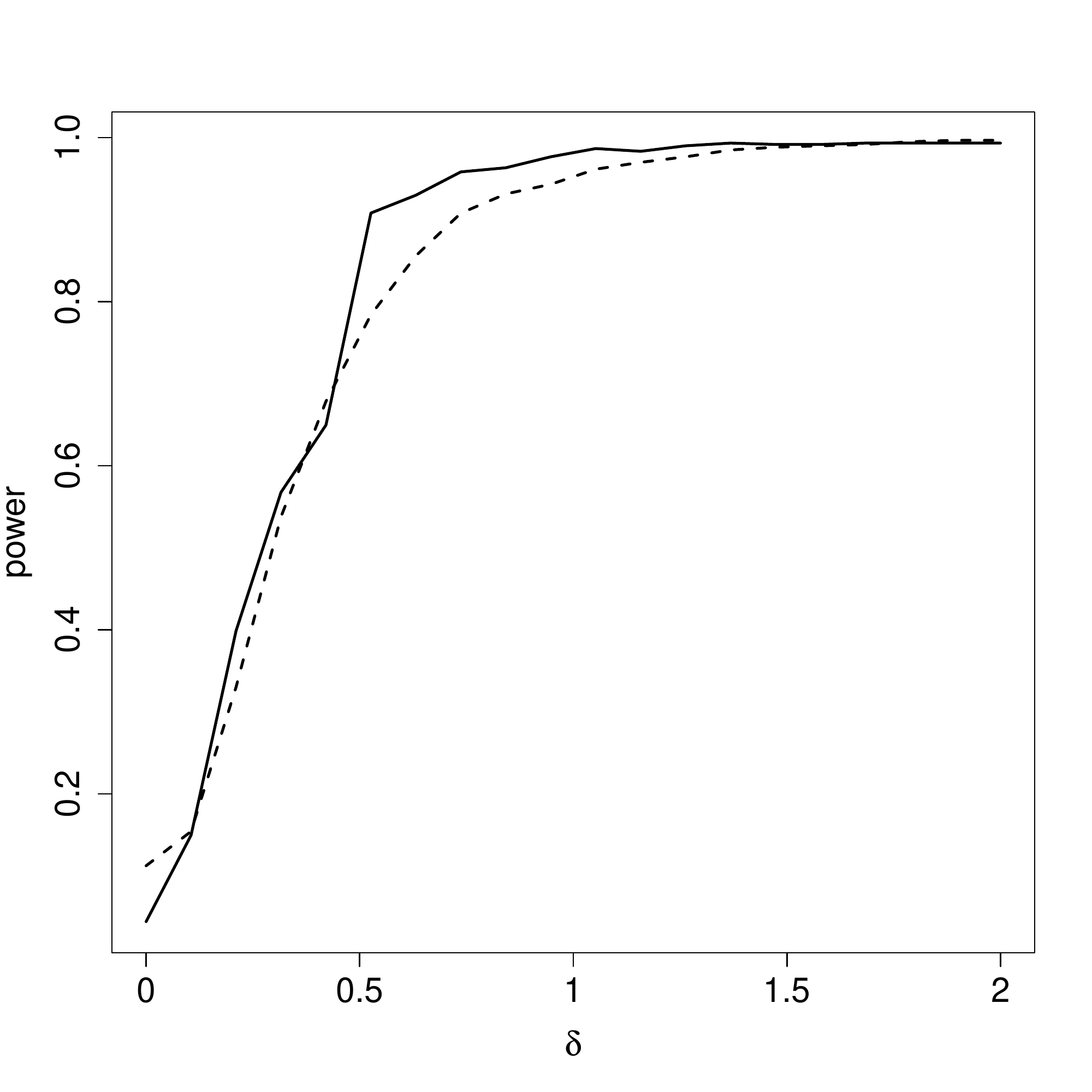} & 
\includegraphics[width=0.4\textwidth]{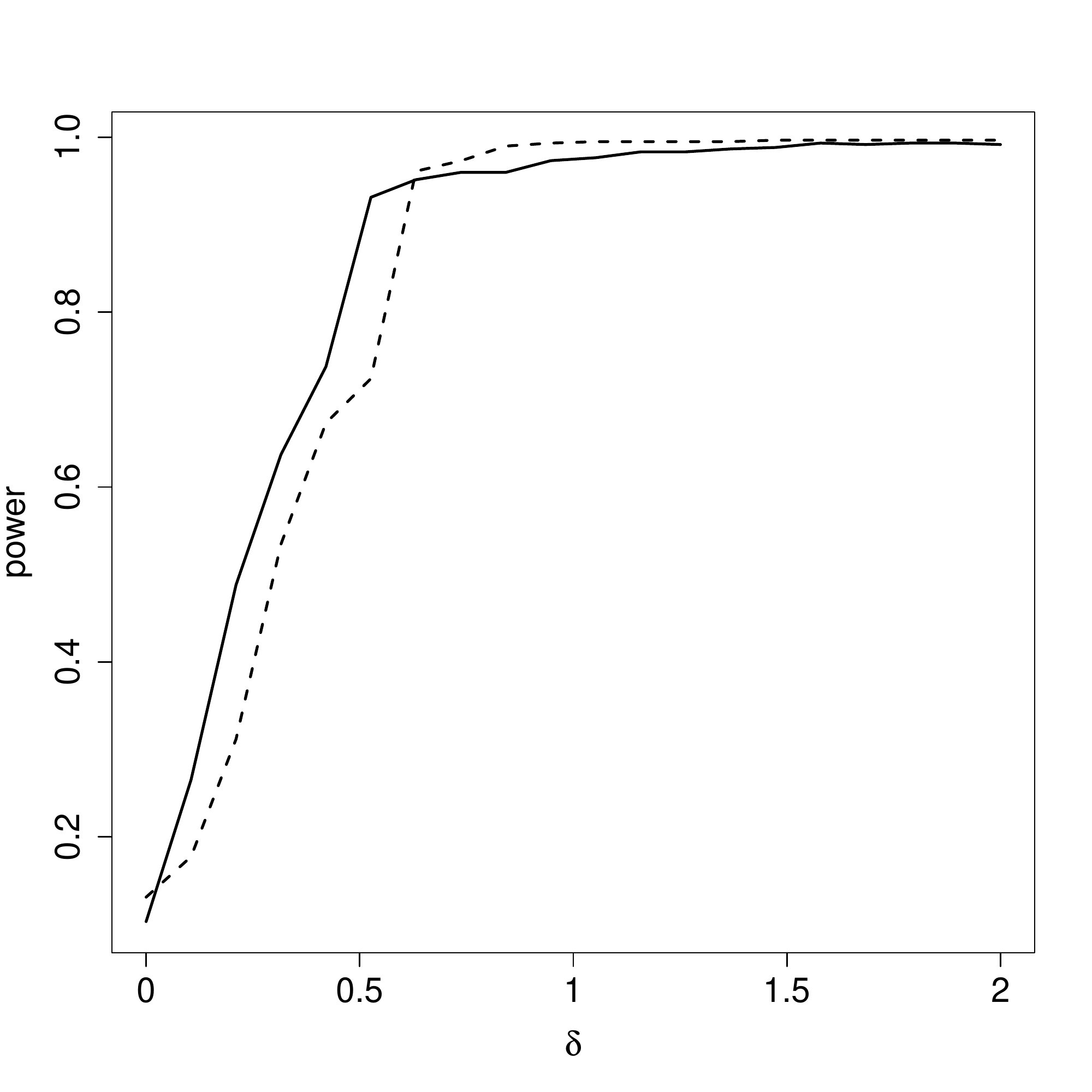} \\
$H_{A2}$: $(n,p)=(100,600)$  &  $H_{A2}$: $(n,p)=(200,900)$  \\
\end{tabular}
\caption{Empirical size and power of testing a whole modality in Case 1 for the factor-adjusted score test (solid line), and the score test of \citet{ning2017general} (dashed line).}
\label{fig:1}
\end{figure}

\begin{figure}[t!]
\centering
\begin{tabular}{cc}
\includegraphics[width=0.4\textwidth]{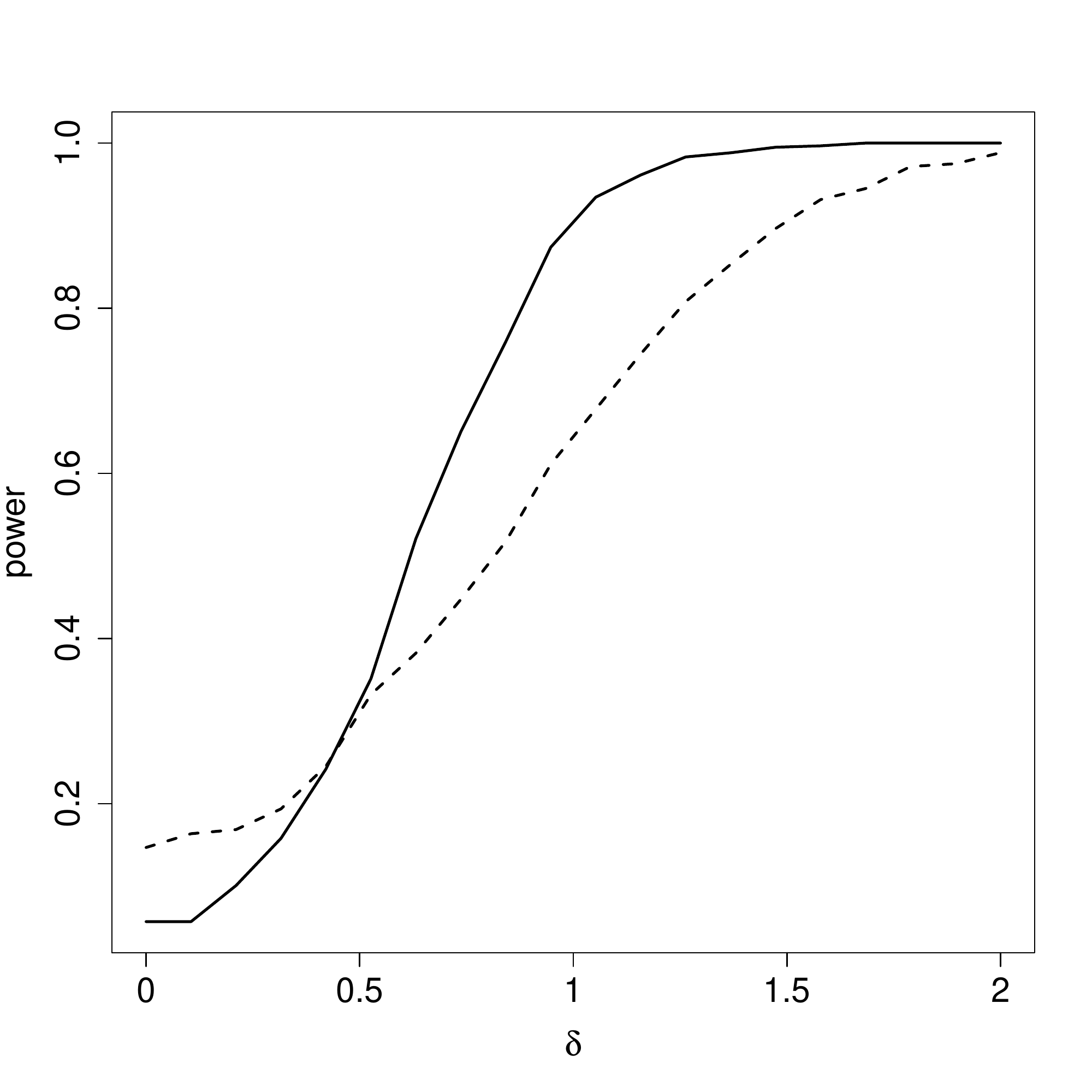} & 
\includegraphics[width=0.4\textwidth]{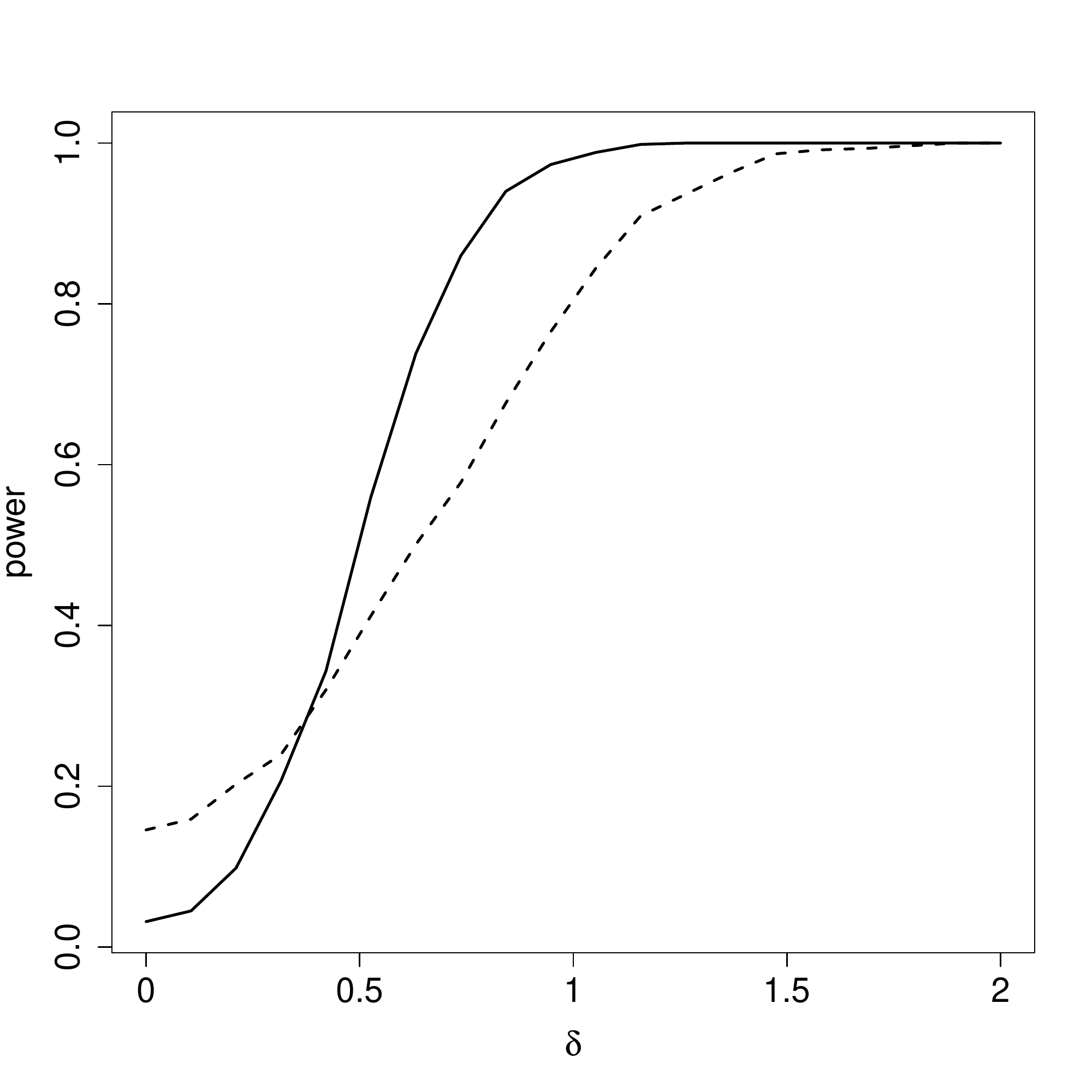} \\
$H_{A1}$: $(n,p)=(100,600)$  &  $H_{A1}$: $(n,p)=(200,900)$  \\
\includegraphics[width=0.4\textwidth]{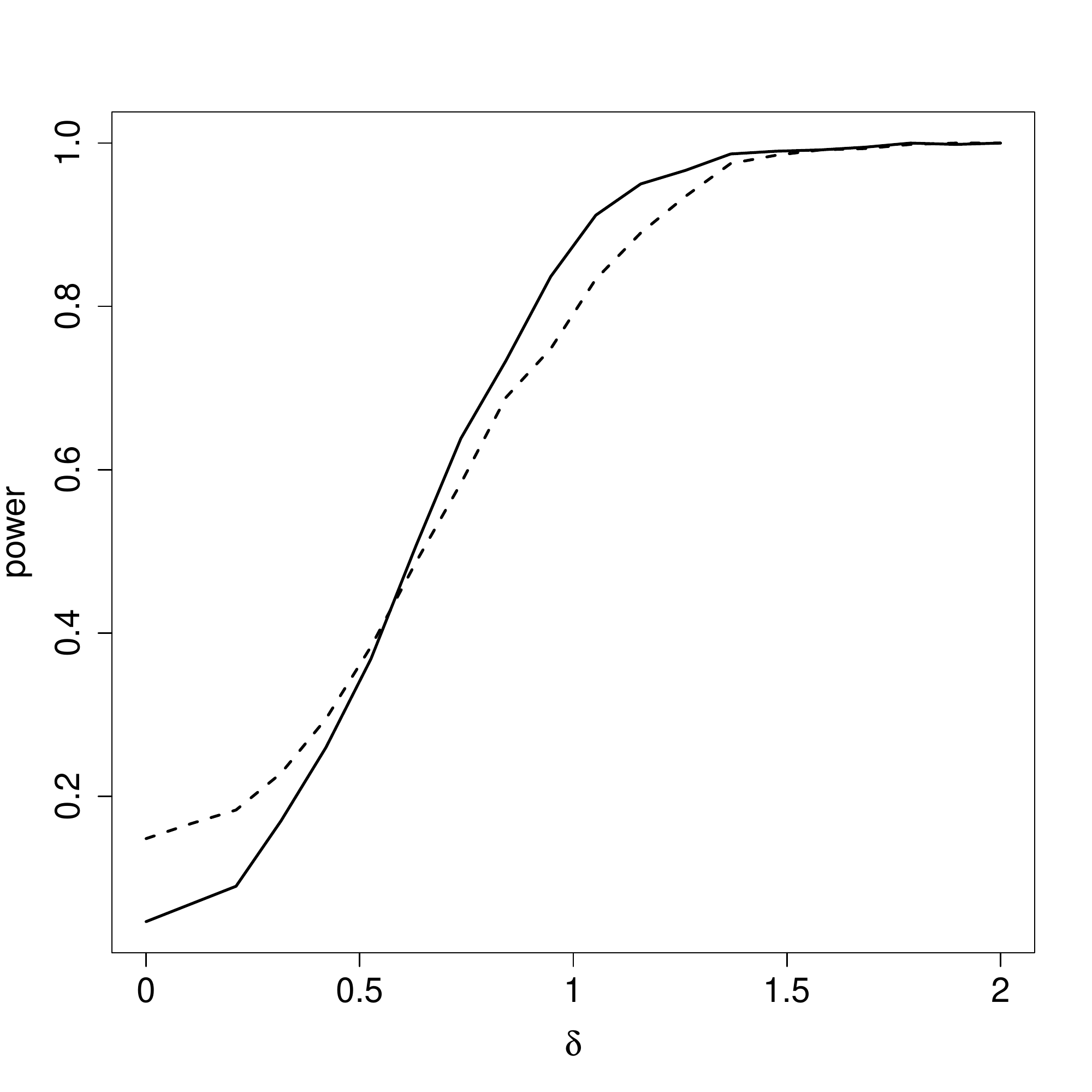} & 
\includegraphics[width=0.4\textwidth]{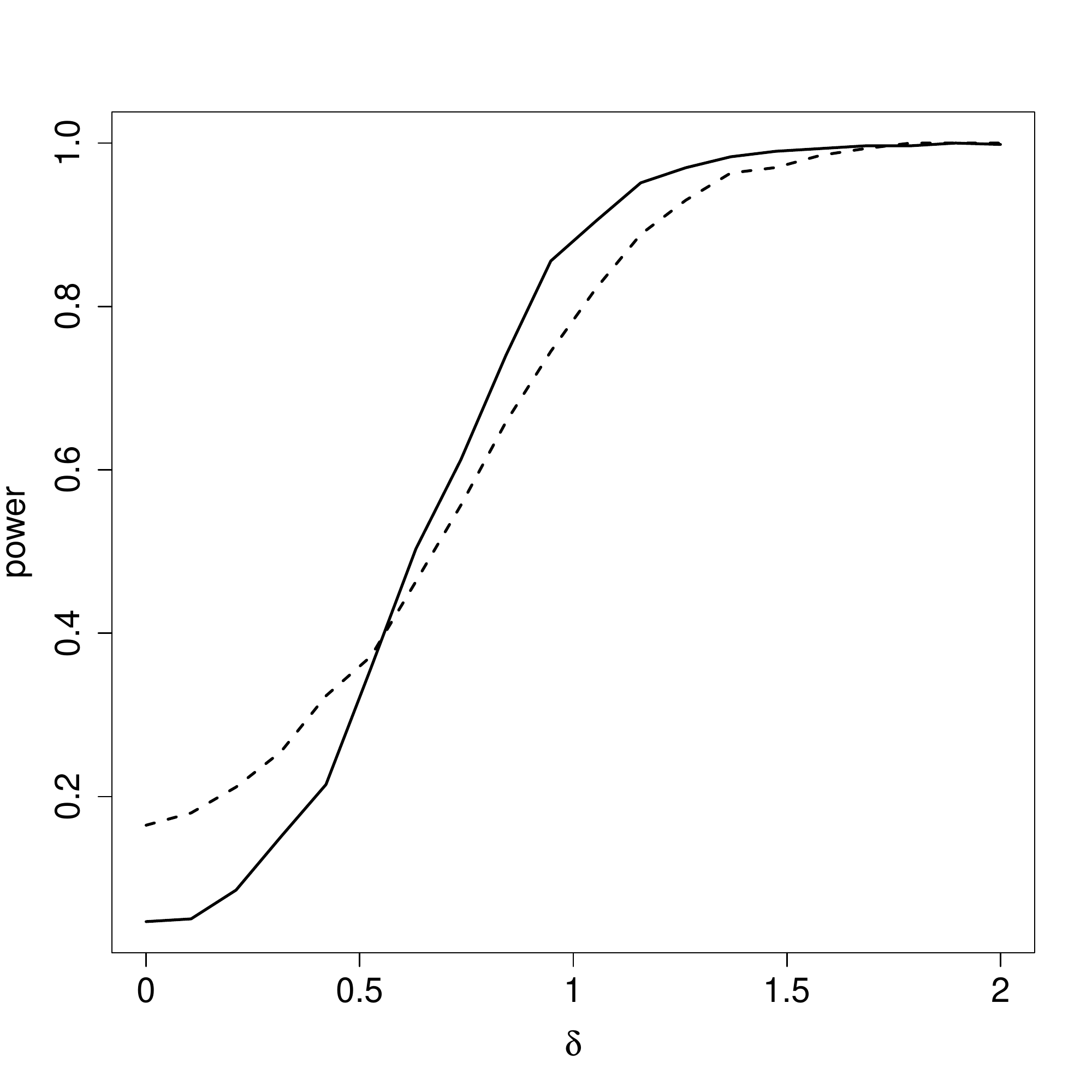} \\
$H_{A2}$: $(n,p)=(100,600)$  &  $H_{A2}$: $(n,p)=(200,900)$  \\
\end{tabular}
\caption{Empirical size and power of testing a whole modality in Case 2 for the factor-adjusted score test (solid line), and the score test of \citet{ning2017general} (dashed line).}
\label{fig:2}
\end{figure}

Next, we simulate data from the three examples as we discussed in Section \ref{sec:test-whole} to further examine the performance of the proposed test. For all three examples, we generate $M=2$ modalities with $K_m=1$ factor in each modality, and set $n=100$, $p_m=200$, and $\alpha=0.05$. We aim to test the significance of the first modality. We generate $\x_m$ as $n$ i.i.d.\ random samples from $N_{p_m}(0,\Sigmab_m)$, where $\Sigmab_m = \Lambdab_m\Lambdab_m'+0.5\I_{p_m}$ for $m=1,2$. For the second modality, we always choose $\Lambdab_2=(1,1,1,\ldots,1)'$, and set its coefficients as $\betas_{21}=1,\betas_{22}=2,\betas_{23}=\ldots=\betas_{2p_2}=0$.  For the first modality, we choose different loadings and test different local alternatives. For Example \ref{exa:1}, we choose $\Lambdab_1=(20,1,1,\ldots,1)'$, and the alternative $H_A:\betas_{11}=0.08, \betas_{12}=\ldots=\betas_{1p_1}=0$. For Example \ref{exa:2}, we choose $\Lambdab_1=(1,1,1,\ldots,1)'$, and the alternative $H_A:\betas_{11}=\ldots=\betas_{18}=0.01, \betas_{19}=\ldots=\betas_{1p_1}=0$. For Example \ref{exa:3}, we choose $\Lambdab_1=(0,1,1,\ldots,1)'$, and the alternative $H_A:\betas_{11}=0.2, \betas_{12} = \ldots=\betas_{1p_1}=0$. Table \ref{tab:2} reports the empirical size and power of the proposed factor-adjusted score test and the score test of \citet{ning2017general} based on 600 data replications. For Example \ref{exa:1}, thanks to the large loading of the first variable, the proposed test achieves a better power than the test of \citet{ning2017general}. For Example \ref{exa:2}, the nonzero coefficients are spread out in eight covariates, and they all have loadings on the latent factor. Therefore, thanks to the cumulative effects, our method again achieves a better power to detect the alternative. For Example \ref{exa:3}, while the nonzero coefficient $\betas_{11}$ is large, the corresponding loading is zero. Our test thus has no power in detecting such an alternative. These numerical results agree with our discussion in Section \ref{sec:test-whole} on the local power of the proposed test.

\begin{table}[t!]
\centering
\begin{tabular}{c|ccccc} \hline
    & \multicolumn{2}{c}{Factor-adjusted test} & \hbox{   } & \multicolumn{2}{c}{Test of Ning and Liu}\\
    & Size & Power & & Size & Power \\ \hline
    Example \ref{exa:1} & 0.06 & 0.50 & & 0.05 & 0.40 \\
    Example \ref{exa:2} & 0.06 & 0.60 & & 0.07 & 0.45\\
    Example \ref{exa:3} & 0.06 & 0.07 & & 0.05 & 0.35 \\ \hline
\end{tabular}
\caption{Empirical size and power of testing a whole modality in Examples \ref{exa:1} to \ref{exa:3} for the factor-adjusted score test, and the decorrelated score test of \citet{ning2017general}.}
\label{tab:2}
\end{table}

\subsection{Test of a linear combination of predictors}
\label{sec:sim-test-lincomb}

We next evaluate the empirical performance of the factor-adjusted Wald test of a linear combination of predictors proposed in Section \ref{sec:test-lincomb}. We again generate $M=3$ modalities and two cases of $\x$ similarly as in Section \ref{sec:sim-test-whole}, except that in the second case we increase the off-diagonal elements of $\Sigmab_m$ to 0.8 for $m=1,2,3$. We set $\betabs=({\betabs_1}',{\betabs_2}',{\betabs_3}')'$, $\betabs_1=(-2+\delta, -1,0,\ldots,0)'$, $\betabs_2=(1+\delta,2,0,\ldots,0)'$, $\betabs_3=(1+\delta,1,0,\ldots,0)'$, and aim to test the linear combination of the first variable in each modality that $H_0: \betas_{11} + \betas_{21} + \betas_{31} = 0$ versus $H_A: \betas_{11} + \betas_{21} + \betas_{31} \neq 0$. The rest of the simulation setup is the same as that in Section \ref{sec:sim-test-whole}. Since we only consider testing a single linear combination in this study, we compare our test with both the partially penalized Wald test proposed in \citet{shi2019linear}, and the test proposed in \cite{ZhuBradic2018}.

\begin{figure}[t!]
\centering
\begin{tabular}{cc}
\includegraphics[width=0.4\textwidth]{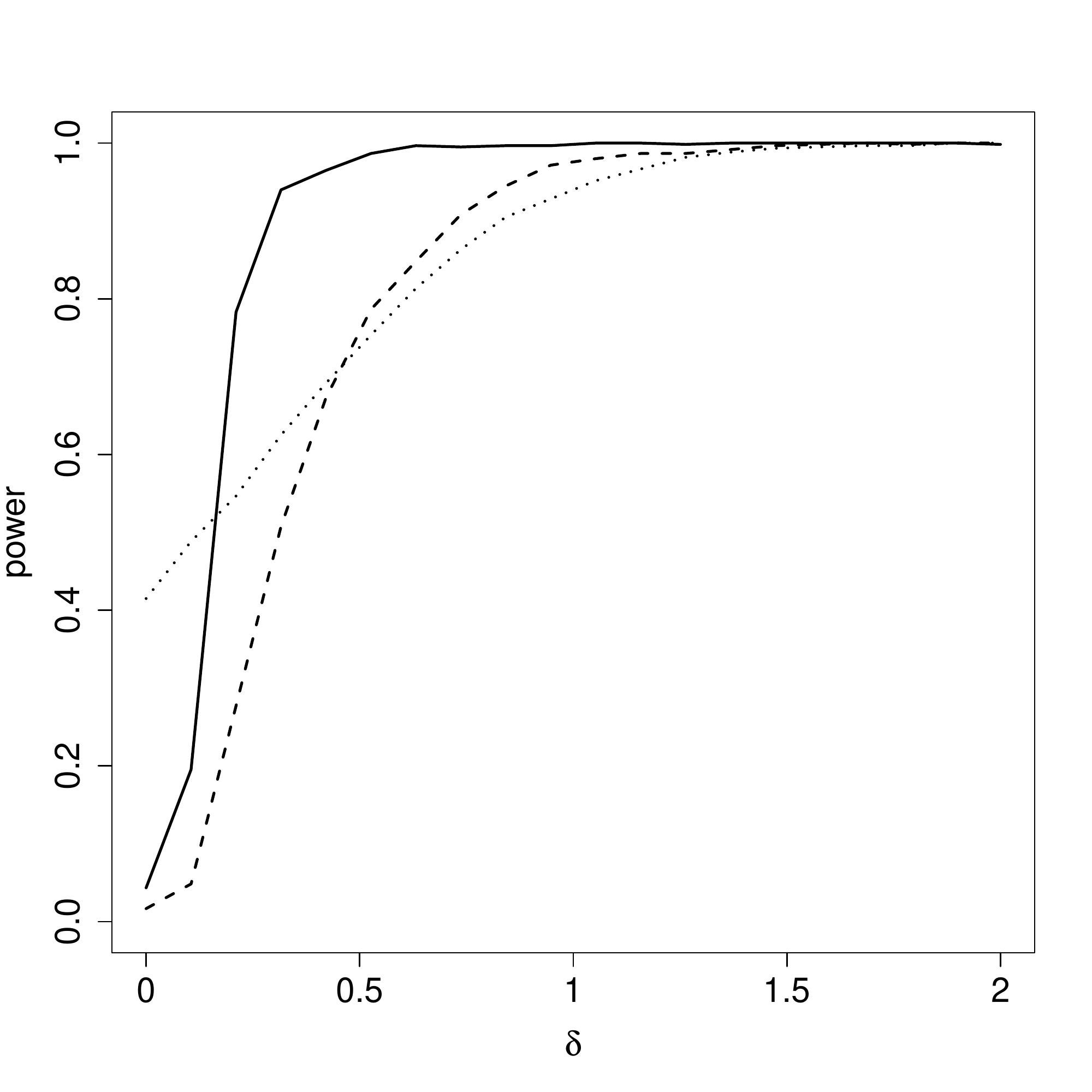} & 
\includegraphics[width=0.4\textwidth]{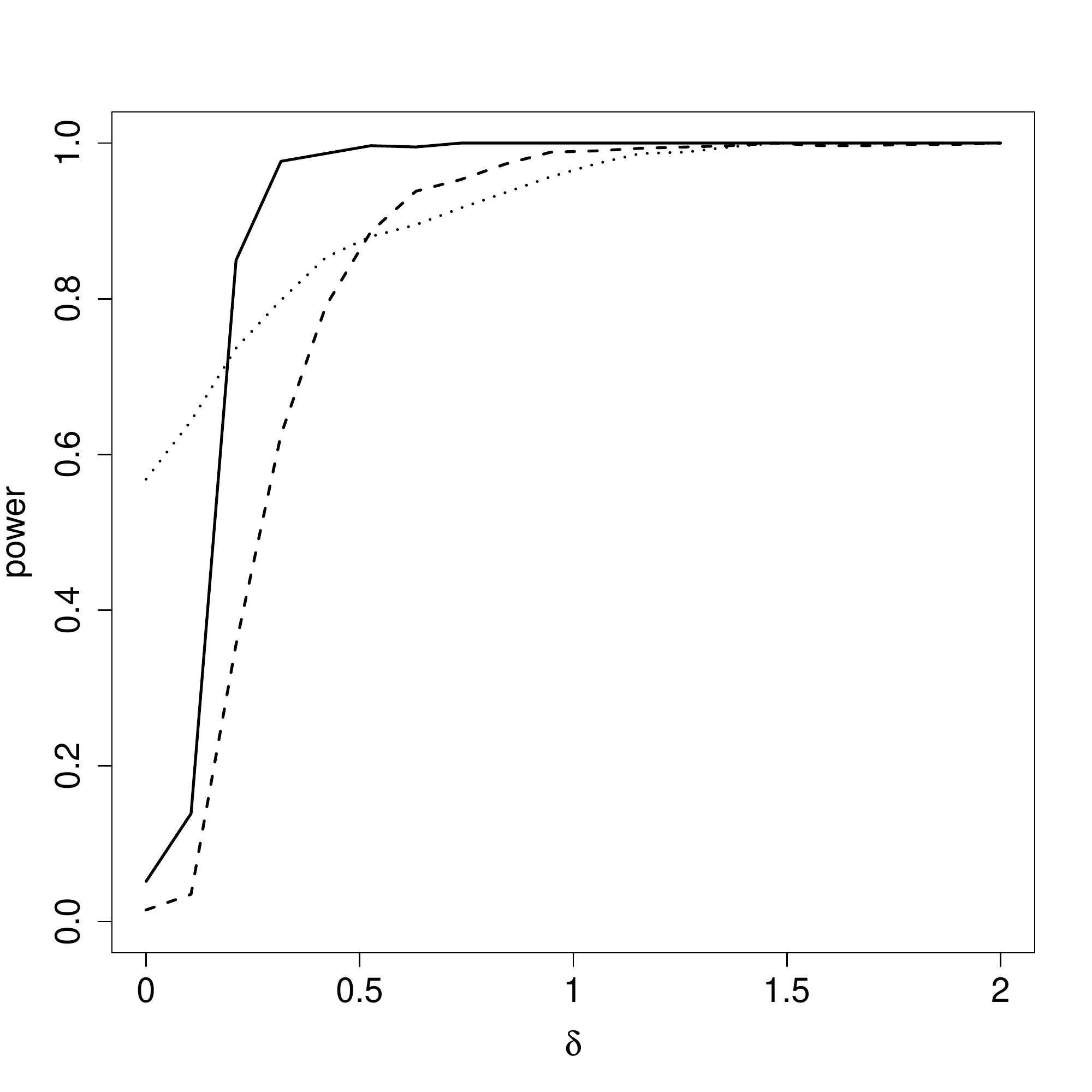} \\
Case 1: $(n,p)=(100,600)$  &  Case 1: $(n,p)=(200,900)$  \\
\includegraphics[width=0.4\textwidth]{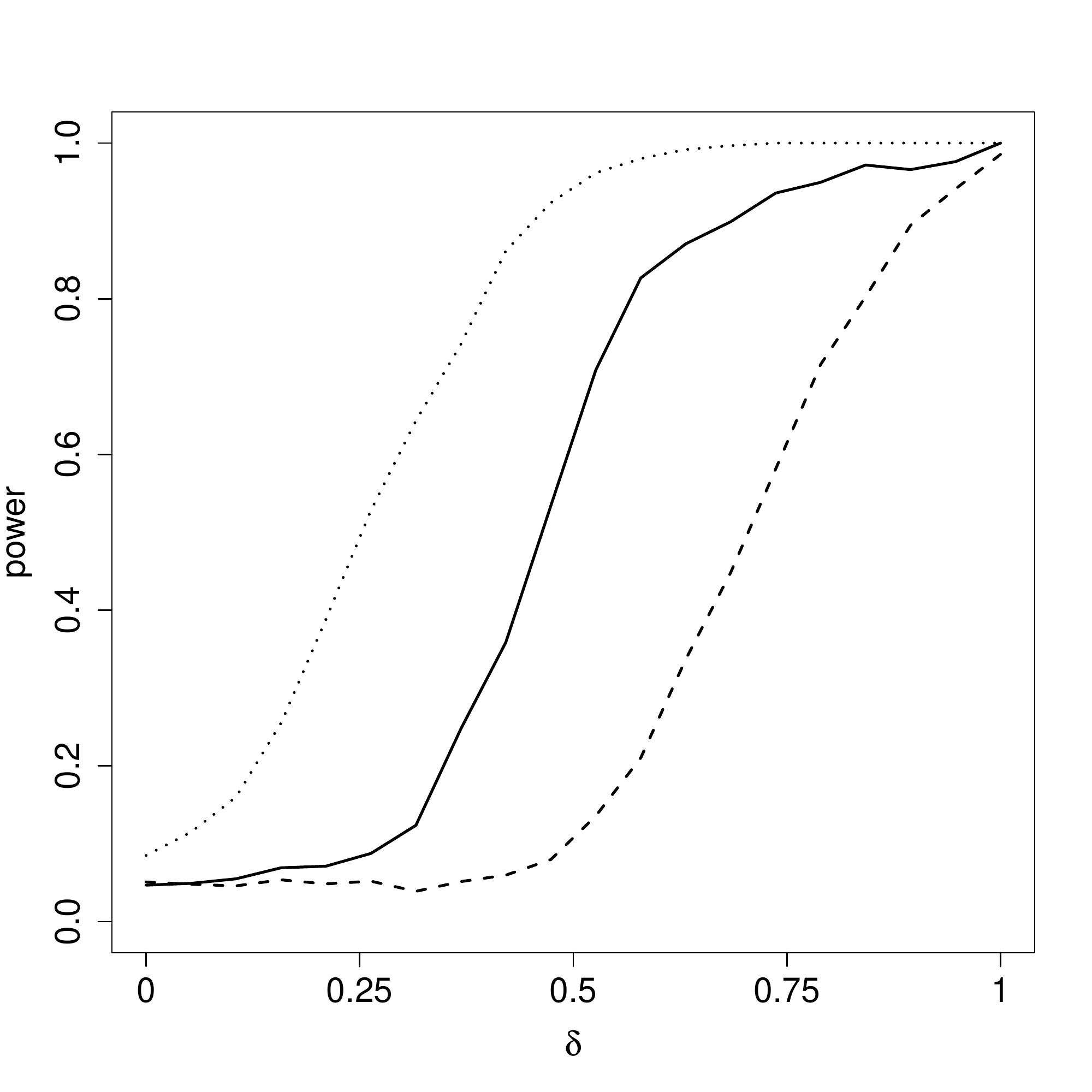} & 
\includegraphics[width=0.4\textwidth]{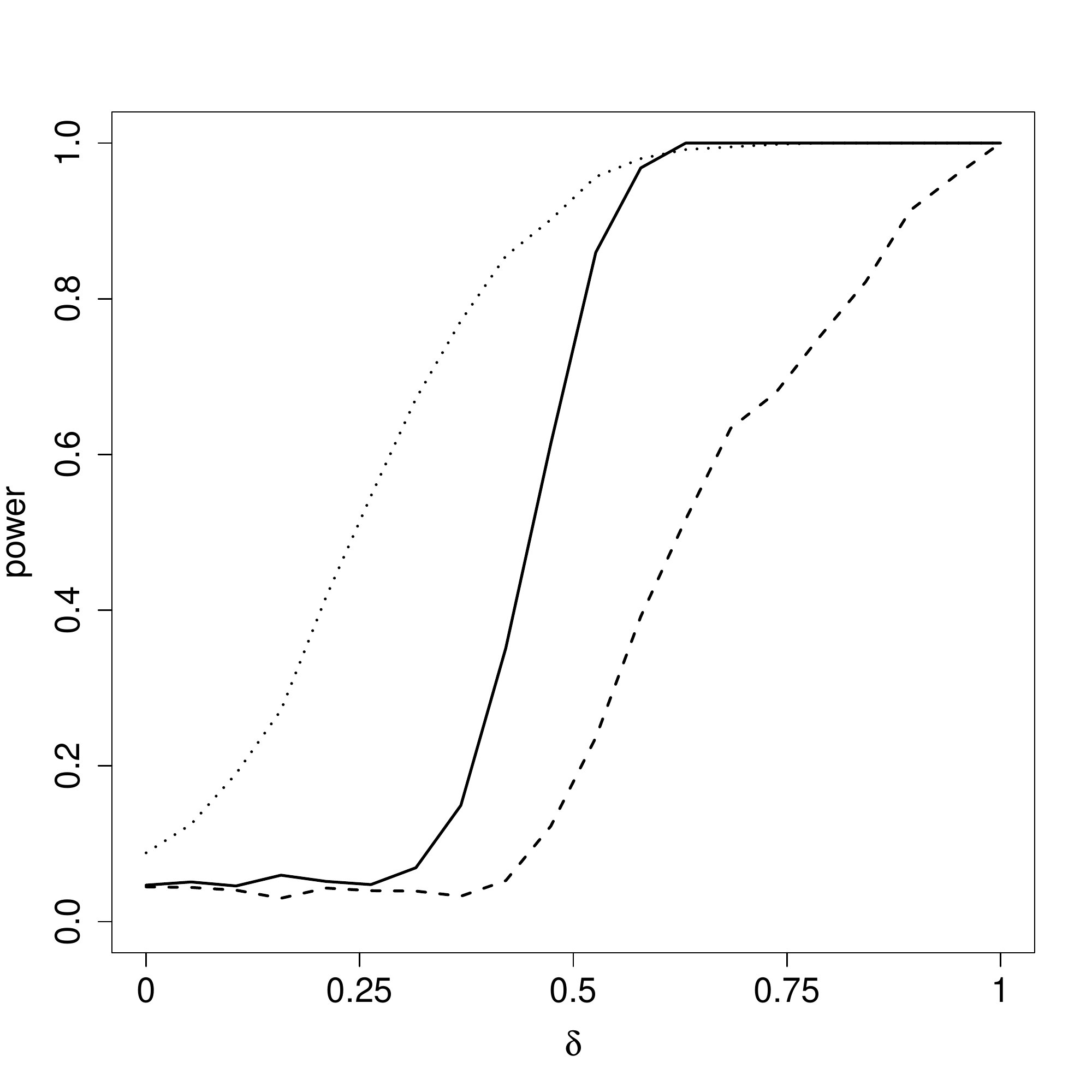} \\
Case 2: $(n,p)=(100,600)$  &  Case 2: $(n,p)=(200,900)$  \\
\end{tabular}
\caption{Empirical size and power of testing a linear combination of predictors for the factor-adjusted Wald test (solid line), the Wald test of \citet{shi2019linear} (dashed line), and the test of \cite{ZhuBradic2018} (dotted line).}
  \label{fig:3}
\end{figure}

We again report the proportion of rejections of $H_0$ out of 600 data replications as we vary the value of $\delta$. Figure \ref{fig:3} reports the results for both Cases 1 and 2. In both cases, we see that our factor-adjusted Wald test and the test of \citet{shi2019linear} achieve a good control of the type I error at the nominal level $\alpha=0.05$ when $\delta = 0$. But our test achieves a much improved power as $\delta$ increases. The main reason is that, in this example, the variables are highly correlated with each other. The factor adjustment alleviates such high correlations, and yields a better variable selection and estimation of the true regression coefficients, which in turn benefits the inference. Meanwhile, in Case 1, the test of \cite{ZhuBradic2018} yields a type I error that is much larger than the nominal level. This is because, in this case, the variables are driven by the latent factors, which leads to a wide spectrum of the eigenvalues of the covariate covariance matrix. However, \cite{ZhuBradic2018} required the eigenvalues to be bounded. In Case 2, the variables are not generated from a factor model. In this case, the test of \cite{ZhuBradic2018} enjoys the best power. Nevertheless, it still suffers from an inflated type I error, which is around 0.09 and is about twice as large as the nominal level. Moreover, we recall that both our test and the test of \citet{shi2019linear} can jointly test multiple linear combinations, while the test of \cite{ZhuBradic2018} was designed to test a single linear combination.

\subsection{Multimodal neuroimaging analysis}
\label{sec:real-data}

We illustrate our methods with a multimodal neuroimaging analysis to study Alzheimer's disease (AD). AD is an irreversible neurodegenerative disorder characterized by progressive impairment of cognitive and memory functions. It is the leading form of dementia in elderly subjects, and is the sixth leading cause of death in the United States. In 2018, AD affects over 5.5 million Americans, and without any effective treatment and prevention, this number is projected to almost triple by 2050 \citep{ADfacts2018}. Tau is a hallmark pathological protein of AD, and is believed to be part of the driving mechanism of the disorder. It is present in the brains of both AD subjects and the elderly absent of dementia. Brain atrophy is another well known characteristic that differentiates between AD and normal aging. We study a dataset with $n = 125$ subjects. Each subject receives a positron emission tomography (PET) scan with AV-1451 tracer that measures accumulation of tau protein, as well as an anatomical magnetic resonance imaging (MRI) scan that measures brain grey matter cortical thickness. We map both types of images to a common brain atlas from Free Surfer, then summarize each PET and MRI image by a 58-dimensional vector, with each entry measuring the tau accumulation and cortical thickness of a particular brain region-of-interest (ROI), respectively. We remove some regions with quality issues for the PET images, which result in $p_1 = 51$ ROIs for PET, and $p_2 = 58$ ROIs for MRI, for each subject. In our integrative analysis, the tau and cortical thickness measurements form the two modalities $\x_1$ and $\x_2$. Memory score is a critical measure of cognitive decline for AD, and in our analysis, the memory score after removing potential age and sex effects is the response variable $y$. 

\begin{table}[t!]
  \begin{center} \footnotesize
    \begin{tabular}[h]{ll|rrrr} \hline
      & & R.rostantcing & L.superiorparietal & L.inferiorparietal & L.middletemporal\\
      Coefficient & $\betabh_1$ & $0.07$ & $0.12$ & 0 & 0 \\
      & $\betabh_2$ & 0 & 0 & $-0.18$ & $-0.01$ \\
      $p$-value  &             & $0.05$ & $0.03$ & $0.001$ & $0.1$ \\ \hline
      & & L.parahippocampal & L.rostantcing & R.parstriangularis & R.superiortemporal\\
      Coefficient & $\betabh_1$ & 0 & 0 & 0 & 0 \\
      & $\betabh_2$ & $0.07$ & $0.06$ & $0.01$ & $0.12$ \\
      $p$-value  &             & $0.02$ & $0.03$ & $0.14$ & $0.03$ \\ \hline
      & & R.supramarginal & R.temppole &  & \\
      Coefficient & $\betabh_1$ & 0 & 0 &  &  \\
      & $\betabh_2$ & $0.11$ & $0.04$ &  &  \\
      $p$-value  &             & $0.02$ & $0.05$ &  & \\ \hline
    \end{tabular}
    \caption{The identified brain regions with the coefficient estimates and the corresponding $p$-values of the factor-adjusted Wald test for the significance of the brain regions.}
    \label{tab:1}
  \end{center}
\end{table}

We estimate the number of latent factors using the method of \citet{Bai2012}, which concludes that there are $\hat{K}_1=3$ factors in the tau modality and $\hat{K}_2 = 1$ factor in the cortical thickness modality. We then estimate $\betabs$ and $\gammabs$ using (\ref{eq:7}) with a SCAD penalty, where the tuning parameter was chosen by cross-validation. We then apply the three methods we develop in this article. We first test the significance of the entire modality using the factor-adjusted score test in Section \ref{sec:test-whole}. The $p$-values are $4.9 \times 10^{-7}$ and $1.2 \times 10^{-3}$, for testing the significance of tau and cortical thickness modality, respectively. As such, both modalities are clearly significantly associated with the memory outcome. We then report the estimated non-zero coefficients from our integrative factor model and their corresponding brain regions in Table \ref{tab:1}. We further carry out the factor-adjusted Wald test in Section \ref{sec:test-lincomb} to evaluate if the identified regions are significantly correlated with the outcome in either modality. We report the corresponding $p$-values in Table \ref{tab:1} as well. Our findings agree with the AD literature. For instance, the ROI with the smallest $p$-value we found is inferior parietal lobe, which is one of brain regions that is known to be associated with progression from healthy aging to AD \citep{Greene2010}. Another significant ROI is parahippocampal gyrus, and cortical thinning of this region has been identified as an early biomarker of AD \citep{Echavarri2011, Krumm2016}. Finally, we evaluate the contribution of each individual modality. If we include the tau modality $\x_1$ in the model first, and add the cortical thickness modality $\x_2$ next, we have $\hat{\sigma}_1^2 = \hat{\var}(\x_1'\betabs_1) = 0.11$, and $\hat{\sigma}^2_{2|1} = \hat{\var}(\x_2'\betabs_2|\x_1) = 0.18$. Correspondingly, $\hat{\sigma}_1^2 / \hat\sigma_y^2 = 14\%$, and $\hat{\sigma}^2_{2|1}/\hat{\sigma}^2_y=24\%$. In other words, the tau modality explains $14\%$ total variation in the response, and adding the cortical thickness modality explains an additional $24\%$ total variation. On the other hand, if we include the cortical thickness modality $\x_2$ in the model first, and add the tau modality $\x_1$ next, we have $\hat{\sigma}_2^2 = \hat{\var}(\x_2'\betabs_2) = 0.19$, and $\hat{\sigma}^2_{1|2} = \hat{\var}(\x_1'\betabs_1|\x_2) = 0.08$. Correspondingly, $\hat{\sigma}_2^2 / \hat\sigma_y^2 = 25\%$, and $\hat{\sigma}^2_{1|2}/\hat{\sigma}_y^2=11\%$. In other words, the cortical thickness modality explains $25\%$ total variation in the response, and adding the tau modality explains an additional $11\%$ total variation. We also note that, the explained variation depends on which modality is already included in the model, and the total explained variations do not necessarily match if the two modalities enter the model in different orders.

\section{Discussion}
\label{sec:discussion}

In recent years, high-dimensional inference has seen many fruitful results. Particularly, \cite{javanmard2020general,Cai2017,ZhuBradic2017} have developed a family of flexible debiased methods to test the hypothesis of the form $\betabs\in \mathcal{C}$, where $\mathcal{C}$ is a general set. By choosing different $\mathcal{C}$, these methods can handle a wide range of inference problems, and can potentially address the inference questions we target in this paper too. We view the proposed solution and the debiased method as two legitimate and complementary inferential approaches, each with its own strength and limitation. Specifically, the debiased method does not require the beta-min condition that is difficult to check in practice, but instead requires the eigenvalues of $\var(\x)$ to be bounded. Besides, it can test the linear combination of the whole $p$-dimensional parameters $\betabs$, while the existing work so far focuses on a single linear combination. By contrast, our proposal assumes the variables are driven by latent factors that satisfy a pervasive condition, and thus allows the largest eigenvalue of $\var(\x)$ to diverge. This may be desirable in the context of multimodal data analysis, as the predictor covariance matrix can have spiked eigenvalues. Nevertheless, our proposal requires the beta-min condition. In addition, we aim at testing multiple linear combinations, but require the number of linear combinations $r$ and the number of involved parameters $t$ in \eqref{eqn:hypothesis-lincomb} to be smaller than the sample size. In conclusion, we believe our proposal offers a useful solution to address some important inference questions in multimodal data analysis. Meanwhile, developing the debiased counterpart provides an important alternative, and is warranted for future research.


\section{Supplementary Materials}
\label{sec:suppl-mater}

\subsection{Proof of Theorem \ref{thm:1}}
\begin{proof}
We first prove the case when $\sigmae^2$ is known, then prove it when $\sigmae^2$ is unknown. Let
\begin{eqnarray*}
\tilde{\I}_{\thetab_{m}|\betab_{-m}} & = & \sigmae^{-2}\left\{\frac{1}{n} \sum_{i=1}^n \hat{\f}_{i,m}^{\otimes 2} - \hat{\W}' \left(\frac{1}{n}\sum_{i=1}^n \x_{i,-m} \hat{\f}_{i,m}' \right)\right\}, \\ 
\tilde{\S}(\betabh_{-m},\zero) & = & \frac{1}{n\sigmae^2} \sum_{i=1}^n (y_i-\x_{i,-m}'\betabh_{-m})(\hat{\f}_{i,m}-\hat{\W}'\x_{i,-m}). \end{eqnarray*}
We first show that, under $H_0$, when $\sigmae^2$ is known, 
\begin{equation} \label{eq:16}
\tilde{\T}_n=\sqrt{n}\tilde{\I}_{\thetab_m|\betab_{-m}}^{-1/2}\tilde{\S}(\betabh_{-m},\zero) \xrightarrow{D} N(\zero,\I_{K_m}).
\end{equation}

To prove (\ref{eq:16}), we show the following two results:
\begin{eqnarray}
    \sqrt{n}\Itruehalf\tilde{\S}(\betabh_{-m},\zero) & \xrightarrow{D} & N(\zero,\I_{K_m}), \label{eq:17} \\
    \Itilde & \xrightarrow{P} & \Itrue. \label{eq:18}
\end{eqnarray}
Then given (\ref{eq:17}) and (\ref{eq:18}), we obtain (\ref{eq:16}) by applying Slutsky's Theorem. 
  
To prove (\ref{eq:17}), define
\vspace{-0.05in}
\begin{equation*}
\ell(\betasnom,\thetasm)=\frac{1}{2n\sigmae^2} \ltwonorm{\Y-\X_{-m}\betasnom-\hat{\F}_m\thetasm}^2.
\end{equation*}
Then,
$\nabla_{\betanom} \ell(\betasnom,\thetasm) = - (n\sigmae^2)^{-1} (\Y-\X_{-m}\betabs_{-m}-\hat{\F}_m\thetabs_m)'\X_{-m}$, $\nabla^2_{\betanom\thetam} \ell(\betasnom,\thetasm) = (n\sigmae^2)^{-1} \hat{\F}_m'\X_{-m}$, and $\nabla^2_{\betanom\betanom} \ell(\betasnom,\thetasm) = (n\sigmae^2)^{-1} \X_{-m}'\X_{-m}$. Let  
\[
\S(\betasnom,\zero) = \frac{1}{n\sigmae^2}\sum_{i=1}^n (y_i-\x_{i,-m}\betasnom)(\f_{i,m}-{\W^{\ast}}'\x_{i,-m}).
\]
Noting that $\W^{\ast}=\E(\x_{i,-m}^{\otimes 2})^{-1}$ $\E(\x_{i,-m}\f_{i,m}')$, we have
\begin{equation} \label{eq:19}
\begin{split}
      \tilde{\S}(\betahnom,\zero) & = \S(\betasnom,\zero)-\frac{1}{n\sigmae^2}(\W^{\ast}-\hat{\W})' \X_{-m}'(\Y-\X_{-m}\betasnom) \\
      & \hspace{5.5ex} - \frac{1}{n\sigmae^2}(\hat{\F}_m'\X_{-m}-\hat{\W}'\X_{-m}'\X_{-m})(\betabh_{-m}-\betabs_{-m}).
\end{split}
\end{equation}
For the second summand on the right-hand-side of (\ref{eq:19}), by Lemma \ref{lemma:B2}, for each $k\in [K_m]$,
\begin{equation}
\label{eq:104}
\begin{split}
      & \hspace{3ex}\left|\frac{1}{n\sigmae^2} (\ws_k-\wh_k)'\X_{-m}'(\Y-\X_{-m}\betabs_{-m}) \right| \\
      & \leq  \lonenorm{\ws_k-\wh_k} \ssupnorm{\frac{1}{n\sigmae^2} \X'_{-m}(\Y-\X_{-m}\betabs_{-m})} 
      = \lonenorm{\ws_k-\wh_k} \ssupnorm{\frac{1}{n\sigmae^2} \X_{-m}'\epsilonb} \\
      & =  O_P\left(\ss_k \left[\sqrt{(\log p_{-m})/n}\{1\vee (n^{1/4}/\sqrt{p_m}) \}\right] \right)O_P\left( \sqrt{(\log p_{-m})/n}\right)
      =  o_P\left(n^{-1/2} \right),
\end{split}
\end{equation}
where $\ss_k=|\supp(\ws_k)|$. For the third summand on the right-hand-side of (\ref{eq:19}), let $(\hat{\F}_m)_k$ denote the $k$th column of $\hat{\F}_m$. By Lemmas \ref{lemma:B1} and \ref{lemma:B3}, for each $k \in [K_m]$,
\begin{equation}
\label{eq:105}
\begin{split}
      &\hspace{3ex}\left|\frac{1}{n\sigmae^2}\{(\hat{\F}_m)_k'\X_{-m}-\wh_k'\X_{-m}'\X_{-m} \}(\betabh_{-m}-\betabs_{-m}) \right| \\
      &\leq  \ssupnorm{\frac{1}{n\sigmae^2} \{(\hat{\F}_m)_k'\X_{-m}-\wh_k'\X_{-m}'\X_{-m} \}} \lonenorm{\betahnom-\betasnom} \\
      &=  O_P\left(\ss_{-m} \left\{\sqrt{(\log p_{-m})/n}+1/\sqrt{p_m} \right\} \right) O_P\left(\sqlogqn \left(1\vee \frac{n^{1/4}}{\sqrt{p_m}} \right) \right) \\
      &=  o_P\left(n^{-1/2} \right).
\end{split}
\end{equation}
Together with the central limit theorem of $\sqrt{n} \Itruehalf\S(\betasnom,\zero)$, we prove (\ref{eq:17}). 

To prove (\ref{eq:18}), since the dimension $K_m$ is fixed, all matrix norms are equivalent. In particular, we show that
\begin{equation}
\label{eq:103}
    \supnorm{\tilde{\I}_{\thetab_m|\betab_{-m}}-\Itrue}=o_P\left(1 \right).
\end{equation}
Let $( \tilde{\I}_{\thetab_m|\betab_{-m}}-\Itrue )_k$ denote the $k$th row of $\tilde{\I}_{\thetab_m|\betab_{-m}}-\Itrue$. By the definition of the information matrix $\I^{\ast}_{\thetab_{m}|\betab_{-m}}$, the identifiability assumption that $\E(\f_{i,m}^{\otimes 2})=\I_{K_m}$, and furthermore the fact that $(1/n) \sum_{i=1}^n \fh_{i,m}^{\otimes 2}=\I_{K_m}$, we have
\begin{eqnarray*}
    & & \supnorm{(\Itilde-\Itrue)_{k}} 
        = \sigmae^{-2}\ssupnorm{\wh_k' \left( \frac{1}{n} \sum_{i=1}^n \x_{i,-m}\fh_{i,m}' \right) -
        {\ws_k}'\E(\x_{i,-m}\f_{i,m}')} \\ 
    & \lesssim & \ssupnorm{(\wh_k-\ws_k)' \left(\frac{1}{n} \sum_{i=1}^n \x_{i,-m}\fh_{i,m}' \right)} 
                 + \ssupnorm{{\ws_k}' \left\{\frac{1}{n} \sum_{i=1}^n \x_{i,-m}\fh'_{i,m}- \E(\x_{i,-m}\f_{i,m}')\right\}}.
\end{eqnarray*}
For the first term, let $\hat{f}_{i,m_h}$ be the $h$th element of $\fh_{i,m}$, we have  
\begin{align*}
    & \hspace{3ex}\ssupnorm{(\wh_k-\ws_k)' \left(\frac{1}{n} \sum_{i=1}^n \x_{i,-m}\fh_{i,m}' \right)} = \max_{h\in [K_m]}
      \left|(\wh_k-\ws_k)' \left(\frac{1}{n} \sum_{i=1}^n \x_{i,-m}\hat{f}_{i,m_h} \right) \right| \\
    & \leq \slonenorm{\wh_k-\ws_k}\cdot \max_{h\in [K_m]} \left[\ssupnorm{\frac{1}{n} \sum_{i=1}^n \x_{i,-m}
      \{\hat{f}_{i,m_h}- \x_{i,-m}'\ws_h\}} \right. \\ 
    & \hspace{3ex} + \left. \ssupnorm{\E(\x_{i,-m}\x_{i,-m}'\ws_h)} 
      + \ssupnorm{\frac{1}{n} \sum_{i=1}^n \left\{\x_{i,-m}\x_{i,-m}'\ws_h - \E(\x_{i,-m}\x_{i,-m}'\ws_h)\right\}}\right] \\
    & = O_P\left(\ss_k \sqlogqn \left(1\vee \frac{n^{1/4}}{\sqrt{p_m}} \right) \right) =o_P\left(1 \right),
\end{align*}
where $s_k=|\supp(\ws_k)|$, and the second-to-last equality follows from Lemma \ref{lemma:B2}, and the fact that the dominating term in the bracket is $\supnorm{\E(\x_{i,-m}\x_{i,-m}'\ws_h)}=O(1)$. For the second term,
\begin{align*}
    & \ssupnorm{{\ws_k}' \left\{\frac{1}{n} \sum_{i=1}^n \x_{i,-m}\fh'_{i,m} -
      \E(\x_{i,-m}\f'_{i,m})\right\}}
      = \max_{h\in[K_m]}\left|{\ws_k}' \left\{\frac{1}{n} \sum_{i=1}^n \x_{i,-m}\hat{f}_{i,m_h} -
      \E(\x_{i,-m}f_{i,m_h})\right\} \right| \\
    &\leq \lonenorm{\ws_k}\cdot \max_{h\in [K_m]} \left[\ssupnorm{\frac{1}{n} \sum_{i=1}^n \x_{i,-m}\left(\hat{f}_{i,m_h} -
      \x_{i,-m}'\ws_h\right)} \right. \\
    & \hspace{8ex}\left.+ \ssupnorm{\frac{1}{n} \sum_{i=1}^n \left\{\x_{i,-m}\x_{i,-m}'\ws_h -
      \E(\x_{i,-m}\x_{i,-m}'\ws_h) \right\}} \right] \\   
    & = O_P\left(\ss_k \sqlogqn \left(1\vee \frac{n^{1/4}}{\sqrt{p_m}} \right) \right) =o_P\left(1 \right),
\end{align*}
where the second-to-last equality follows from (\ref{eq:47}), and the sub-Gaussian assumption on $X_{ij}$ and $\x_{i,-m}'\ws_h$, which is implied by (\ref{eq:1}) and Condition \ref{con:1}. 

Next, when $\sigmae^2$ is unknown, we have $\hat{\T}_n-\tilde{\T}_n=\tilde{\T}_n(\sigmae/\sigmaeh-1)=o_P\left(1 \right)$, which is implied by Condition \ref{con:5}. Then, applying Slutsky's Theorem completes the proof. 
\end{proof}

\subsection{Proof of Theorem \ref{thm:4}}
\begin{proof}
We divide the proof into two main steps. 

In Step 1, letting $\T_n^{\ast}=\sqrt{n} \Itruehalf \left\{ \S(\betabs,\gammabs_m)-\Itrue\gammasm \right\}$, and $\Qs_n = (\T_n^{\ast})'\T_n^{\ast}$, we show that $Q_n=\Qs_n+o_P\left(1 \right)$. First, recall $\tilde{\T}_n$ as defined in (\ref{eq:16}), we have
\begin{equation} \label{eqn:thm4-eq1}
    \T_n=(\sigma_{\epsilon}\sigmah_{\epsilon}^{-1})\tilde{\T}_n=\tilde{\T}_n +
    \frac{\sigma_{\epsilon}-\sigmah_{\epsilon}}{\sigmah_{\epsilon}}\tilde{\T}_n=\tilde{\T}_n+o_P\left(1 \right),
\end{equation}
where the last equality follows from Condition \ref{con:5}, and the $o_P$ statement applies to each element of $\T_n$. Next, we show that $\tilde{\T}_n=\T_n^{\ast}+o_P\left(1 \right)$. Letting $\T_n^{\dagger}=\sqrt{n} \Itruehalf \S(\betasnom,\zero)$, we have that 
\begin{equation*}
    \tilde{\T}_n=\T_n^{\dagger}+ \sqrt{n} \Itruehalf \{\tilde{\S}(\hat{\betab}_{-m},\zero)-\S(\betasnom,\zero) \} +
    \sqrt{n} \{\Itildehalf-\Itruehalf \}\tilde{\S}(\hat{\betab}_{-m},\zero).
\end{equation*}
By Lemmas \ref{lemma:B6} and \ref{lemma:B7}, uniformly for all $\betabs\in \mathcal{N}$, we have that 
\begin{align*}
    \sqrt{n} \Itruehalf \{\tilde{\S}(\hat{\betab}_{-m},\zero)-\S(\betasnom,\zero) \}&=o_P\left(1 \right),\\
    \sqrt{n} \{\Itildehalf-\Itruehalf \}\tilde{\S}(\hat{\betab}_{-m},\zero)&= o_P\left(1 \right).
\end{align*}
Therefore, $\tilde{\T}_n=\T^{\dagger}_n+o_P\left(1 \right)$. Recall that $\S(\betabs,\gammasm)= (n\sigmae^2)^{-1} \sum_{i=1}^n \epsilon_i \left( \f_{i,m}-{\W^{\ast}}'\x_{i,-m} \right)$. Henceforth, we have
\begin{align*}
\T^{\dagger}_n &= \sqrt{n} \Itruehalf\{\S(\betabs,\gammabs_m)-\Itrue\gammasm \} \\
&\hspace{3ex}+ \sqrt{n}\Itruehalf \{ \S(\betasnom,\zero)-\S(\betabs,\gammasm)+\Itrue\gammasm    \}\\
& = \sqrt{n} \Itruehalf\{\S(\betabs,\gammabs_m)-\Itrue\gammasm \} + o_P\left(1 \right) 
= \T_n^{\ast}+o_P\left(1 \right),
\end{align*}
where the second equality follows from Lemma \ref{lemma:B8}.  Therefore, we have $\tilde{\T}_n=\T_n^{\ast}+o_P\left(1 \right)$. Together with \eqref{eqn:thm4-eq1} and the continuous mapping theorem, we have $Q_n=\Qs_n+o_P\left(1 \right)$, which completes Step 1. 

In Step 2, we derive the $\chi^2$ approximation of $\Qs_n$. By definition,
\begin{align*}
    &\hspace{3ex}\sqrt{n} \Itruehalf\{\S(\betabs,\gammasm)-\Itrue\gammasm \}\\
    & = \sqrt{n} \Itruehalf \left\{\frac{1}{n\sigmae^2} \sum_{i=1}^n \epsilon_i(\f_{i,m}-{\W^{\ast} }'\x_{i,-m})
      \right\}- \sqrt{n} \I^{\ast^{1/2}}_{\gammam|\betab_{-m}}\gammasm\\
    &= \sum_{i=1}^n \xib_i-\sqrt{n} \I^{\ast^{1/2}}_{\gammam|\betab_{-m}} \gammasm,
\end{align*}
where $\xib_i=(\sqrt{n}\sigmae^2)^{-1}\Itruehalf \epsilon_i(\f_{i,m}- {\W^{\ast} }'\x_{i,-m})$. By direct calculation, we have that $\E(\xib_i)=\zero$, and $\sum_{i=1}^n \var(\xib_i)=\I_{K_m}$. By Conditions \ref{con:1} and \ref{con:21}, we have
\begin{align*}
    \sum_{i=1}^n \E \ltwonorm{\xib}^3
    &= \frac{1}{(n\sigmae^4)^{3/2}} \E|\epsilon_i|^3 \sum_{i=1}^n \E
      \ltwonorm{\Itruehalf(\f_{i,m}- {\W^{\ast} }'\x_{i,-m})}^3 \\
    &\lesssim \frac{1}{\sqrt{n}} (\E \ltwonorm{\f_{i,m}}^3+\E \ltwonorm{{\W^{\ast}
      }'\x_{i,-m}}^3)=o(1). 
\end{align*}
Then, by Lemma \ref{lemma:C4}, we have that  
\begin{equation} \label{eq:101}
    \sup_{\mathcal{C}} |\pr(\sum_{i=1}^n\xib_i\in \mathcal{C})-\pr(\Z\in \mathcal{C})| \to 0,
\end{equation}
where $\Z\sim N(0,K_m)$, and the supremum is taken over all convex sets $\mathcal{C} \in \mathcal{R}^{K_m}$. Consider a special subset $\mathcal{C}_x$ of $\mathcal{C}$, such that $\mathcal{C}_x= \{\z\in \mathcal{R}^{K_m}: \ltwonorm{\z-\sqrt{n}\I^{\ast^{1/2}}_{\gammab_m|\betab_{-m}}\gammasm}^2 \leq x \}$. It then follows from (\ref{eq:101}) that
\begin{equation*}
    \sup_x|\pr(\Qs_n\leq x) -
    \pr(\chi^2(1,h_n)\leq x)|=\sup_{x} |\pr(\sum_{i=1}^n\xib_i \in \mathcal{C}_x)-\Pr(\Z\in \mathcal{C}_x) | \to 0,  
\end{equation*}
where $h_n=n{\gammasm}'\Itrue\gammasm$. Since $Q_n= \Qs_n+o_P\left(1 \right)$. For any $x$ and $\varepsilon>0$, we have
\begin{equation} \label{eq:102}
    \begin{split} 
      \pr(\chi^2(1,h_n) & \leq  x-\varepsilon)+o(1) \leq \pr(\Qs_n\leq x - \varepsilon) + o(1) \leq \pr(Q_n\leq x) \\
      & \leq  \pr(\Qs_n\leq x+\varepsilon)+o(1) \leq \pr\left\{ \chi^2(1,h_n)\leq x+\varepsilon
      \right\} + o(1).  
    \end{split}
\end{equation}
In addition, by Lemma \ref{lemma:C5}, we have 
\begin{equation*} \label{eq:39}
    \lim_{\varepsilon\to 0} \limsup_n \left|\pr\{\chi^2(1,h_n)\leq x+\varepsilon \} - \pr\{\chi^2(1,h_n)\leq
      x-\varepsilon \}\right| \to 0.
\end{equation*}
Together, we have 
\begin{equation*}
    \sup_x \left|\pr(Q_n\leq x)-\pr(\chi^2(1,h_n)\leq x) \right| \to 0. 
\end{equation*}
This completes the proof.
\end{proof}

\subsection{Proof of Corollary \ref{cor:1}}
\begin{proof}
We only prove (b) when $\phi_{\gamma_m}=1/2$. The proofs of (a) and (c) are similar. 

Note that
\begin{align*}
    &\hspace{3ex}|\pr(Q_n\leq x)-\pr(\chi^2(K_m,h)\leq x)| \\
    &\leq |\pr(Q_n\leq x)-\pr(\chi^2(K_m,h_{m_n})\leq x)|+|\pr(\chi^2(K_m,h_{m_n})\leq
    x)-\pr(\chi^2(K_m,h)\leq x)| 
\end{align*}
Then, by Theorem \ref{thm:4}, it suffices to prove that
\begin{equation*}
    \lim_{n\to \infty} \sup_{x>0}|\pr(\chi^2(K_m,h_{m_n})\leq x)-\pr(\chi^2(K_m,h)\leq x)| =0.
\end{equation*}
Let $F(x;k,\lambda)$ denote the cumulative distribution function of the non-central $\chi^2$-distribution with $k$ degrees of freedom and non-centrality parameter $\lambda$, and $F(x;k)$ that of the central $\chi^2$-distribution with $k$ degrees of freedom. Then, 
\begin{equation*}
    F(x;k,\lambda)= e^{-\lambda/2} \sum_{j=1}^{\infty} \frac{(\lambda/2)^j}{j!} F(x;k+2j),
\end{equation*}
We have that
\begin{align*}
    &\hspace{3ex} \pr(\chi^2(K_m,h_{m_n})\leq x) - \pr(\chi^2(K_m,h)\leq x) \\
    &= e^{-h_{m_n}/2} \left\{F(x;K_m) + \sum_{j=1}^{\infty} \frac{(h_{m_n}/2)^j}{j!}F(x;K_m+2j) \right\}\\
    &\hspace{2.5ex}- e^{-h/2} \left\{F(x;K_m) + \sum_{j=1}^{\infty} \frac{(h/2)^j}{j!}F(x;K_m+2j) \right\}  
\end{align*}
Then, as $n\to \infty$,
\begin{align*}
    &\hspace{3ex}\sup_{x>0} |\pr(\chi^2(K_m,h_{m_n})\leq x) - \pr(\chi^2(K_m,h)\leq x)|\\
    &\leq |e^{-h_{m_n}/2}-e^{-h/2}| \sup_{x>0}\left\{F(x;K_m)+\sum_{j=1}^{\infty} \frac{(h_{m_n}/2)^j}{j!}F(x;K_m+2j)
      \right \} \\
    &\hspace{3ex}+ e^{-h/2}\sup_{x>0}\left\{\sum_{j=1}^{\infty} \frac{|(h_{m_n/2})^j-(h/2)^j|}{j!} F(x;K_m+2j)\right\}\\
    &\leq |e^{-h_{m_n}/2}-e^{-h/2}|e^{h_{m_n}/2}+e^{-h/2} \sum_{j=1}^{\infty}\frac{|(h_{m_n}/2)^j-(h/2)^j|}{j!} \to 0.
\end{align*}
This completes the proof. 
\end{proof}

\subsection{Proof of Theorem \ref{thm:2}}
\begin{proof} 
We first prove (a) to (c) of the theorem. Our main idea is to show that, by Lemma \ref{lemma:C2}, $\Fh$ is consistent to $\H\F$ for some nonsingular matrix $\H$. Moreover, $\Uh$ is consistent to $\U$. Then solving (\ref{eq:14}) is equivalent to solving the same problem by replacing $(\Fh,\Uh)$ with $(\F, \U)$. More precisely, by the Karush-Kuhn-Tucker conditions, any vector $(\gammabh_a,\betabh_a)$ satisfying the following equations is a solution to (\ref{eq:14}):
\begin{align}
    \frac{1}{n} \Fh'(\Y-\Fh\gammabh_a-\Uhts\betahats)&=\zero; \label{eq:20}\\
    \frac{1}{n} \Uh_T'(\Y-\Fh\gammabh_a-\Uhts\betahats)&=\zero;  \label{eq:21}\\
    \frac{1}{n} \Uh_{S_a}'(\Y-\Fh\gammabh_a-\Uhts\betahats)&=\lambda_a
                                                             \dot{p}(|\betabh_{a,S_a}|)I(\betabh_{a,S_a}>\zero);  \label{eq:22}\\ 
    \ssupnorm{\frac{1}{n}\Uh_{(\ts)^c}'(\Y-\Fh\gammabh_a-\Uhts\betahats)} &< \lambda_a \dot{p}(0+). \label{eq:23} 
\end{align}
where $\dot{p}(\cdot)$ is a vector of first derivatives of $p(\cdot)$, and $I(\cdot)$ is a vector of indicator functions applied to each coordinate of $\betabh_{a,S_a}$.

We divide the proof into two main steps. In Step 1, letting $\mathcal{M}=\{(\gammab,\betab): \supnorm{\gammab-\H'\gammabs}\leq C\delta_n, \supnorm{\betab-\betabs_{\ts}}\leq C\delta_n\}$ for some constant $C$, we show that, with probability tending to 1, there exists a vector $(\gammabh_a,\betahats)$ in $\mathcal{M}$ that satisfies (\ref{eq:20}), (\ref{eq:21}) and (\ref{eq:22}). In Step 2, we set $\betabh_a=(\betahats,\zero)$, and show that $(\gammabh_a,\betabh_a)$ satisfies (\ref{eq:23}). Together these two steps prove (a) and (b) of the theorem. Then (c) follows from $\ltwonorm{\betahats-\betasts}\leq \sqrt{t+s_a} \supnorm{\betahats-\betasts}$.

For Step 1, by (\ref{eq:1}), (\ref{eq:4}), and $\zetab=\F\gammabs-\Fh\gammabh_a+(\Uts-\Uhts)\betahats$, we have
\begin{equation*}
    \Y-\Fh\gammabh_a-\Uhts\betahats=\Uts(\betasts-\betahats)+\epsilonb+\zetab.
\end{equation*}
Therefore,
\begin{eqnarray}
    & & \frac{1}{n}
        \begin{pmatrix}
          \Uht' \\ \Uhs'
        \end{pmatrix}
    \left( \Y-\Fh\gammabh_a-\Uhts\betahats \right) \nonumber \\
    & = & \frac{1}{n}
          \begin{pmatrix}
            \Uht'\Ut & \Uht'\Us \\ \Uhs'\Ut  & \Uhs'\Us
          \end{pmatrix}
    \begin{pmatrix}
    \betabs_T-\betabh_{a,T} \\ \betabs_{S_a}-\betabh_{a,S_a}
    \end{pmatrix}
    + \frac{1}{n}
    \begin{pmatrix}
      \Uht'(\epsilonb+\zetab) \\ \Uhs'(\epsilonb+\zetab)
    \end{pmatrix} \nonumber \\
    & = &\K_n
          \begin{pmatrix}
            \betabs_T-\betabh_{a,T} \\ \betabs_{S_a}-\betabh_{a,S_a}
          \end{pmatrix}+ \frac{1}{n}
    \begin{pmatrix}
      \Ut'\epsilonb \\ \Us'\epsilonb
    \end{pmatrix} 
    + \frac{1}{n}
    \begin{pmatrix}
      (\Uht-\Ut)'\epsilonb+\Uht'\zetab \\ (\Uhs-\Us)'\epsilonb+\Uhs'\zetab
    \end{pmatrix} \nonumber \\
    & & - \, \frac{1}{n}
        \begin{pmatrix}
          (\Ut-\Uht)'\Ut & (\Ut-\Uht)'\Us \\ (\Us-\Uhs)'\Ut & (\Us-\Uhs)'\Us
        \end{pmatrix}
    \begin{pmatrix}
    \betabs_T-\betabh_{a,T} \\ \betabs_{S_a}-\betabh_{a,S_a}
    \end{pmatrix} \nonumber \\ 
    & = & \K_n
          \begin{pmatrix}
            \betabs_T-\betabh_{a,T} \\ \betabs_{S_a}-\betabh_{a,S_a}
          \end{pmatrix} 
    + \frac{1}{n}
    \begin{pmatrix}
      \Ut'\epsilonb \\ \Us'\epsilonb
    \end{pmatrix} 
    +
    \begin{pmatrix}
      \RR_T \\\RR_{S_a}
    \end{pmatrix}, \label{eq:24}
\end{eqnarray}
where
\begin{equation*} \label{eq:25}
    \RR_{\ts} = \frac{1}{n} \left\{ (\Uhts-\Uts)'\Uts\betasts+\Uhts'(\F\gammabs-\Fh\gammabh_a) + (\Uhts-\Uts)'\epsilonb \right\}.
\end{equation*}
By Lemma \ref{lemma:B4}, we have
\vspace{-0.05in}
\begin{equation*}
    \supnorm{\RR_{\ts}}=O_P\left(\sqrt{\frac{\log p}{n}}\left(\frac{1}{\sqrt{n}}+\frac{n^{1/4}}{\sqrt{p_{\min}}} \right) \right).
\end{equation*} 
In addition, by Condition \ref{con:1}, $U_{ij}\epsilon_i$ is sub-exponential. Therefore, by Bernstein inequality and the union bound, we have $\supnorm{n^{-1}\U_{\ts}'\epsilonb}=O_P\left(\sqrt{(\log p)/n} \right)$. Together, 
\begin{equation} \label{eq:26}
    \supnorm{\frac{1}{n}\U_{\ts}'\epsilonb}+\supnorm{\RR_{\ts}}=O_P\left(\delta_n \right), \text{ where } \delta_n=\keyrate.
\end{equation}
By Condition \ref{con:6} and the sub-Gaussian assumption on $\u$, it holds with probability tending to 1 that $\lambda_{\min}(\K_n)$ is bounded away from 0. Then by (\ref{eq:24}), there exists $\betabh_{a,T}\in \mathcal{M}$ that solves (\ref{eq:21}). By the assumption that $\sqrt{n}\lambda_a\dot{p}(d_n)=o(1)$, we have  $\lambda_a\dot{p}(|\betab_{a,S_a}|)\leq\lambda_a\dot{p}(d_n)=o(\delta_n)$ for all $\betab_{a,S_a}\in \mathcal{M}$. Then, by (\ref{eq:24}), there exists $\betabh_{a,S_a}\in \mathcal{M}$ that solves (\ref{eq:22}). Finally, as we show in Lemma \ref{lemma:B5}, $\supnorm{n^{-1}\Fh'(\Y-\Uh \betabh_{a}- \Fh\H'\gammabs)}=O_P\left(\delta_n \right)$, and thus there exists $\betabh_{a,S_a} \in \mathcal{M}$ that solves (\ref{eq:20}), which completes Step 1. 

For Step 2, let $\bdm{\varphi}_{\ts}=(\zero, \lambda_a\dot{p}(|\betabh_{a,S_a}|)I(\betabh_{a,S_a}>\zero))$, we have
\begin{eqnarray*} \label{eq:27}
    & & n^{-1}\Uh_{(\ts)^c}'(\Y-\Fh\gammabh_a-\Uhts\betahats) \\
    & = & n^{-1}\Uh_{(\ts)^c}'\U_{\ts}(\betasts-\betahats) +n^{-1}\Uh_{(\ts)^c}'\epsilonb+n^{-1}\Uh_{(\ts)^c}'\zetab \\
    & = & n^{-1}\U_{(\ts)^c}'\U_{\ts}\K_n^{-1}\{\bdm{\varphi}_{\ts}-n^{-1}\U_{\ts}'\epsilonb-\RR_{\ts}\}\\
    & & + \, n^{-1}\Uh_{(\ts)^c}'\epsilonb + n^{-1}\Uh_{(\ts)^c}'(\F\gammabs-\Fh\gammabh_a) + n^{-1}\Uh_{(\ts)^c}(\U_{\ts}-\Uh_{\ts})\betasts.
\end{eqnarray*}
By Condition \ref{con:8} and the sub-Gaussian assumption on $\U$, with probability tending to 1, we have $\Lsupnorm{n^{-1}\U_{(\ts)^c}'\U_{\ts}\K_n^{-1}} = O(1)$. Therefore, by (\ref{eq:26}), we have
\begin{eqnarray*} \label{eq:28}
    & & \lambda_a^{-1}\supnorm{n^{-1}\U_{(\ts)^c}'\U_{\ts}\K_n^{-1}\{n^{-1}\U_{\ts}'\epsilonb+\RR_{\ts}\}} \\
    & \leq & \lambda_a^{-1}\Lsupnorm{n^{-1}\U_{(\ts)^c}'\U_{\ts}\K_n^{-1}} \supnorm{n^{-1}\U_{\ts}'\epsilonb+\RR_{\ts}} \\
    & = & O_P\left(\delta_n\lambda_a^{-1} \right) = o_P(1).
\end{eqnarray*}
Next,
\begin{eqnarray*} \label{eq:29}
    \lambda_a^{-1}\supnorm{n^{-1}\U_{(\ts)^c}'\U_{\ts}\K_n^{-1}\lambda_a\dot{p}(|\betabh_{a,\ts}|)}
    \lesssim\dot{p}(|\betabh_{a,\ts}|)<\dot{p}(d_n)<\dot{p}(0+).
\end{eqnarray*}
Moreover, by Lemma \ref{lemma:B4},
\begin{eqnarray*}
    & & \lambda_a^{-1}n^{-1}\supnorm{\Uh_{(\ts)^c}'\epsilonb +
        \Uh_{(\ts)^c}'(\F\gammabs-\Fh\gammabh_a)+\Uh_{(\ts)^c}(\U_{\ts}-\Uh_{\ts})\betasts} \nonumber \\ %
    & = & O_P\left(\delta_n \lambda_a^{-1}\right) = o_P(1).     \label{eq:30}
\end{eqnarray*}
Putting together the above results, we have that $\betabh$ satisfies (\ref{eq:23}), which completes Step 2.

Finally, we prove (d) of the theorem. By Lemma \ref{lemma:B4}, when $p_{\min}\gg n^{3/2}$, $\supnorm{\RR_{T \cup S}}=o_P\left(n^{-1/2} \right)$.  Then, it follows from (\ref{eq:21}), (\ref{eq:23}) and (\ref{eq:24}) that 
\begin{eqnarray*}
\sqrt{n}(\betabh_{a,T}-\betabs_T) & = & \frac{1}{\sqrt{n}}\K_n^{-1}\Ut'\epsilonb +o_P\left(1 \right).  \label{eq:31}\\
\sqrt{n}(\betabh_{a,S_a}-\betabs_{S_a}) & = & \frac{1}{\sqrt{n}}\K_n^{-1}\Us'\epsilonb - \K_n^{-1}\sqrt{n}\lambda_a\dot{p}(|\betabh_{a,S_a}|)I(\betabh_{a,S_a}>\zero) + o_P\left(1 \right).  \label{eq:32} 
\end{eqnarray*}
Since $\sqrt{n}\lambda_a\dot{p}(|\betabh_{a,S_a}|)\leq \sqrt{n}\lambda_a\dot{p}(d_n)=o(1)$. Therefore,
\begin{equation*}
    \K_n^{-1}\sqrt{n}\lambda_a\dot{p}(|\betahas|)=O_P\left(\sqrt{n}\lambda_a\dot{p}(d_n) \right)=o_P\left(1 \right).
\end{equation*}
This completes the proof.
\end{proof}

\subsection{Proof of Theorem \ref{thm:3}}
\begin{proof}
We divide the proof into two main steps. 

In Step 1, define $T_0$ as
\begin{equation*} \label{eq:33}
T_0 = \sigmae^{-2}(\omegab_n+\sqrt{n}\h_n)'\Psib^{-1}(\omegab_n + \sqrt{n}\h_n),
\end{equation*}
where $\h_n=\A\betabs-\b$, $\Psib=\A\Omegab_T\A'$, $\Omegab_{T}$ is the the submatrix of $\Omegab_{\ts}=\Sigmab_{u,\ts}^{-1}$ with rows and columns in $T$, and $\Sigmab_{u,\ts}^{-1}$ is inverse of the submatrix of $\Sigmab_u$ with rows and columns in $\ts$,  and $\omegab_n=n^{-1/2}
\begin{pmatrix}
\A & \zero
\end{pmatrix}
\K_n^{-1} \Uts'\epsilonb$. We first show that  $T_w/r=T_0/r+o_P\left(1 \right)$. By Theorem \ref{thm:2}, we have
\begin{equation*}
    \sqrt{n}
    \begin{pmatrix}
      \betahat -\betabs_T \\ \betabh_{S_a} - \betabs_{S_a} 
    \end{pmatrix} 
    = \frac{1}{\sqrt{n}} \K_n^{-1}
    \begin{pmatrix}
      \Ut'\epsilonb \\ \Us'\epsilonb
    \end{pmatrix} 
    + \RR_a,
\end{equation*}
for some remainder term $\RR_a$ such that $\ltwonorm{\RR_a}=o_P\left(1 \right)$. Then we have
\begin{equation*}
    \sqrt{n}\A(\betahat-\betabs_T)=\omegab_n+\A\RR_{a,T},
\end{equation*}
where $\RR_{a,T}$ is the subvector of $\RR_a$ with indices in $T$. By definition, $\A\betabs_T-\b=\h_n$. Then,
\begin{equation*}
    \sqrt{n}(\A\betahat-\b)=\omegab_n+\A\RR_{a,T}+\sqrt{n}\h_n. 
\end{equation*}
Let $\Psib_n=\A(\K_{n}^{-1})_T\A'$, where $(\K_{n}^{-1})_T$ is submatrix of $\K_n^{-1}$ with rows and columns in $T$. We have
\begin{equation} \label{eq:34}
    \sqrt{n} \Psib_n^{-1/2}(\A\betahat-\b)=\Psib_n^{-1/2}(\omegab_n+\A\RR_{a,T}+\sqrt{n}\h_n).
\end{equation}
Next, we bound $\ltwonorm{\sqrt{n} \Psib_n^{-1/2}(\A\betahat-\b)}$. By Lemmas \ref{lemma:C3} and \ref{lemma:B4}, $\ltwonorm{\Psib_n^{-1/2}\A}=O_P\left(1 \right)$, and $\ltwonorm{\RR_{a,T}}=o_P\left(1 \right)$. Then it follows that
\begin{equation} \label{eq:35}
    \ltwonorm{\Psib_n^{-1/2}\A\RR_{a,T}} \leq \ltwonorm{\Psib_n^{-1/2}\A}\ltwonorm{\RR_{a,T}}=o_P\left(1 \right). 
\end{equation}
Therefore, $\sqrt{n}\Psib_n^{-1/2}(\A\betahat-\b)=\Psib_n^{-1/2}(\omegab_n+\sqrt{n}\h_n)+o_P\left(1 \right)$. We further note that,
\begin{equation*}
\E \ltwonorm{\Psib_n^{-1/2}\omegab_n}^2=\tr\left[ \E_{\u}\{\Psib_n^{-1/2}\E_{\epsilon}(\omegab_n\omegab_n')\Psib_n^{-1/2}\} \right] = r\sigmae^2. 
\end{equation*}
Then, by Markov's inequality, $\ltwonorm{\Psib_n^{-1/2}\omegab_n}=O_P\left(\sqrt{r} \right)$. By Lemma \ref{lemma:C3}, $\lambda_{\max}(\Psib_n^{-1})= O_P\left(1 \right)$. By Condition \ref{con:10}, $\ltwonorm{\h_n}=O(\sqrt{r/n})$, so that $\ltwonorm{\sqrt{n}\Psib_n^{-1/2}\h_n}=O_P\left(\sqrt{r} \right)$. Therefore, $\ltwonorm{\sqrt{n}\Psib_n^{-1/2}(\A\betahat-\b)}=O_P\left(\sqrt{r} \right)$. By Lemma \ref{lemma:C3}, $\ltwonorm{\Psib_n^{1/2}(\A\Omegabh_T\A')^{-1}\Psib_n^{1/2}-\I}=O_P\left((s_a+t)/\sqrt{n} \right)$, where $\Omegabh_T$ is the first $T$ rows and columns of the matrix
\begin{equation*}
    \Omegabh_{T\cup\Sha} = n
    \begin{pmatrix}
      \Uht'\Uht & \Uht'\Uh_{\Sha} \\ \Uh_{\Sha}'\Uht & \Uh_{\Sha}'\Uh _{\Sha }
    \end{pmatrix}^{-1}.
\end{equation*}
Therefore, under the assumption that $s_a+t=o(n^{1/2})$, we have 
\begin{eqnarray} \label{eq:36}
    & & \ltwonorm{\{\sqrt{n}\Psib_n^{-1/2}(\A\betahat-\b) \}' \{\Psib_n^{1/2}(\A\Omegabh_T\A')^{-1}\Psib_n^{1/2}-\I \} \{\sqrt{n}\Psib_n^{-1/2}(\A\betahat-\b) \}}^2  \nonumber \\
    & \leq & \ltwonorm{\Psib_n^{1/2}(\A\Omegabh_T\A')^{-1}\Psib_n-\I} \ltwonorm{\sqrt{n}\Psib_n^{-1/2}(\A\betahat-\b)}^2 \nonumber \\
    & = & O_P\left(\frac{r(s_a+t)}{\sqrt{n}} \right)=o_P\left(r \right).
\end{eqnarray}
Let $T_{w,0}=\sigmaeh^{-2}n(\A\betahat-\b)'\Psib^{-1}_n(\A\betahat-\b)$. By $T_w$'s definition and (\ref{eq:36}), we have $\sigmaeh^2|T_w-T_{w,0}|=o_P\left(r \right)$. Condition \ref{con:5} implies $1/\sigmaeh^2=O_P(1)$. Therefore, $|T_w-T_{w,0}|=o_P\left(r \right)$. 

Next, we show that $|T_{w,0}-T_0|=o_P\left(r \right)$. Let $T_{w,1}=\sigmaeh^{-2}\ltwonorm{\Psib^{-1/2}\omegab_n+\sqrt{n} \Psib^{-1/2}\h_n}$. By (\ref{eq:34}) and (\ref{eq:35}), we have
\begin{eqnarray*}
    \sigmaeh^2 T_{w,0}
    & = & \ltwonorm{\Psib_n^{-1/2}\omegab_n+\sqrt{n}\Psib_n^{-1/2}\h_n+o_P\left(1 \right)}^2 \\
    & = & \ltwonorm{\Psib_n^{-1/2}\omegab_n+\sqrt{n}\Psib_n^{-1/2}\h_n}^2 + o_P\left(1 \right) + o_P\left(\Psib_n^{-1/2}(\omegab_n+\sqrt{n}\h_n) \right) \\
    & = &\ltwonorm{\Psib_n^{-1/2}\omegab_n+\sqrt{n}\Psib_n^{-1/2}\h_n}^2 + o_P\left(1 \right) + o_P\left(r \right) \\
    & = & \sigmaeh^2T_{w,1}+ (\omegab_n+\sqrt{n}\h_n)'(\Psib_n^{-1}-\Psib^{-1})(\omegab_n+\sqrt{n}\h_n)+o_P\left(r \right).
\end{eqnarray*}
By Lemma \ref{lemma:C3}, we have that $\ltwonorm{\Psib_n^{-1}-\Psib^{-1}}=o_P(1)$. Since $\ltwonorm{\omegab_n}\leq \ltwonorm{\Psib_n^{1/2}}\ltwonorm{\Psib_n^{-1/2}\omegab_n}=O_P\left(\sqrt{r} \right)$, and by Condition \ref{con:10}, $\ltwonorm{\sqrt{n}\h_n}=O_P\left(\sqrt{r} \right)$. Therefore, $\ltwonorm{\omegab_n+\sqrt{n}\h_n}=O_P(\sqrt{r})$. Considering that $1/\sigmaeh^2=O(1)$, we have $T_{w,0}=T_{w,1}+o_P\left(r \right)$. Finally, we have
\begin{equation*}
    |T_{w,1}-T_0|=\frac{|\sigmaeh^2-\sigmae^2|}{\sigmaeh^2\sigmae^2} \ltwonorm{\Psib^{-1/2}\omegab_n+\sqrt{n} \Psib^{-1/2}\h_n} = o_P\left(r \right).
\end{equation*}
Therefore $|T_{w,0}-T_0|=o_P\left(r \right)$. 

Combining the results $|T_w-T_{w,0}|=o_P\left(r \right)$ and $|T_{w,0}-T_0|=o_P\left(r \right)$ completes Step 1.
  
In Step 2, we show that the $\chi^2$ approximation holds for $T_0$. Recall the definition of $T_0$, which can be written as $T_0 = \sigmae^{-2} \ltwonorm{\Psib^{-1/2}\omegab_n+\sqrt{n} \Psib^{-1/2}\h_n}^2$. By the definition of $\omegab_n$,  
\begin{equation*}
    \sigmae^{-1}\Psib^{-1/2}\omegab_n= \sum_{i=1}^n \frac{1}{\sqrt{n\sigmae^2}} \Psib^{-1/2}
    \begin{pmatrix}
      \A & \zero
    \end{pmatrix} \Omegab_{\ts}
    \U_{i,\ts}\epsilon_i = \sum_{i=1}^n \xib_i. 
\end{equation*}
By direct calculation, we have $\sum_{i=1}^n \var(\xib_i)=\I_r$. Because of the sub-Gaussian assumption on $\epsilon$ in Condition \ref{con:1}, we have $E|\epsilon|^3< \infty$. Then,
\begin{eqnarray*}
    r^{1/4} \sum_{i=1}^n \E \ltwonorm{\xib_i}^3 
    &=& \frac{r^{1/4}}{(n\sigmae^2)^{3/2}} \sum_{i=1}^n \E \ltwonorm{\Psib^{-1/2}
        \begin{pmatrix}
          \A & \zero
        \end{pmatrix} \Omegab_{\ts}\U_{i,\ts}\epsilon_i}^3\\
    &=& \frac{r^{1/4}}{(n\sigmae^2)^{3/2}} \E|\epsilon|^3\sum_{i=1}^n \E \ltwonorm{\Psib^{-1/2}
        \begin{pmatrix}
          \A & \zero
        \end{pmatrix} \Omegab_{\ts} \U_{i,\ts}}^3\\
    &\lesssim& \frac{r^{1/4}}{(n\sigmae^2)^{3/2}} \sum_{i=1}^n \E \{\ltwonorm{\Psib^{-1/2}
               \begin{pmatrix}
                 \A & \zero
               \end{pmatrix} \Omegab_{\ts}^{1/2}}^3 \ltwonorm{\Omegab_{\ts}^{1/2}\U_{i,\ts}}^3 \}\\
    & \lesssim & \frac{r^{1/4}}{(n\sigmae^2)^{3/2}} \sum_{i=1}^n \E |\U_{i,\ts}'\Omegab_{\ts}\U_{i,\ts}|^{3/2},\\
    & \lesssim & \frac{r^{1/4}}{n^{1/2}} \E |\u_{\ts}'\Omegab_{\ts}\u_{\ts}|^{3/2}=o(1), 
\end{eqnarray*}
where the third-to-last relation is due to the fact that
\begin{equation*}
    \ltwonorm{\Psib^{-1/2}
      \begin{pmatrix}
        \A & \zero
      \end{pmatrix} \Omegab_{\ts}^{1/2}}=\lambda_{\max} \left[\tr \left\{ \Psib^{-1/2}
        \begin{pmatrix}
          \A & \zero
        \end{pmatrix} \Omegab_{\ts}
        \begin{pmatrix}
          \A' \\ \zero'
        \end{pmatrix}
        \Psib^{-1/2}
      \right\} \right]=r,
  \end{equation*}
and the last equality follows from Condition \ref{con:11}. Then by Lemma \ref{lemma:C4}, we have
\begin{equation} \label{eq:37}
    \sup_{\mathcal{C}} \left|\pr\{\sigmae^{-1}\Psib^{-1/2}\omegab_n \in \mathcal{C} \}- \pr(\Z\in \mathcal{C})
    \right| \to 0, 
\end{equation}
where $\Z \sim N(\zero,\I_r)$ and the supremum is taken over all convex sets $\mathcal{C}\in \Rcal^r$.
  
Consider a special subset $\mathcal{C}_x$ of $\mathcal{C}$, where 
\begin{equation*}
    \mathcal{C}_x = \{\z\in \Rcal^r: \ltwonorm{\z+ \sqrt{n/\sigmae^2} \Psib^{-1/2}\h_n}^2\leq x \}. 
\end{equation*}
It follows from (\ref{eq:37}) that
\begin{equation*}
    \sup_x \left|\pr\{(n\sigmae^2)^{-1/2} \Psib^{-1/2}\omegab_n\in \mathcal{C}_x \}- \pr(\Z\in \mathcal{C}_x) \right| 
    = \sup_x \left|\pr(T_0\leq x) - \pr(\chi^2(r,\nu_n)\leq x)\right| \to 0,
\end{equation*}
where $\nu_n=n\sigmae^{-2}\h_n'\Psib^{-1}\h_n$.
  
Consider any statistic $T^{\ast}= T_0+o_P\left(r \right)$. For any $x$ and $\varepsilon>0$, we have
\begin{eqnarray*} \label{eq:38}
    \pr(\chi^2(r,\nu_n) & \leq & x-r\varepsilon)+o(1) \leq \pr(T_0\leq x - r\varepsilon) + o(1) \leq \pr(T^{\ast}\leq x) \\
                        & \leq & \pr(T_0\leq x+r\varepsilon)+o(1) \leq \pr\left\{ \chi^2(r,\nu_n)\leq x+r\varepsilon
                                 \right\} + o(1).  
\end{eqnarray*}
In addition, by Lemma \ref{lemma:C5}, we have 
\begin{equation*} \label{eq:39}
    \lim_{\varepsilon\to 0} \limsup_n \left|\pr\{\chi^2(r,\nu_n)\leq x+r\varepsilon \} - \pr\{\chi^2(r,\nu_n)\leq
      x-r\varepsilon \}\right| \to 0.
\end{equation*}
Together, we have 
\begin{equation*}
    \sup_x \left|\pr(T^{\ast}\leq x)-\pr(\chi^2(r,\nu_n)\leq x) \right| \to 0. 
\end{equation*}
This completes Step 2.

Combining the results of Step 1 and Step 2 completes the proof.
\end{proof}

\subsection{Consistence of $\hat{\sigma}_{\epsilon}^2$} \label{sec:cons-hats}
Recall the estimator for $\sigmae$ in Section \ref{sec:estimation}, $\sigmaeh^2=n^{-1} \sum_{i=1}^n (y_i-\x_i'\betabt)^2$, where $\betabt=\argmin_{\betab} (2n)^{-1}$ $\sum_{i=1}^n (y_i-\x_i'\betab)^2 + \lambda_{\epsilon} \lonenorm{\betab}$. The next proposition shows that $\sigmaeh^2$ is a consistent estimator of $\sigmae^2$, and thus $\sigmaeh^2$ satisfies Condition \ref{con:5}.

\begin{pro} \label{pro:2}
Suppose $\x_m $ satisfies the factor decomposition in \eqref{eq:1} for $m\in [M]$, and Conditions \ref{con:1} and \ref{con:3} hold. Suppose $s^{\ast}(\log p)/n = o(1)$, where $s^{\ast}=|\supp(\betabs)|$, and $\lambda_{\epsilon}=C \sqrt{(\log p)/n}$, where $C$ is a positive constant. Then $\sigmaeh^2=\sigmae^2+o_P\left(1 \right)$.
\end{pro}

\begin{proof}
Letting $\H_x=n^{-1} \sum_{i=1}^n \x_i^{\otimes 2}$ and $\Deltabh=\betabt-\betabs$, we have
\begin{equation}
    \label{eq:106}
    \sigmaeh^2-\sigmae^2= \frac{1}{n} \sum_{i=1}^n \epsilon_i^2 - \sigmae^2 + \Deltabh'\H_x\Deltabh - \Deltabh'
    \left(\frac{2}{n} \sum_{i=1}^n \epsilon_i\x_i \right).
\end{equation}
By the sub-Gaussian assumption on $\epsilon_i$ in Condition \ref{con:1}, it follows from the standard concentration result \citep[e.g.,][Lemma H.2]{ning2017general} that
\[
\frac{1}{n} \sum_{i=1}^n \epsilon_i^2 - \sigmae^2=O_P\left(\sqrt{(\log n)/n} \right). 
\]
Condition \ref{con:3} implies that $\lambda_{\min}( \E(\x_i^{\otimes 2})) \geq \lambda_{\min}( \E(\u_i^{\otimes 2}))>c$. Then, it follows from \citet[Proposition 1]{raskutti2011minimax} that the restricted eigenvalue condition holds for $\H_x$. Then, following a similar argument as in the proof of \citet[Lemma B.3]{ning2017general}, we have
\begin{equation} \label{eq:107}
    \Deltabh'\H_x\Deltabh=O_P\left(s^{\ast}(\log p)/n \right).
\end{equation}

Moreover, $\epsilon_ix_{ij}=\epsilon_i (\sum_{k=1}^K \lambda_{jk}f_{ik}+u_{ij})$, where $\lambda_{jk}$ is the $(j,k)$th element of $\Lambdab$. It follows from Conditions \ref{con:1} and \ref{con:3} that $x_{ij}$ is also sub-Gaussian. Therefore, $\epsilon_ix_{ij}$ is sub-exponential. This implies that $\supnorm{n^{-1} \sum_{i=1}^n \epsilon_i\x_{i}}=O_P(\sqrt{(\log p)/{n}} )$. By the well-known estimation error of the Lasso estimator \cite[e.g.,][Corollary 2]{negahban2012unified}, it holds that $\lonenorm{\Deltabh}=O_P\left(s^{\ast} \sqrt{(\log p)/n} \right)$. Therefore, 
\begin{equation}
\label{eq:108}
\left|(2/n)\Deltabh'\sum_{i=1}^n \epsilon_i\x_i\right| \leq \lonenorm{\Deltabh}\supnorm{(2/n) \sum_{i=1}^n \epsilon_i\x_i}
= O_P\left(s^{\ast}(\log p)/n \right).
\end{equation}
Putting (\ref{eq:106}), (\ref{eq:107}) and (\ref{eq:108}) together completes the proof.  
\end{proof}

\subsection{Proof of Proposition \ref{prop:1}}

\begin{proof}
To prove (a), consider regressing $y$ using all but the $m$th modality. Letting $\P_{-m}=\X_{-m}(\X_{-m}'\X_{-m})^{-1}\X'_{-m}$, then $\hat{\Y}_{-m}=\P_{-m}\Y$. Letting $\xib=\X_m\betabs_m+\epsilonb$, then
\begin{align*}
    \ltwonorm{\Y-\hat{\Y}_{-m}}^2 & =  \Y'(\I_n-\P_{-m})\Y = \xib'(\I_n-\P_{-m})\xib \\
                                  & =  \epsilonb'(\I_n-\P_{-m})\epsilonb + 2(\X_m\betabs_m)'(\I_n-\P_{-m})\epsilonb +
                                    (\X_m\betabs_m)'(\I_n-\P_{-m})\X_m\betabs_m. 
\end{align*}
Taking the expectation on both sides, and noting that, 
\begin{eqnarray*}
    \E\{\epsilonb'(\I_n-\P_{-m})\epsilonb \} & = & (n-p_{-m})\sigmae^2, \\
    \E\{2(\X_m\betabs_m)'(\I_n-\P_{-m})\epsilonb \} & = & 0, \\
    \E \{(\X_m\betabs_m)'(\I_n-\P_{-m})\X_m\betabs_m \} & = & \E_{\x_{-m}} [\E_{\x_m|\x_{-m}}\{(\X_m\betabs_m)'(\I_n-\P_{-m})\X_m\betabs_m \}] \\
 & = & \E_{\x_{-m}}[\tr\{(\I_n-\P_{-m})\sigma_{m|-m}^2 \}] = (n-p_{-m})\sigma_{m|-m}^2.
\end{eqnarray*}

To prove (b), note that $\E \ltwonorm{\Y-\hat{\Y}}^2=(n-p)\sigma_{\epsilon}^2$ when regressing $y$ on all data modalities. Then a direct calculation proves (b).

To prove (c), by factor decomposition, we have $\x_m=\Lambdab_m\f+\u_m$ and $\x_{-m}=\Lambdab_{-m}\f+\u_{-m}$. Then
\begin{equation*}
    \begin{pmatrix}
      \u_{-m}\\ \x_{-m}
    \end{pmatrix}
    \sim N \left(
      \begin{pmatrix}
        0 \\ 0
      \end{pmatrix},
      \begin{pmatrix}
        \Sigmaunom & \Sigmaunom \\ \Sigmaunom & \Sigmaxnom
      \end{pmatrix}
    \right),
\end{equation*}
where $\Sigmaxnom=\Lambdab_{-m}\Lambdab_{-m}'+\Sigmaunom$. Consequently,  
\begin{eqnarray*}
    \E(\u_{-m}|\x_{-m})=\Sigmaunom\Sigmaxnom^{-1}\x_{-m},  \quad
    \var(\u_{-m}|\x_{-m})=\Sigmaunom-\Sigmaunom\Sigmaxnom^{-1}\Sigmaunom. 
\end{eqnarray*}
As $\Lambdab_{-m}\f=\x_{-m}-\u_{-m}$, then $\f=\D_{-m}^{-1}\Lambdab_{-m}'(\x_{-m}-\u_{-m})$, where $\D_{-m}=\Lambdab_{-m}'\Lambdab_{-m}$. Then,
\begin{equation*}
    \x_m'\betabs_m=\f'\Lambdab_m'\betabs_m+\u_m'\betabs_m=(\x_{-m}-\u_{-m})'\C_m\betabs_m+\u_m'\betabs_m,
\end{equation*}
where $\C_m=\Lambdab_{-m}\D_{-m}^{-1}\Lambdab_m'$. Therefore, we have 
\begin{eqnarray} \label{eq:40}
    \E(\x_m'\betabs_m|\x_{-m}) =  \{\x_{-m}-\E(\u_{-m}|\x_{-m}) \}'\C_m\betabs_m = \x_{-m}'(\I-\Sigmaunom\Sigmaxnom^{-1})'\C_m\betabs_m.  
\end{eqnarray}
Moreover, 
\begin{eqnarray} \label{eq:41}
    \var(\x_m'\betabs_{-m}|\x_{-m}) & = & {\betabs_m}'\{\C_m'\var(\u_{-m}|\x_{-m})\C_m +\Sigmab_{u_m}\}\betabs_m \nonumber \\
                                    & = & {\betabs_m}'\{\C_m'(\Sigmaunom-\Sigmaunom\Sigmaxnom^{-1}\Sigmaunom)\C_m
                                          +\Sigmab_{u_m} \}\betabs_m.   
\end{eqnarray}
  
Then, by Woodbury matrix identity, we have 
\begin{equation*}
    \Sigmaxnom^{-1} = \Sigmaunom^{-1}-\Sigmaunom^{-1}\Lambdab_{-m}(\I_{K}+\Lambdab_{-m}'\Sigmaunom^{-1}\Lambdab_{-m})^{-1} \Lambdab_{-m}'\Sigmaunom^{-1}.  
\end{equation*}
Then we have, 
\begin{eqnarray*}
    \Sigmaunom-\Sigmaunom\Sigmaxnom^{-1}\Sigmaunom & = & \Lambdab_{-m}(\I_{K}+\Lambdab_{-m}'\Sigmaunom^{-1}\Lambdab_{-m})^{-1} \Lambdab_{-m}', \\
    \C_m'(\Sigmaunom-\Sigmaunom\Sigmaxnom^{-1}\Sigmaunom)\C_m & = & \Lambdab_m (\I_{K}+\Lambdab_{-m}'\Sigmaunom^{-1}\Lambdab_{-m})^{-1} \Lambdab_m'.
\end{eqnarray*}
Plugging these equalities into (\ref{eq:41}) gives
\begin{equation*}
    \var(\x_m'\betabs_{-m}|\x_{-m}) ={\betabs_m}'\{\Lambdab_m(\I_K+\Lambdab_{-m}'\Sigmaunom^{-1}\Lambdab_{-m})^{-1}\Lambdab_m'+\Sigmab_{u_m} \} \betabs_m,
\end{equation*}
which completes the proof of (c).
\end{proof}

\subsection{Additional technical lemmas}

\begin{lem} \label{lemma:C1} 
Suppose Conditions \ref{con:1}--\ref{con:3} hold. For any $m\in [M]$, there exists a nonsingular matrix $\H_m\in \Rcal^{K_m\times K_m}$, such that 
  \begin{enumerate}[(a)]
  \item $\max_{k\in [K_m]} (1/n) \sum_{i=1}^n |(\hat{\F}_m\H_m)_{ik} - f_{i,m_k}|^2=O_P\left(1/n+1/p_m \right)$, where $(\hat{\F}_m\H_m)_{ik}$ is the $(i,k)$th element of $\hat{\F}_m\H_m$, and $f_{i,m_k}$ is the $k$th element of $\f_{i,m}$.
  \item $\max_{i\in [n]} \ltwonorm{\hat{\f}_{i,m}-\H_m\f_{i,m}} = O_P\left(1/\sqrt{n}+n^{1/4}/\sqrt{p_m} \right)$.
  \item $\ltwonorm{\I_{K_m}-\H_m\H_m'}=O_P\left(1/\sqrt{n}+1/\sqrt{p_m} \right)$.
  \item $\max_{i\in [n],j\in [p_{m}]}|\hat{u}_{ij}-u_{ij}|=O_P\left(\sqrt{(\log n)/n}+n^{1/4}/\sqrt{p_m} \right)$.
  \end{enumerate}
\end{lem}
\begin{proof}
The results of (a) and (b) follow from Lemma C.9 of \cite{Fan2013b}. The results of (c) and (d) follow from Lemmas 3 and A.3 of \cite{li2018embracing}.
\end{proof}

\begin{lem} \label{lemma:C2}
Let $\F=(\F_1,\ldots,\F_M)$, $\Fh=(\Fh_1,\ldots,\Fh_M)$, where $\Fh_m$ is obtained by running PCA on the $m$th modality, $\H=\mathrm{diag}(\H_1,\ldots,\H_M)$, $\U=(\U_1,\ldots,\U_m)$, $\Uh=(\Uh_1,\ldots,\Uh_M)$, and $p_{\min}=\min_{m\in [M]} p_m$. Then the following results hold:
  \begin{enumerate}[(a)]
  \item $\max_{i\in [n]} \ltwonorm{\hat{\f}_{i}-\H\f_{i}} = O_P\left(1/\sqrt{n}+n^{1/4}/\sqrt{p_{\min}} \right)$.
  \item $\ltwonorm{\I_{K}-\H\H'}=O_P\left(1/\sqrt{n}+1/\sqrt{p_{\min}} \right)$.
  \item $\max_{i\in [n],j\in [p]}|\hat{u}_{ij}-u_{ij}|=O_P\left(\sqrt{(\log n)/n}+n^{1/4}/\sqrt{p_{\min}} \right)$.
  \end{enumerate}
\end{lem}
\begin{proof}
The results follow given Lemma \ref{lemma:C1} and the fact that the convergence rate depends on the slowest one among all $M$ modalities.
\end{proof}

\begin{lem} \label{lemma:C3}
Suppose the conditions of Theorem \ref{thm:3} hold. Then the following results hold:
\begin{enumerate}[(a)]
  \item $\ltwonorm{\Psib_n^{-1/2}\A}=O_P\left(1 \right)$.
  \item $\ltwonorm{\sqrt{n}\Psib_n^{-1/2}\h_n}=O_P\left(\sqrt{r} \right)$.
  \item $\ltwonorm{\Psib_n^{1/2}(\A\Omegabh_T\A')^{-1}\Psib_n^{1/2}-\I}=O_P\left((t+s_a)/\sqrt{n} \right)$.
  \item $\ltwonorm{\Psib_n^{-1}-\Psib^{-1}}=o_P\left(1 \right)$.
\end{enumerate}
\end{lem}
\begin{proof} 
The results of (a)--(c) follow from (5.4), (5.5) and (5.6) of \cite{shi2019linear}.

For (d), we have $\Psib_n^{-1}-\Psib^{-1}=\Psib_n^{-1}(\Psib-\Psib_n)\Psib^{-1}$. Therefore, $\ltwonorm{\Psib_n^{-1}-\Psib^{-1}}\asymp \ltwonorm{\Psib-\Psib_n}$. Moreover,
\begin{eqnarray*}
    \ltwonorm{\Psib-\Psib_n} = \ltwonorm{\A\{\Omegab_T-(\K_n^{-1})_T \}\A'} \lesssim \ltwonorm{\Omegab_T-(\K_n^{-1})_T}
    \leq \ltwonorm{\Sigmab_{u,\ts}^{-1}-\K_n^{-1}},
\end{eqnarray*}
where $\Sigmab_{u,\ts}$ is the submatrix of $\Sigmab_u$ with rows and columns in $\ts$. By the sub-Gaussian assumption on $\u$, we have $\supnorm{\K_{n}-\Sigmab_{u,\ts}}= O_P\left(\sqrt{\{\log (t+s_a) \}/n} \right)$. Then,
\begin{equation*}
    \ltwonorm{\K_n-\Sigmab_{u,\ts}}\leq (t+s_a) \supnorm{\K_{n}-\Sigmab_{u,\ts}} =o_P\left(1 \right). 
\end{equation*}
Consequently, 
\begin{equation*}
    \ltwonorm{\Psib-\Psib_n} \lesssim \ltwonorm{\Sigmab_{u,\ts}^{-1}-\K_n^{-1}} \leq \ltwonorm{\K_n^{-1}}\ltwonorm{\Sigmab_{u,\ts}^{-1}} \ltwonorm{\K_n-\Sigmab_{u,\ts}}=o_P\left(1 \right).  
\end{equation*}
This completes the proof. 
\end{proof}

\begin{lem} \label{lemma:C4} Let $\{\X_i \}_{i=1}^n$ denote independent $p$-dimensional random vectors, with $\E(\X_i)=\zero$ and $\sum_{i=1}^n \var(\X_i)=\I_p$. Let $\Z$ denote a $p$-dimensional multivariate normal vector, with mean $\zero$ and covariance matrix $\I_p$. Then,
\begin{equation*}
\sup_C \left|P\left(\sum_{i=1}^n \X_i\in C\right)-P(\Z \in C) \right| \leq c_0p^{1/4} \sum_{i=1}^n \E( \ltwonorm{\X_i}^3),
\end{equation*}
for some constant $c_0$, where the supremum is taken over all convex subsets in $\Rcal^p$.
\end{lem}
\begin{proof}
The result follows from Lemma S6 of \cite{shi2019linear}, which was originally given in Theorem 1 of \cite{bentkus2005lyapunov}.
\end{proof}

\begin{lem} \label{lemma:C5}
Let $\chi^2(r,\gamma)$ denote a $\chi^2$ random variable with $r$ degrees of freedom and the non-centrality parameter $\gamma$. Then,
\begin{equation*}
\lim_{\epsilon\to 0+} \sup_{r \geq 1, \gamma \geq 0} |P(\chi^2(r,\gamma) \leq x+r\varepsilon)-P(\chi^2(r,\gamma)\leq x-r\varepsilon)| \to 0. 
\end{equation*}
\end{lem}
\begin{proof}
The result follows from Lemma S7 of \cite{shi2019linear}.
\end{proof}

\begin{lem} \label{lemma:B1}
Suppose Conditions \ref{con:1}--\ref{con:3} hold, and $\lambda_1\asymp \sqrt{(\log p_{-m})/n}+1/\sqrt{p_m}$. Then,  
\begin{eqnarray*}
    \lonenorm{\betahnom-\betasnom}=O_P\left(\ss_{-m} \left\{\sqrt{(\log p_{-m})/n}+1/\sqrt{p_m} \right\} \right).
\end{eqnarray*}
\end{lem}
\begin{proof}
Recall that
\begin{equation} \label{eq:42}
    (\betabh_{-m},\thetabh_m)=\argmin_{(\betanom,\thetab_m)} \frac{1}{2n} \sum_{i=1}^n (y_i-\x_{i,-m}'\betanom-\hat{\f}'_{i,m}\thetab_m)^2 + \lambda_1 \lonenorm{\betab}.
\end{equation}
By Lemma \ref{lemma:C1}, there exists a nonsingular matrix $\H_m\in\Rcal^{K_m\times K_m}$ such that $\tilde{\F}_m=\hat{\F}_m\H_m$ is a consistent estimator of $\F_m$. We note that (\ref{eq:42}) is equivalent to 
\begin{equation} \label{eq:43}
    (\tilde{\betab}_{-m},\tilde{\thetab}_m)=\argmin_{(\betanom,\thetab_m)} \frac{1}{2n} \sum_{i=1}^n (y_i-\x_{i,-m}'\betanom-\tilde{\f}'_{i,m}\thetab_m)^2 + \lambda_1 \lonenorm{\betab},
\end{equation}
where $\tilde{\f}'_{i,m}$ is the $i$th row of $\tilde{\F}_m$. Then, solving (\ref{eq:43}) is equivalent as replacing $\tilde{\f}_{i,m}$ with $\f_{i,m}$, which becomes a standard $M$-estimation problem.

Let $\Q=(\X_{-m}, \F_m)$, $\hat{\Q}=(\X_{-m},\Fh_m)$, and $\tilde{\Q}=\Q\tilde{\H}$, where $\tilde{\H}=\text{diag}(\I_{p_{-m}},\H_m)\in \Rcal^{q_{-m}\times q_{-m}}$ is a block-diagonal matrix, and $q_{-m}=p_{-m}+K_m$. Let $\varthetab=(\betab'_{-m},\thetab'_m)'\in \Rcal^{q_{-m}}$, $\hat{\varthetab}=(\betabh'_{-m},\thetabh'_m)'\in \Rcal^{q_{-m}}$ denote the solution of (\ref{eq:42}), $\tilde{\varthetab}=\tilde{\H}^{-1}\hat{\varthetab}\in \Rcal^{q_{-m}}$ and $\tilde{\varthetab}^{\ast}=\tilde{\H}^{-1}(\betab_{-m}^{\ast'},\thetab_m^{\ast'})'\in \Rcal^{q_{-m}}$.  By direct calculation, we can verify that $\betabh_{-m}=\hat{\varthetab}_{[p_{-m}]}=\tilde{\varthetab}_{[p_{-m}]}$, where $\tilde{\varthetab}=(\tilde{\betab}'_{-m},\tilde{\thetab}'_m)'$  solves (\ref{eq:43}). Then, it follows that
\begin{equation*}
    \lonenorm{\betabh_{-m}-\betabs_{-m}}=\lonenorm{\tilde{\varthetab}_{[p_{-m}]}-\tilde{\varthetab}^{\ast}_{[p_{-m}]}} \leq 
    \lonenorm{\tilde{\varthetab}-\tilde{\varthetab}^{\ast}}. 
\end{equation*}
  
To bound $\lonenorm{\tilde{\varthetab}-\tilde{\varthetab}^{\ast}}$, we turn to bound $\supnorm{\nabla\ell(\varthetab)}$, and check the restricted eigenvalue condition on $\nabla^2\ell(\varthetab)$, where $\ell(\varthetab)=(2n)^{-1} \sum_{i=1}^n(y_i-\x_{i,-m}'\betanom-\tilde{\f}'_{i,m}\thetab_m)^2$. 

To bound $\supnorm{\nabla\ell(\varthetab)}$, we aim to show that
\begin{equation} \label{eq:44}
    \ssupnorm{\nabla\ell(\varthetab)}=\ssupnorm{\frac{1}{n} \sum_{i=1}^n \tilde{\Q}_i\epsilon_i} = O_P\left(
      \sqrt{\frac{\log p_{-m}}{n}}+\frac{1}{\sqrt{p_m}} \right), 
\end{equation}
where $\tilde{\Q}_i$ is the $i$th row of $\tilde{\Q}$. Indeed,
\begin{equation*}
    \ssupnorm{\frac{1}{n} \sum_{i=1}^n \tilde{\Q}_i\epsilon_i} \leq  \max_{j\in [q_{-m}]}
    \left| \frac{1}{n} \sum_{i=1}^n Q_{ij}\epsilon_i\right| + \max_{j\in [q_{-m}]} \left|\frac{1}{n}\sum_{i=1}^n
      (\tilde{Q}_{ij}-Q_{ij})\epsilon_i\right|. 
\end{equation*}
By Condition \ref{con:1} and Bernstein inequality,
\begin{equation*}
    P \left(\left|\frac{1}{n} \sum_{i=1}^n Q_{ij}\epsilon_i \right| > C \sqrt{\frac{\log q_{-m}}{n}} \right) \leq q_{-m}^{-2} \text{ for all } j\in [q_{-m}].
\end{equation*}
Then, by the union bound, 
\begin{equation} \label{eq:45}
    \max_{j\in [q_{-m}]} \left|\frac{1}{n} \sum_{i=1}^n Q_{ij}\epsilon_i\right|=O_P\left(\sqrt{(\log q_{-m})/n} \right) = O_P\left(\sqrt{(\log p_{-m})/n} \right),
\end{equation}
where the last equality is due to the fact that, since $K_m$ is fixed, $q_{-m}\asymp p_{-m}$. We choose to present the results using $p_{-m}$ in order to unify the presentation. Then by Cauchy-Schwatz inequality,
\begin{equation*}
    \max_{j\in [q_{-m}]} \left|\frac{1}{n}\sum_{i=1}^n (\tilde{Q}_{ij}-Q_{ij})\epsilon_i\right| \leq \max_{j\in [q_{-m}]} \left(\frac{1}{n} \sum_{i=1}^n (\tilde{Q}_{ij}-Q_{ij})^2 \right)^{1/2} \left(\frac{1}{n} \sum_{i=1}^n \epsilon_i^2 \right)^{1/2} 
\end{equation*}
If we use the $m$th modality to obtain $\fh_{i,m}$, it follows from Lemma \ref{lemma:C1} that  
\begin{eqnarray*}
    \max_{k\in [K_m]} (1/n) \sum_{i=1}^n |(\hat{\F}_m\H_m)_{ik} - f_{i,m_k}|^2=O_P\left(1/n+1/p_m \right),
\end{eqnarray*}
where $(\hat{\F}_m\H_m)_{ik}$ is the $(i,k)$th element of $\hat{\F}_m\H_m$ and $f_{i,m_k}$ is the $k$th element of $\f_{i,m}$. This implies that
\begin{equation*}
    \max_{j\in [q_{-m}]} \left(\frac{1}{n} \sum_{i=1}^n (\tilde{Q}_{ij}-Q_{ij})^2 \right)^{1/2}= O_P\left(1/\sqrt{n} + 1/\sqrt{q_m}\right)=O_P\left(1/\sqrt{n} + 1/\sqrt{p_m}\right).
\end{equation*}
Since $(n^{-1} \sum_{i=1}^n \epsilon_i^2)^{1/2}=O_P\left(1 \right)$, we have
\begin{equation} \label{eq:46} \max_{j\in [q_{-m}]} \left|\frac{1}{n}\sum_{i=1}^n
      (\tilde{Q}_{ij}-Q_{ij})\epsilon_i\right|=O_P\left(1/\sqrt{n}+1/\sqrt{p_m} \right).
\end{equation}
Combining (\ref{eq:45}) and (\ref{eq:46}) together proves (\ref{eq:44}).

To check the restricted eigenvalue condition, it follows from Condition \ref{con:3} and the factor decomposition (\ref{eq:1}) that $\lambda_{\min}(\E (\x_{-m}^{\otimes 2}))\geq \lambda_{\min}(\E (\u_{-m}^{\otimes 2}))>c$. In addition, the sub-Gaussian assumptions on $\f_{-m}$ and $\u_{-m}$ imply that $\x_{-m}$ is also sub-Gaussian. Since $\nabla_{\betab_{S_{-m}}\betab_{S_{-m}}}^2\ell(\varthetab)= n^{-1} \sum_{i=1}^n \x_{i,S_{-m}}^{\otimes 2}$, where $S_{-m}=\{j \in [p_{-m}]:\beta_j^{\ast}\neq 0 \}$. Then, it follows from Proposition 1 of \cite{raskutti2011minimax} that the restricted eigenvalue condition holds with high probability.
  
Given that both (\ref{eq:44}) and the restricted eigenvalue condition hold, the rest of the proof follows the standard arguments of the high-dimensional $M$-estimator \citep[Theorem 1]{negahban2012unified}. A relevant proof in the context of factor model can be found in Theorem 4.2 of \cite{fan2016factor}. This completes the proof. 
\end{proof}

\begin{lem} \label{lemma:B3} Suppose Conditions \ref{con:1}--\ref{con:3} hold. Then, 
\begin{equation*}
    \ssupnorm{\frac{1}{n} \sum_{i=1}^n \x_{i,-m}(\hat{f}_{i,m_k}-\x_{i,-m}'\wh_k)} = O_P\left( \sqrt{\frac{\log
          p_{-m}}{n}}\left(1\vee \frac{n^{1/4}}{\sqrt{p_m}}\right) \right). 
\end{equation*}
\end{lem}
\begin{proof}
By definition, it suffices to show that, if we choose $\lambda_2=C \sqrt{(\log p_{-m})/n}\{1\vee (n^{1/4}/\sqrt{p_m}) \}$ for some large enough constant $C$, $\ws_k$ is in the feasible set with high probability; i.e.,
\begin{equation} \label{eq:47}
    \frac{1}{n} \ssupnorm{\sum_{i=1}^n \x_{i,-m}(\hat{f}_{i,m_k}-\x'_{i,-m}\ws_k)} \leq C \sqrt{(\log p_{-m})/n}\{1\vee (n^{1/4}/\sqrt{p_m})  \}. 
\end{equation}
We have 
\begin{equation*}
\frac{1}{n} \sum_{i=1}^n \x_{i,-m}(\hat{f}_{i,m_k}-\x_{i,-m}'\ws_k) = \frac{1}{n} \sum_{i=1}^n \x_{i,-m}(\hat{f}_{i,m_k}-f^{\dagger}_{i,m_k}) + \frac{1}{n} \sum_{i=1}^n \x_{i,-m}(f^{\dagger}_{i,m_k}-\x_{i,-m}'\ws_k), 
\end{equation*}
where $\f^{\dagger}_{i,m}=\H_m\f_{i,m}$ for some non-singular matrix $\H_m\in \Rcal^{K_m\times K_m}$, and $f^{\dagger}_{i,m_k}$ is the $k$th element of $\f^{\dagger}_{i,m}$. The sub-Gaussian assumption in Condition \ref{con:1} implies that $X_{ij}f^{\dagger}_{i,m_k}$ and $X_{ij}\x'_{i,-m}\ws_k$ are sub-exponential. Then, by Bernstein inequality and the union bound,
\begin{equation*} \label{eq:48}
    \ssupnorm{\frac{1}{n} \sum_{i=1}^n \x_{i,-m}(f^{\dagger}_{i,m_k}-\x_{i,-m}'\ws_k)} =O_P\left( \sqrt{\frac{\log p_{-m}}{n}} \right). 
\end{equation*}
On the other hand, we have 
\begin{eqnarray*}
    & & \ssupnorm{\frac{1}{n} \sum_{i=1}^n \x_{i,-m}(\hat{f}_{i,m_k}-f^{\dagger}_{i,m_k})} 
        \leq \left(\max_{i\leq n} |\hat{f}_{i,m_k}-f^{\dagger}_{i,m_k}|\right) \ssupnorm{\frac{1}{n} \sum_{i=1}^n \x_{i,-m}} \\
    & \leq & \left(\max_{i\leq n} \ltwonorm{\hat{\f}_{i,m}-\f^{\dagger}_{i,m}}\right) \ssupnorm{\frac{1}{n} \sum_{i=1}^n \x_{i,-m}} 
             = O_P\left( \frac{1}{\sqrt{n}}+\frac{n^{1/4}}{\sqrt{p_m}} \right)O_P\left(\sqlogqn \right),
\end{eqnarray*}
where the last equality follows from Lemma \ref{lemma:C1}. This completes the proof. 
\end{proof}

\begin{lem} \label{lemma:B2}
Suppose Conditions \ref{con:1}--\ref{con:3} hold, and $\lambda_2 \asymp \sqrt{(\log p_{-m})/n}\{1\vee (n^{1/4}/\sqrt{p_m}) \}$. Then,
\begin{eqnarray*}
\lonenorm{\wh_k-\ws_k}=O_P\left(\ss_k \left\{\sqlogqn \left(1\vee \frac{n^{1/4}}{\sqrt{p_m}} \right) \right \} \right).
\end{eqnarray*}
\end{lem}
\begin{proof}
Recall that, for the $k$th column of $\W$, we solve
\begin{equation*}
    \wh_k=\argmin \lonenorm{\w_k}, \quad \text{such that} \quad \ssupnorm{\frac{1}{n} \sum_{i=1}^n \x_{i,-m}(\hat{f}_{i,m_k}-\x_{i,-m}'\w_k)}\leq \lambda_2,
\end{equation*}
where $\hat{f}_{i,m_k}$ is the $k$th element of $\fh_{i,m}$. Let $S_k=\supp(\ws_k)$, where $\ws_k$ is the $k$th column of $\Wbs=\E(\x_{i,-m}^{\otimes 2})^{-1}\E(\x_{i,-m}\f_{i,m}')$. Then, we have $\lonenorm{\ws_{S_k}}\geq \lonenorm{\wh_{S_k}}+\lonenorm{\wh_{S_k^c}}$. By the triangle inequality, we have $\lonenorm{\wh_{S_k^c}}\geq \lonenorm{\ws_{S_k}}- \lonenorm{\wh_{S_k}-\ws_{S_k}}$. Let $\Deltabh_k=\wh_k-\ws_k$. Then, by noting that $\lonenorm{\ws_{S_k^c}}=0$, we have $\lonenorm{\Deltabh_{S_k}}\geq \lonenorm{\Deltabh_{S_k^c}}$. It follows from Lemma \ref{lemma:B3} that
\begin{eqnarray*}
\ssupnorm{\frac{1}{n} \sum_{i=1}^n \x_{i,-m}(\hat{f}_{i,m_k}-\x_{i,-m}'\wh_k)} = O_P\left( \sqlogqn \left(1\vee \frac{n^{1/4}}{\sqrt{p_m}}\right)\right). 
\end{eqnarray*}
Denote $\H_x= n^{-1} \sum_{i=1}^n \x_{i,-m}^{\otimes 2}$, $\H_{xf}=n^{-1} \sum_{i=1}^n \x_{i,-m}\hat{f}_{i,m_k}$, and $\Deltab_k=\wh_k-\ws_k$. Then, 
\begin{equation*}
    \ssupnorm{\H_x\Deltabh_k} \leq \ssupnorm{\H_{xf}-\H_x\wh_k} + \ssupnorm{\H_{xf} - \H_x\ws_k} = O_P\left( \sqlogqn \left(1\vee \frac{n^{1/4}}{\sqrt{p_m}}\right)\right).
\end{equation*}
Together with $\lonenorm{\Deltabh_k}\leq 2 \lonenorm{\Deltabh_{S_k}}\leq 2 \sqrt{\ss_k} \ltwonorm{\Deltabh_k}$, we have
\begin{eqnarray} \label{eq:49}
\Deltabh_k\H_x\Deltabh_k &\leq & \lonenorm{\Deltabh_k} \ssupnorm{\H_x\Deltabh_k} = O_P\left( \lonenorm{\Deltabh_k} \sqlogqn \left(1\vee \frac{n^{1/4}}{\sqrt{p_m}}\right) \right) \nonumber \\
& = & O_P\left( \ltwonorm{\Deltabh_k}\sqrt{\frac{\ss_k \log p_{-m}}{n}} \left(1\vee \frac{n^{1/4}}{\sqrt{p_m}}\right) \right).
\end{eqnarray}
Note that Condition \ref{con:1} and (\ref{eq:1}) imply that $\x_{-m}$ is sub-Gaussian. In addition, Condition \ref{con:3} implies that $\lambda_{\min}(\E (\x_{i,-m}^{\otimes 2}))>c$. Then, it follows from Proposition 1 of \cite{raskutti2011minimax} that the restricted eigenvalue condition holds for $\H_{x}$ with high probability, i.e., $\Deltabh_k\H_x\Deltabh_k\geq \kappa \ltwonorm{\Deltabh_k}^2$ for some $\kappa>0$ and all $\Deltabh_k$,  such that  $\lonenorm{\Deltabh_{S_k^c}}\leq \lonenorm{\Deltabh_{S_k}}$. This result, together with (\ref{eq:49}), implies that
\begin{equation*}
\lonenorm{\Deltabh_k}\leq 2 \sqrt{\ss_k} \ltwonorm{\Deltabh_k}=O_P\left( \ss_k\sqlogqn \left(1\vee \frac{n^{1/4}}{\sqrt{p_m}}\right) \right).
\end{equation*}
This completes the proof. 
\end{proof}

\begin{lem}\label{lemma:B6}
Suppose the conditions of Theorem \ref{thm:4} hold. Then, uniformly for all $\betabs\in \N$,
\begin{equation*}
\ltwonorm{\sqrt{n}(\Itildehalf-\Itruehalf)\tilde{\S}(\betahnom,\zero)}=o_P(1).
\end{equation*}
\end{lem}
\begin{proof}
First, we note that
 \begin{align*}
&\hspace{3ex}\ltwonorm{\Itildehalf-\Itruehalf}
= \ltwonorm{\Itruehalf({\I}^{\ast^{1/2}}_{\gammab_m|\betab_{-m}}-\tilde{\I}_{\gammab_m|\betab_{-m}}^{1/2}) \tilde{\I}_{\gammab_m|\betab_{-m}}^{-1/2}} \\
& \lesssim \ltwonorm{{\I}^{\ast^{1/2}}_{\gammab_m|\betab_{-m}}-\tilde{\I}_{\gammab_m|\betab_{-m}}^{1/2}} 
\leq \ltwonorm{{\I}^{\ast}_{\gammab_m|\betab_{-m}}-\tilde{\I}_{\gammab_m|\betab_{-m}}}^{1/2}
\leq \lonenorm{{\I}^{\ast}_{\gammab_m|\betab_{-m}}-\tilde{\I}_{\gammab_m|\betab_{-m}}}^{1/2}\\
& \leq \sqrt{K_m} \supnorm{{\I}^{\ast}_{\gammab_m|\betab_{-m}}-\tilde{\I}_{\gammab_m|\betab_{-m}}}^{1/2} 
= o_P(1), 
\end{align*}
where $\lesssim$ follows from Condition \ref{con:21}, and the last equality follows from (\ref{eq:103}) and that $K_m$ is fixed. Lemma \ref{lemma:B7} implies that $\ltwonorm{\sqrt{n} \tilde{\S}(\betahnom,\zero)}=o_P\left(1 \right)$. Therefore, 
\begin{equation*}
      \ltwonorm{\sqrt{n}(\Itildehalf-\Itruehalf)\tilde{\S}(\betahnom,\zero)}\leq
      \ltwonorm{\Itildehalf-\Itruehalf}\ltwonorm{\sqrt{n}\tilde{\S}(\betahnom,\zero)}=o_P\left(1 \right).
\end{equation*}
This completes the proof. 
\end{proof}

\begin{lem} \label{lemma:B7}
Suppose the conditions of Theorem \ref{thm:4} hold. Then, uniformly for all $\betabs\in \N$, 
\begin{equation*}
\sqrt{n} \Itruehalf \{\tilde{\S}(\betahnom,\zero)-\S(\betasnom,\zero) \}=o_P(1). 
\end{equation*}
\end{lem}
\begin{proof}
By (\ref{eq:19}), we have that 
\begin{equation*} 
\begin{split}
\S(\betasnom,\zero)-\tilde{\S}(\betahnom,\zero)
= \; & \frac{1}{n\sigmae^2}(\W^{\ast}-\hat{\W})' \X_{-m}'(\Y-\X_{-m}\betasnom) \\
& + \frac{1}{n\sigmae^2}(\hat{\F}_m'\X_{-m}-\hat{\W}'\X_{-m}'\X_{-m})(\betabh_{-m}-\betabs_{-m})
\equiv I + II.
\end{split}
\end{equation*}
For the term $I$, under $H_a$, $\Y-\X_{-m}\betasnom=\X_m\betasm+\epsilonb$. Therefore,
\begin{align*}
&\hspace{3ex}\supnorm{\frac{1}{n\sigmae^2}\X_{-m}'(\Y-\X_{-m}\betasnom)} 
= \supnorm{\frac{1}{n\sigmae^2}\X_{-m}'(\X_{m}\betasm+\epsilonb)}\\
&\leq \supnorm{\frac{1}{n\sigmae^2}\X_{-m}' \epsilonb} + \supnorm{\frac{1}{n\sigmae^2}\X_{-m}'(\F_m\gammabs_m+\U_m\betasm)} = O_P\left(\sqrt{\frac{\log p_{-m}}{n}} \right),
\end{align*}
where the last equation follows from the sub-Gaussian assumption in Condition \ref{con:1}, and an application of Bernstein inequality. A careful inspection of the proof of Lemma \ref{lemma:B2} shows that the lemma still holds under $H_a$. Therefore, following the same argument as in (\ref{eq:104}), for each $k\in [K_m]$, we have
\begin{equation} \label{eq:105}
\begin{split}
|(n\sigmae^2)^{-1}(\ws_k-\wh_k)'\X_{-m}'(\Y-\X_{-m}\betabs_{-m})| 
& \leq \lonenorm{\ws_k-\wh_k} \supnorm{(n\sigmae^2)^{-1}(\Y-\X_{-m}\betabs_{-m})} \\
& = o_P\left(1/\sqrt{n} \right).
\end{split}
\end{equation}
Therefore, $I=o_P\left(1/\sqrt{n} \right)$.
  
For the term $II$, we bound $\lonenorm{\betahnom-\betasnom}$ and
$\supnorm{(n\sigmae^2)^{-1}\hat{\F}_m'\X_{-m}-\hat{\W}'\X_{-m}'\X_{-m}}$ respectively.  To bound
$\lonenorm{\betahnom-\betasnom}$, under $H_a$, we only need to replace $\epsilon_i$ in (\ref{eq:44}) with
$\epsilon_i+\U_{i,m}\betasm$. Due to the sub-Gaussian assumption of $\U_{i,m}$ in Condition \ref{con:1}, and the fact that
$\epsilon_i$ and $\U_{i,m}$ are uncorrelated, the bounds we have established in (\ref{eq:45}) and (\ref{eq:46}) still hold. Therefore, uniformly for all $\betabs\in \N$, $\lonenorm{\betahnom-\betasnom}$ has the same upper bound as the one established in Lemma \ref{lemma:B1}. In addition, the bound we have established for $\supnorm{(n\sigmae^2)^{-1}\hat{\F}_m'\X_{-m}-\hat{\W}'\X_{-m}'\X_{-m}}$ also holds uniformly for all $\betabs \in \N$. Then, it follows from (\ref{eq:105}) that $II=o_P\left(1/\sqrt{n} \right)$.

Combining the bounds of the terms $I$, $II$, and Condition \ref{con:21} completes the proof.
\end{proof}
\smallskip

\begin{lem} \label{lemma:B8}
Suppose the conditions of Theorem \ref{thm:4} hold. Then, uniformly for all $\betabs\in \N$, 
\begin{equation*}
\S(\betabs,\gammasm)-\S(\betasnom,\zero)-\Itrue\gammasm=o_P\left(n^{-1/2} \right). 
\end{equation*}
\end{lem}
\begin{proof}
By definition, we have
\begin{align*}
          & \hspace{3ex}\sqrt{n}\{\S(\betabs,\gammasm)-\S(\betasnom,\zero)-\Itrue\gammasm \}\\
          &=\frac{1}{\sqrt{n}\sigmae^2} \sum_{i=1}^n \x_{i,m}'\betasm(\f_{i,m}-{\W^{\ast}}'\x_{i,-m})-\Itrue\gammasm\\
          &=\frac{1}{\sqrt{n}\sigmae^2} \sum_{i=1}^n
            [(\f_{i,m}-{\W^{\ast}}'\x_{i,-m})\f_{i,m}'\gammasm-\E\{(\f_{i,m}-{\W^{\ast}}'\x_{i,-m})\f'_{i,m}\gammasm
            \}]\\
          &\hspace{3ex}+ \frac{1}{\sqrt{n}\sigmae^2} \sum_{i=1}^n \U_{i,m}'\betasm(\f_{i,m}-{\W^{\ast}}'\x_{i,-m}) 
          \equiv I+II. 
\end{align*}

For the term $I$, its $k$th element equals
\begin{equation*}
 \frac{1}{\sqrt{n}\sigmae^2} \sum_{i=1}^n \left[\left( f_{i,m_k}-{{\w}_k^{\ast}}'\x_{i,-m} \right) \f_{i,m}'\gammasm - \E\left\{\left(\f_{i,m}-{{\w}_k^{\ast}}'\x_{i,-m} \right)\f'_{i,m} \gammasm \right\} \right].
\end{equation*}
By the sub-Gaussian assumptions on $f_{i,m_k}$ and ${\ws_k}'\x_{i,-m}$ in Condition \ref{con:1}, and the standard concentration inequality \citep[e.g.,][Lemma H.2]{ning2017general}, the $k$th element of $I$ is bounded by $O_P\left(\ltwonorm{\c_{m_n}} \sqrt{\log n} \right)=o_P\left(1 \right)$ for all $k\in [K_m]$, and this bound is uniform for all $\betabs\in \N$. 

For the term $II$, its $k$th element satisfies that, uniformly for all $\betabs\in \N$,
\begin{equation*}
\frac{1}{\sqrt{n}\sigmae^2} \sum_{i=1}^n \U_{i,m}'\betasm(f_{i,m_k}-{\ws_k}'\x_{i,-m})=O_P\left(\ltwonorm{\b_{m_n}} \sqrt{\log n}\right) = o_P\left(1 \right). 
\end{equation*}

Combining the results for the terms $I$ and $II$ completes the proof. 
\end{proof}

\begin{lem} \label{lemma:B4}
Suppose the conditions of Theorem \ref{thm:2} hold. Then, 
\begin{equation*}
    \supnorm{\RR}=O_P\left(\sqrt{\frac{\log p}{n}}\left(\frac{1}{\sqrt{n}}+\frac{n^{1/4}}{\sqrt{p_{\min}}} \right) \right), 
\end{equation*}
where $\RR = n^{-1} \left\{ (\Uh-\U)'\U\betabs+\Uh'(\F\gammabs-\Fh\gammabh_a)+(\Uh-\U)'\epsilonb \right\}$.
\end{lem}
\begin{proof}
First, we bound $n^{-1}\Uh'(\F\gammabs-\Fh\gammabh_a)$. By Lemma \ref{lemma:C2}, $\supnorm{\Uh-\U}=o_P\left(1\right)$. Therefore,
\begin{eqnarray}
    & & \ssupnorm{\frac{1}{n} \Uh'(\F\gammabs-\Fh\gammabh_a)} \lesssim  \ssupnorm{\frac{1}{n}\U'(\F\gammabs-\Fh\gammabh_a)} \nonumber \\
    & \leq & \ssupnorm{\frac{1}{n}\U'(\F\gammabs-\F\H\H'\gammabs)} + \ssupnorm{\frac{1}{n} \U'(\F\H\H'\gammabs-\Fh\gammabh_a)} \nonumber \\
    & \leq & \ssupnorm{\frac{1}{n}\U'\F(\I_K-\H\H')\gammabs}+\ssupnorm{\frac{1}{n} \U'(\F\H\H'\gammabs-\Fh\gammabh_a)}. \label{eq:50}
\end{eqnarray}
Let $\gammat=(\I_K-\H\H')\gammabs$. Then, $\ltwonorm{\gammat} \leq \ltwonorm{\I_K-\H\H'}\ltwonorm{\gammabs}$. By Lemma \ref{lemma:C2}, 
\begin{equation*}
    \ltwonorm{\I_K-\H\H'}=O_P\left(1/\sqrt{n}+1/\sqrt{p_{\min}} \right). 
\end{equation*}
Since $\ltwonorm{\gammabs}^2\lesssim \var(\f'\gammabs)\lesssim \sigma^2_y=O(1)$, we have $\supnorm{\gammat}\leq \ltwonorm{\gammat}=O_P\left(1/\sqrt{n}+1/\sqrt{p_{\min}} \right)$. Then,
\begin{eqnarray*}
    \ssupnorm{\frac{1}{n} \U'\F(\I_K-\H\H')\gammabs} =\max_{j\in [p]} \Big|\frac{1}{n} \sum_{i=1}^n U_{ij} \sum_{k=1}^K f_{ik}\tilde{\gamma}_k\Big|.
\end{eqnarray*}
By Condition \ref{con:1}, $K$ is fixed, and $\ltwonorm{\gammat}=O_P\left(1/\sqrt{n}+1/\sqrt{p_{\min}} \right)$, we have that $\sum_{k=1}^K f_{ik}\tilde{\gamma}_k$ is sub-Gaussian with variance bounded by $O\left(1/n+1/p_{\min} \right)$. Moreover, $U_{ij}$ is also sub-Gaussian and uncorrelated with $f_{ik}$ for any $k\in [K]$. Then, $U_{ij}\sum_{k=1}^K f_{ik}\tilde{\gamma}_k$ is sub-exponential. Applying Bernstein inequality, we have
\begin{equation} \label{eq:51}
    \ssupnorm{\frac{1}{n} \U'\F(\I_K-\H\H')\gammabs}=O_P\left(\sqrt{\frac{\log p}{n}}\left(\frac{1}{\sqrt{n}}+\frac{1}{\sqrt{p_{\min}}}\right)\right).
\end{equation}

On the other hand, we have
\begin{equation} \label{eq:52}
    \ssupnorm{\frac{1}{n}\U'(\F\H\H'\gammabs-\Fh\gammabh_a)} \leq \ssupnorm{\frac{1}{n}\U'(\F\H-\Fh)\H'\gammabs} + \ssupnorm{\frac{1}{n}\U'\Fh(\H'\gammabs-\gammabh_a)}. 
\end{equation}
For the first term in \eqref{eq:52}, we have
\begin{eqnarray}
    & & \ssupnorm{\frac{1}{n}\U'(\F\H-\Fh)\H'\gammabs} 
        \leq \left(\max_{i\in [n]}\sum_{k=1}^K |(\F\H-\Fh)_{ik}(\H'\gammabs)_k|\right) \ssupnorm{\frac{1}{n} \sum_{i=1}^n \U_i} \nonumber \\
    & \leq & \ltwonorm{\H'\gammabs}\left(\max_{i\in [n]} \ltwonorm{\fh_i-\H\f_i}\right) \ssupnorm{\frac{1}{n}\sum_{i=1}^n\U_i} \nonumber \\
    & = & O_P\left(\sqrt{\frac{\log p}{n}}\right)O_P\left(\frac{1}{\sqrt{n}}+\frac{n^{1/4}}{\sqrt{p_{\min}}} \right), \label{eq:53} 
  \end{eqnarray}
  where the last equality is implied by Lemma \ref{lemma:C1}, and the fact that $\ltwonorm{\H'\gammabs}\leq \ltwonorm{\H'}\ltwonorm{\gammabs}=O(1)$. For the second term in \eqref{eq:52}, we have
\begin{eqnarray*}
    \ssupnorm{\frac{1}{n}\U'\Fh(\H'\gammabs-\gammabh_a)}
    \lesssim \ssupnorm{\frac{1}{n}\U'\F(\H'\gammabs-\gammabh_a)}
    \leq \left(\max_{k\in[K]}|(\H\gammabs-\gammabh_a)_k| \right) \sum_{k=1}^K \ssupnorm{\frac{1}{n}\U_if_{ik}}
\end{eqnarray*}
Since $\gammabh_a\in \mathcal{M}$, we have $\max_{k\in [K]} |(\H\gammabs-\gammabh_a)_k|=O_P\left(\delta_n \right)$. By Condition \ref{con:1}, for each $k\in [K]$, $U_{ij}f_{ik}$ is sub-exponential. Then $\supnorm{n^{-1}\U_if_{ik}}=O_P\left(\sqrt{(\log p)/n} \right)$. Since $K$ is fixed, $\sum_{k=1}^K\supnorm{n^{-1}\U_if_{ik}}=O_P\left(\sqrt{(\log p)/n} \right)$. Then,
\begin{equation*}
    \ssupnorm{\frac{1}{n}\U'\Fh(\H'\gammabs-\gammabh_a)}=O_P\left(\delta_n \right)O_P\left(\sqrt{\frac{\log p}{n}} \right) =
    O_P\left(\sqrt{\frac{\log p}{n}}\left(\frac{1}{\sqrt{n}}+\frac{n^{1/4}}{\sqrt{p_{\min}}} \right) \right). 
\end{equation*}
Therefore,
\begin{equation} \label{eq:54}
    \ssupnorm{\frac{1}{n}\U'(\F\H\H'\gammabs-\Fh\gammabh_a)} = O_P\left(\sqrt{\frac{\log p}{n}}\left(\frac{1}{\sqrt{n}}+\frac{n^{1/4}}{\sqrt{p_{\min}}} \right) \right).  
\end{equation}
Then, it follows from (\ref{eq:50}), (\ref{eq:51}) and (\ref{eq:54}) that
\begin{equation} \label{eq:55}
    \ssupnorm{\frac{1}{n} \Uh'(\F\gammabs-\Fh\gammabh_a)} = O_P\left(\sqrt{\frac{\log p}{n}}\left(\frac{1}{\sqrt{n}}+\frac{n^{1/4}}{\sqrt{p_{\min}}} \right) \right).
\end{equation}
For $\supnorm{n^{-1}(\Uh-\U)'\epsilonb}$, we have
\begin{equation} \label{eq:56}
    \supnorm{n^{-1}(\Uh-\U)'\epsilonb} \leq \left(\max_{ij}|\hat{u}_{ij}-u_{ij}| \right) \left|\frac{1}{n} \sum_{i=1}^n \epsilon_i\right|
    = O_P\left(\frac{1}{\sqrt{n}}\left(\sqrt{\frac{\log n}{n}}+\frac{n^{1/4}}{\sqrt{p_{\min}}} \right)\right).  
\end{equation}
where the last equality follows from Lemma \ref{lemma:C2}, and $n^{-1} \sum_{i=1}^n \epsilon_i=O_P\left(1/\sqrt{n} \right)$, since $\epsilon_i$ is sub-Gaussian. Similarly, we have
\begin{equation} \label{eq:57}
    \supnorm{n^{-1} (\Uh-\U)'\U\betab} = O_P\left(\frac{1}{\sqrt{n}}\left(\sqrt{\frac{\log n}{n}}+\frac{n^{1/4}}{\sqrt{p_{\min}}} \right)\right). 
\end{equation}
Combining (\ref{eq:55}), (\ref{eq:56}) and (\ref{eq:57}) completes the proof. 
\end{proof}

\begin{lem} \label{lemma:B5}
Suppose the conditions of Theorem \ref{thm:2} hold. Then, 
\begin{equation*}
    \ssupnorm{\frac{1}{n}\Fh'(\Y-\Uh \betabh_a- \Fh\H'\gammabs)}=O_P\left(\delta_n \right).
\end{equation*}
\end{lem}
\begin{proof}
The proof is similar to Lemma \ref{lemma:B4}. We outline the key steps here. We have 
\begin{eqnarray*}
    & & \ssupnorm{\frac{1}{n}\Fh'(\Y-\Uh \betabh_a- \Fh\H'\gammabs)}
        \lesssim \ssupnorm{\frac{1}{n}\F'\{\F\gammabs-\Fh\H'\gammabs+\U\betabs-\Uh\betabh_a+\epsilonb \}} \\
    & \leq & \ssupnorm{\frac{1}{n}\F'\{\F\gammabs-\F\H\H'\gammabs \}} + \ssupnorm{\frac{1}{n}\F'(\F\H-\Fh)\H'\gammabs} \\
    & & + \ssupnorm{\frac{1}{n} \F'\U\betabs} + \ssupnorm{\frac{1}{n} \F'\Uh\betabh_a} + \ssupnorm{\frac{1}{n}\F'\epsilonb}.
\end{eqnarray*}
Similar to (\ref{eq:51}), we have
\begin{equation*}
    \ssupnorm{\frac{1}{n}\F'\{\F\gammabs-\F\H\H'\gammabs \}}= \ssupnorm{\frac{1}{n} \F'\F(\I_K-\H\H')\gammabs} 
    = O_P\left(\sqrt{\frac{\log K}{n}}\left(\frac{1}{\sqrt{n}}+\frac{1}{\sqrt{p_{\min}}}\right)\right).
\end{equation*}
Similar to (\ref{eq:53}), we have
\begin{equation*}
    \ssupnorm{\frac{1}{n}\F'(\F\H-\Fh)\H'\gammabs} = O_P\left(\sqrt{\frac{\log
          K}{n}}\left(\frac{1}{\sqrt{n}}+\frac{n^{1/4}}{\sqrt{p_{\min}}} \right)\right).  
\end{equation*}
Note that $\ltwonorm{\betabs}^2\lesssim \var(\x'\betabs)\leq \sigma_y^2=O(1)$, and $\lambda_{\max}(\Sigmab_u)=O(1)$. Therefore, $\U_i'\betabs$ is sub-Gaussian with bounded variance. Since $f_{ik}$ is sub-Gaussian, $f_{ik}\U_i'\betabs$ is exponential. Then by Bernstein inequality, we have that $\supnorm{n^{-1}\F'\U\betabs}=O_P\left(\sqrt{(\log K)/n} \right)$. Similarly, $\supnorm{n^{-1}\Fh'\Uh\betabs}=O_P\left(\sqrt{(\log K)/n} \right)$. Finally, $f_{ik}\epsilon_i$ is sub-exponential, then $\supnorm{n^{-1}\F'\epsilonb}=O_P\left(\sqrt{(\log K)/n} \right)$. Combining these results completes the proof.
\end{proof}

\subsection{Additional numerical results}
\label{sec:add-numerical}
Figure \ref{fig:realdata1} shows the heat map of the correlation matrix of the multimodal neuroimaging data analyzed in Section \ref{sec:real-data}. It is seen that some covariates are highly correlated. 

\begin{figure}[t!]
\centering
\includegraphics[width=3.2in, height=3.0in]{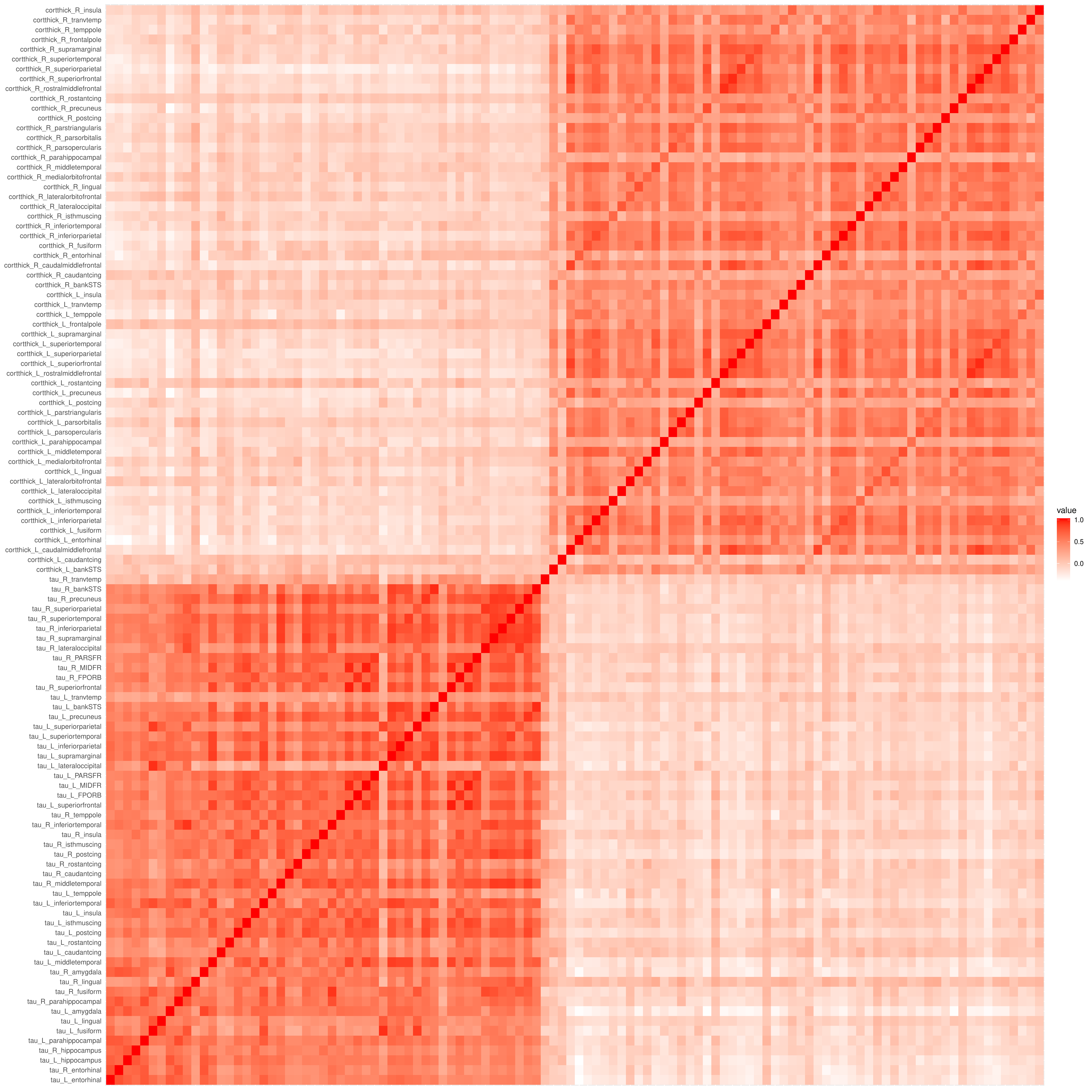}
\caption{The heat map of the correlation matrix of the multimodal neuroimaging data.}
\label{fig:realdata1}
\end{figure}

\bibliographystyle{bibstyle}

\end{document}